\documentclass{article}
\usepackage[english]{babel}
\usepackage{amsmath}
\usepackage{amsthm}
\usepackage{amssymb}
\usepackage{amscd, enumerate}
\usepackage[latin1]{inputenc}
\usepackage{pstricks}
\usepackage{graphicx}
\usepackage[all]{xy}

\newtheorem{teor}{Theorem}[section]
\newtheorem{defi}{Definition}
\newtheorem{lema}[teor]{Lemma}
\newtheorem{prop}[teor]{Proposition}
\newtheorem{cor}[teor]{Corollary}
\newtheorem{rem}[teor]{Remark}

\newtheorem{convention}[teor]{Convention}

\begin{document}

\title{The symmetry, period and Calabi-Yau dimension  of
finite dimensional mesh algebras}
\author{Estefan\'{i}a Andreu Juan\\ eaj1@um.es
\and Manuel Saor\'{i}n\\msaorinc@um.es}

\date{}

\maketitle

\begin{center}

Departamento de Matem\'{a}ticas

Universidad de Murcia

Campus de Espinardo, 30100 Murcia

Spain \end{center}

\vspace{1cm}

\begin{abstract}
Within the class of finite dimensional mesh algebras, we identify
those which are symmetric and those whose stable module category is
Calabi-Yau. We also give, in combinatorial terms, explicit
formulas for the $\Omega$-period of any such algebra, where $\Omega$
is the syzygy functor, and for the  Calabi-Yau Frobenius and the
stable Calabi-Yau dimensions, when they are defined.
\end{abstract}

\vspace{0.4cm}


\vspace{0.4cm}

\noindent Classification Code: 16Gxx ; Representation theory of
rings and algebras.

\section{Introduction}\oddsidemargin=
3cm \evensidemargin=1.5cm \textheight=24cm \voffset=-1.5cm

A $Hom$ finite triangulated $K$-category $\mathcal{T}$, with
suspension functor $\sum : \mathcal{T}\longrightarrow \mathcal{T}$,
is called \emph{Calabi-Yau} (see \cite{Ko}), when there is a natural
number $n$ such that $\sum^n$ is a Serre functor (i.e.
$DHom_{\mathcal{T}}(X,-)$) and $Hom_{\mathcal{T}}(-, \sum^nX)$ are
naturally isomorphic as cohomological functors
$\mathcal{T}^{op}\longrightarrow K-\text{mod}$). In such a case, the
smallest natural number $m$ such that $\sum^m$ is a Serre functor is
called the Calabi-Yau dimension (\emph{CY-dimension} for short) of
$\mathcal{T}$. Calabi-Yau triangulated categories appear in many
fields of Mathematics and Theoretical Physics. In Representation
Theory of  algebras, the concept plays an important role in the
study of cluster algebras and cluster categories (see   \cite{K3}).

When $\Lambda$ is a self-injective finite dimensional (associative
unital) algebra, its stable module category
$\Lambda-\underline{\text{mod}}$ is a triangulated category and the
Calabi-Yau condition on this category naturally  appears and has
been deeply studied (see, e.g., \cite{ESk}, \cite{BS},  \cite{ES},
\cite{D2}, \cite{I}, \cite{IV},...). The concept is related with
that of Frobenius Calabi-Yau algebra, as defined by Eu and Schedler
(\cite{ES}). The algebra $\Lambda$ is called \emph{Calabi-Yau
Frobenius} when $\Omega_{\Lambda^e}^{r+1}(\Lambda )$ is isomorphic
to $D(\Lambda )=\text{Hom}_K(\Lambda ,K)$ as $\Lambda$-bimodule, for
some integer $r\geq 0$. If the algebra $\Lambda$ is  Calabi-Yau
Frobenius, then $\Lambda-\underline{\text{mod}}$ is Calabi-Yau and the
Calabi-Yau dimension of this category is less or equal than the
smallest $r$ such that $\Omega_{\Lambda^e}^{-r-1}(\Lambda )$ is
isomorphic to $D(\Lambda )$, a number which is called here the
\emph{Calabi-Yau Frobenius dimension} of $\Lambda$. In general, it
is not known whether these two numbers are equal.

A basic finite dimensional algebra $\Lambda$ is self-injective
precisely when there is an isomorphism of $\Lambda$-bimodules
between $D(\Lambda )$ and the twisted bimodule ${}_1\Lambda_\eta$,
for some automorphism $\eta$ of $\Lambda$. This automorphism is
uniquely determined up to inner automorphism and is called the
\emph{Nakayama automorphism} of $\Lambda$ (see section 2 for more
details). Then the problem of deciding when $\Lambda$ is Calabi-Yau
Frobenius is part of a more general problem, namely, to determine
under which conditions $\Omega_{\Lambda^e}^r(\Lambda )$ is
isomorphic to a twisted bimodule ${}_1\Lambda_\varphi$, for some
automorphism $\varphi$ of $\Lambda$, which is then determined up to
inner automorphism. By a result of Green-Snashall-Solberg
(\cite{GSS}), this condition on a finite dimensional algebra forces
it to be self-injective. When $\varphi$ is the identity (or an inner
automorphism), the algebra $\Lambda$ is called \emph{periodic} and
the problem of determining the self-injective algebras which are
periodic is widely open. However, there is a lot of work in the
literature were several classes of periodic algebras have been
identified (see, e.g., \cite{BBK}, \cite{ESk2}, \cite{D}). Even when
an algebra $\Lambda$ is known to be periodic, it is usually hard to
calculate explicitly its \emph{period}, that is, the smallest of the
integers $r>0$ such that $\Omega_{\Lambda^e}^r(\Lambda )$ is
isomorphic to $\Lambda$ as a bimodule.

Another interesting problem in the context of finite dimensional
self-injective algebras is that of determining when such an algebra
is weakly symmetric or symmetric. An algebra is \emph{symmetric}
when $D(\Lambda )$ is isomorphic to $\Lambda$ as a
$\Lambda$-bimodule. This is equivalent to saying that the
\emph{Nakayama functor} $DHom_\Lambda (-,\Lambda )\cong D(\Lambda
)\otimes_\Lambda -:\Lambda-\text{Mod}\longrightarrow\Lambda -\text{Mod}$ is
naturally isomorphic to the identity functor. The algebra is
\text{weakly symmetric} when this functor just preserves the
iso-classes of simple modules.

In this paper we tackle the problems mentioned above for a special
class of finite dimensional self-injective algebras, which has
deserved a lot of attention in recent times. Following \cite{ESk2},
if $\Delta$ is one of the Dynkin quivers $\mathbb{A}_r$,
$\mathbb{D}_r$ of $\mathbb{E}_n$ ($n=6,7,8$),  an  \emph{$m$-fold
mesh algebra of type $\Delta$}  is the mesh algebra of the stable
translation quiver $\mathbb{Z}\Delta /G$, where $G$ is a weakly
admissible group of automorphisms of $\mathbb{Z}\Delta$ (see
subsections \ref{subsect.stable translation quivers} and
\ref{subsect.definition mesh algebra} for definitions and details).
By a result of Dugas (\cite{D2}[Theorem 3.1]), the $m$-fold mesh
algebras are precisely the mesh algebras of translation quivers
which are finite dimensional. This class of algebras properly
contains the stable Auslander algebras of all standard
representation-finite self-injective algebras (see \cite{D2}) and,
also,  the Auslander algebras of several hypersurface
singularities (see \cite{ESk2}[Section 8]). In fact, due to the classification by Amiot \cite{A} of the standard algebraic triangulated categories of finite type, we know that the $m$-fold mesh algebras are precisely the Auslander algebras of these triangulated categories. Then any such triangulated category can be identified with the category $\Lambda -\text{proj}$ of finitely generated projective $\Lambda$-modules, where $\Lambda$ is an $m$-fold mesh algebra, when taking as  suspension functor in $\Lambda-\text{proj}$ the functor $\Omega_{\Lambda^e}^{-3}(\Lambda )\otimes -$. Similarly the Serre functor on this latter category is identified with the Nakayama functor $D(\Lambda )\otimes_\Lambda -$. It is well-known (see \cite{ESk2} or \cite{D2}) that $\Omega_{\Lambda^e}^{3}(\Lambda )\cong {}_\mu\Lambda_1$, for an automorphism $\mu$ which has finite order as an outer automorphism, so that in particular, all the algebras in the class are periodic, a fact proved in
\cite{BBK}[Section 6] even before the class of algebras was introduced. Therefore the explicit identification of the automorphisms $\eta$ and $\mu$ for all  $m$-fold mesh algebra, which we obtain in this paper (see Theorem \ref{teor.G-invariant Nakayama automorphism} and Corollary \ref{cor. mu for Lambda}) translates to an explicit description of the Serre functor and a quasi-inverse of the suspension functor for any algebraic triangulated category of finite type. From previous papers it seems that only the action of these functors on objects was known.

A particular case of standard algebraic triangulated category of finite type is the stable category $A-\underline{mod}$ of a representation-finite selfinjective algebra $A$. Within the class of $m$-fold mesh algebras $\Lambda$ which are the Auslander algebra of such a stable category, those which are Calabi-Yau-Frobenius and those for which $\Lambda-\underline{mod}$ is Calabi-Yau  have been completely determined. The task was initiated in
 in \cite{BS} and \cite{ESk}, but the main part of the work was done in  \cite{D2} and \cite{IV}.
In the first of these two papers,  Dugas  identified such an algebra
by the type $(\Delta ,f,t))$ of the original representation-finite selfinjective algebra $A$, as defined by Asashiba (\cite{Asa})
inspired by the work of Riedtmann (\cite{Ri}). He completed the task
when $t$ is $1$ or $3$, and also in many cases with $t=2$. In fact the author even described a relationship between the Calabi-Yau condition of $A-\underline{mod}$ and of $\Lambda-\underline{mod}$ (see \cite{D2}[Proposition 2.1]). The
remaining cases
 for $t=2$ have been recently settled by Ivanov-Volkow (\cite{IV}).
On the question of periodicity,  the explicit calculation of the period of an $m$-fold mesh algebra
has been also done in a few cases. From the papers \cite{RS},
\cite{ES3} and \cite{BES} we know that the period is 6 for all
preprojective algebras of generalized Dynkin type. In \cite{D} and \cite{D2} the
period is calculated when $\Lambda$ is the stable Auslander algebra
of a standard representation-finite self-injective algebra of  type
$(\Delta ,f ,t)$, when $t=1$ or when this type is  equal to $(\mathbb{D}_4,f,3)$,
$(\mathbb{D}_n,f,2)$, with $n>4$ and $f>1$ odd, or
$(\mathbb{E}_6,f,2)$. In fact, the author uses these results to calculate the period of the original finite type selfinjective algebra A.

We now explain the main results of our paper. Let $B=B(\Delta )$ be
the mesh algebra of the translation quiver $\mathbb{Z}\Delta$, where
$\Delta$ is one of the  Dynkin quivers mentioned above, and let $G$
be a weakly admissible automorphism of $\mathbb{Z}\Delta$ which is
viewed also as an automorphism of $B$. The algebra $B$ is graded
Quasi-Frobenius
(see definition \ref{defi.pseudo-Frobenius graded
algebras}), which roughly means that $B$ and its category of graded
modules behave as  a self-injective finite dimensional algebra and
its category of modules. The crucial result of our paper is Theorem
\ref{teor.G-invariant Nakayama automorphism}, which explicitly
defines a graded Nakayama automorphism of $B$ which commutes with
the elements of $G$, for any choice of $(\Delta ,G)$. Since each
$m$-fold mesh algebra $\Lambda$ is isomorphic to the orbit category
$B/G$,  the consequence is that one derives an explicit definition
of a graded Nakayama automorphism, for each $m$-fold mesh algebra.
We then use the key Lemma \ref{lem.criterion of inner automorphism},
which determines when a $G$-invariant graded automorphism of $B$
induces an inner automorphism of $\Lambda =B/G$.
Using the extended type $(\Delta ,m,t)$ (see definition
\ref{defi.extended type of a mesh algebra}) to identify an $m$-fold
mesh algebra, we get, expressed in terms of this type, the main
results of the paper, all referred to $m$-fold mesh algebras:

\begin{enumerate}
\item An identification of all weakly symmetric  and symmetric algebras in the
class (Theorem \ref{teor.weakly symmetric m-fold algebras});
\item An explicit formula for the period of any algebra in the
class (Proposition \ref{prop.periods of Nakayama algebras}, when
$\Delta =\mathbb{A}_2$, and Theorem \ref{teor.period of Lambda} for
all the other cases).
\item An identification of the precise relation between the stable Calabi-Yau dimension and the
Calabi-Yau Frobenius dimension of an $m$-fold algebra, showing that
both dimensions may differ when $\Delta=\mathbb{A}_2$, but are
always equal when $\Delta\neq\mathbb{A}_r$, for $r=1,2$
(Propositions \ref{prop.CY-dim and CYF-dim of 2-nilpotent algebras}
and \ref{prop.CY-dimension-versus-Eu-Schedler})
\item A criterion for an $m$-fold mesh algebra to be stably
Calabi-Yau, together with the identification in such case of the
stable Calabi-Yau dimension (Proposition \ref{prop.2-nilpotent
selfinjective algebras}, for the case $\Delta =\mathbb{A}_2$,
Corollary \ref{cor.CY-criterion in characteristic 2}, for
characteristic $2$, and Theorem \ref{teor.CY-criterion and
CY-dimension} for all other cases).
\end{enumerate}

We now describe the organization of the paper. In section 2 we recall from \cite{AS} the concept of  pseudo-Frobenius graded  algebra with enough idempotents and the corresponding results which we need in the rest of the paper.

%we recall
%some facts concerning the pseudo/Quasi-Frobenius condition on mesh
%algebras $B(\Delta)$, with $\Delta$ a Dynkin diagram,
%and its corresponding orbit
In section 3, we recall the definition of the mesh algebra of a
stable translation quiver and give a list of essentially known
properties (Proposition \ref{prop:Coxeter-Nakayama}) for the case of
the mesh algebra $B=B(\Delta )$ of $\mathbb{Z}\Delta$, when $\Delta$
is a Dynkin quiver. We then recall the definition of an $m$-fold mesh
algebra and their known properties, and introduce the notion of
extended type of such an algebra, which plays a crucial role in the
rest of the paper. With the idea of simplifying some calculations,
we end the section by performing a change of relations which,
roughly speaking, transforms sums of paths of length 2 into
differences.

In section 4 we give the explicit definition of the $G$-invariant
graded Nakayama automorphism of $B$ and give the crucial Lemma
\ref{lem.criterion of inner automorphism} mentioned above. We then
give the list of all weakly symmetric and symmetric $m$-fold mesh
algebras.

In the final section 5 we first calculate explicitly  the initial
part of a '$G$-invariant' minimal projective resolution of $B$ as a
graded $B$-bimodule, We prove in particular that $\Omega_{B^e}^3(B)$
is always isomorphic to ${}_\mu B_1$, for a graded automorphism
$\mu$ of $B$ which is in the centralizer of $G$ and which is
explicitly calculated. Then the induced automorphism $\bar{\mu}$ of
$\Lambda =B/G$ has the property that $\Omega_{\Lambda^e}^3(\Lambda
)\cong {}_{\bar{\mu}}\Lambda_1$ and this is fundamental in the rest
of the paper. We then calculate the period of all $m$-fold mesh
algebras and
 find the precise relation between the Calabi-Yau Frobenius
condition of $\Lambda$, in the sense of \cite{ES}, and the condition
that $\Lambda -\underline{\text{mod}}$ be a Calabi-Yau category. We end the
paper by giving necessary and sufficient conditions for a mesh
algebra to be stably Calabi-Yau and, when this is the case, we
calculate explicitly the Calabi-Yau dimension of $\Lambda
-\underline{\text{mod}}$.

\section{Pseudo-Frobenius graded algebras with enough idempotents}

This section is devoted to compiling from \cite{AS}  information concerning the class of
 pseudo-Frobenius graded algebras with enough idempotents,  which will very useful in the subsequent sections. The reader is referred
to that paper for proofs and further details.

\vspace{0.2cm}

Let $A$ be an associative algebra over a field $K$. Such an
algebra is said to be an \emph{algebra with enough idempotents}
when there is a  family $(e_i)_{i\in I}$ of nonzero orthogonal
idempotents, called \emph{distinguished family}, such that
$\oplus_{i\in I}e_iA=A=\oplus_{i\in I}Ae_i$. If in addition, $H$ is
a fixed abelian group with additive notation, an \emph{$H$-graded algebra with enough
idempotents} will be an algebra with enough idempotents $A$,
together with an $H$-grading $A=\oplus_{h\in H}A_h$, such that one
can choose a distinguished family of orthogonal idempotents which
are homogeneous of degree $0$. Frequently, we will interpret such an algebra as a (small) graded $K$-category with $I$ as set of objects, where $e_iAe_j$ is the set of morphisms from $i$ to $j$ and where the composition of morphisms is just the anti-multiplication:  $b\circ a=ab$. Then the concept of functor between such algebras makes sense and will be used sometimes.

Throughout this section by $A$ we will mean an $H$-graded algebra
with enough idempotents on which  a distinguished family of
orthogonal idempotents is fixed. All considered left (resp. right) $A$-modules are supposed to be
unital, i.e.,  $AM=M$ (resp. $MA=M$) for any left (resp. right) $A$-module $M$.
We will also denote by $A-Gr$ (resp.  $Gr-A$) the
category of ($H$-)graded unital left (resp. right) modules,
where the morphisms are the graded homomorphisms of degree $0$.

%All considered (left or right) $A$-modules are supposed to be
%unital. For a left (resp. right) $A$-module $M$, that means that
%$AM=M$ (resp. $MA=M$) or, equivalently, that $M=\oplus_{i\in I}e_iM$
%(resp. $M=\oplus_{i\in I}Me_i$). We denote by $A-\text{Mod}$ and
%$\text{Mod}-A$ the categories of left and right $A$-modules,
%respectively.

The enveloping algebra of $A$ is the algebra $A^e=A\otimes A^{op}$,
where if $a,b\in A$ we will denote by $a\otimes b^o$ the
corresponding element of $A^e$. This is also an $H$-graded algebra with enough
idempotents where $(A\otimes
A^{op})_h=\oplus_{s+t=h}A_s\otimes A^{op}_t$. The distinguished family of orthogonal idempotents with
which we will work is the family $(e_i\otimes e_j^o)_{(i,j)\in
I\times I}$. A left graded $A^e$-module $M$ will be identified with a graded
$A$-bimodule by putting $axb=(a\otimes b^o)x$, for all $x\in M$ and
$a,b\in A$. Similarly, a right graded $A^e$-module is identified with a graded
$A$-bimodule by putting $axb=x(b\otimes a^o)$, for all $x\in M$ and
$a,b\in A$. In this way, we identify the three categories
$A^e-Gr$, $Gr-A^e$ and $A-Gr-A$, where the
last one is the category of graded unitary $A$-bimodules, which we will
simply name `graded bimodules'.

Recall that if $M=\oplus_{h\in H}M_h$ is a graded $A$-module and $k\in H$ is any
element, then we get a graded module $M[k]$ having the same
underlying ungraded $A$-module as $M$, but where $M[k]_h=M_{k+h}$
for each $h\in H$. If $M$ and $N$ are graded left $A$-modules, then
$\text{HOM}_{A}(M,N):=\oplus_{h\in H}\text{Hom}_{A-Gr}(M,N[h])$
 has a structure of graded $K$-vector space, where the homogeneous
 component of degree $h$ is
 $\text{HOM}_{A}(M,N)_h:=\text{Hom}_{A-Gr}(M,N[h])$, i.e.,
 $\text{HOM}_{A}(M,N)_h$ consists of the graded homomorphisms $M\longrightarrow N$ of
 degree $h$.

%When $A=\oplus_{h\in H}A_h$ and $B=\oplus_{h\in H}B_h$ are graded
%algebras with enough idempotents, the tensor algebra $A\otimes B$
%inherits a structure of graded $H$-algebra, where $(A\otimes
%B)_h=\oplus_{s+t=h}A_s\otimes B_t$. In particular, this applies to
%the enveloping algebra $A^e$ and, as in the ungraded case, we will
%identify the categories $A^e-Gr$ (resp. $Gr-A^e$) and $A-Gr-A$ of
%graded left (resp. right) $A$-modules and graded $A$-bimodules. We
%will denote by $A-lfgr-A$ the full subcategory of $A-Gr-A$
%consisting of locally finite dimensional graded $A$-bimodules.

The graded algebras we are interested in have some extra properties.
For reader's convenience we recall the following definitions.

\begin{defi} \label{defi.locally f.d. graded algebra}
Let $A=\oplus_{h\in H}A_h$ be a graded algebra with enough
idempotents. It will be called \emph{locally finite dimensional} when
$e_iA_he_j$ is finite dimensional, for all $(i,j,h)\in I\times
I\times H$. Such a graded algebra $A$ will be called \emph{graded
locally bounded} when the following two conditions hold:

\begin{enumerate}
\item For each $(i,h)\in I\times H$, the set $I^{(i,h)}=\{j\in I:$ $e_iA_he_j\neq
0\}$ is finite.
\item For each $(i,h)\in I\times H$, the set $I_{(i,h)}=\{j\in I:$ $e_jA_he_i\neq
0\}$ is finite.
\end{enumerate}
\end{defi}

Observe that these definitions do not depend on the distinguished family $(e_i)$ considered.

\begin{defi} \label{def.basic split graded algebra}
A locally finite dimensional graded algebra with enough idempotents
$A=\oplus_{h\in H}A_h$ will be called \emph{weakly basic} when it
has a distinguished family $(e_i)_{i\in I}$ of orthogonal
homogeneous idempotents of degree $0$ such that:

\begin{enumerate}

\item   $e_iA_0e_i$ is a local algebra, for each $i\in I$. \item
$e_iAe_j$ is contained in the graded Jacobson radical
$J^{gr}(A)$, for all $i,j\in I$, $i\neq j$.
\end{enumerate}
It will be called \emph{basic} when, in addition,
$e_iA_he_i\subseteq J^{gr}(A)$, for all $i\in I$ and $h\in
H\setminus\{0\}$.

We will use also the term `(weakly) basic' to denote any
distinguished family $(e_i)_{i\in I}$ of ortho\-go\-nal idempotents
satisfying the above conditions.

A weakly basic graded algebra with enough idempotents will be called
\emph{split} when  $e_iA_0e_i/e_iJ(A_0)e_i$ $\cong K$, for each $i\in
I$.

\end{defi}

We can consider $K$ as an $H$-graded algebra by putting $K_h=0$, for
all $h\neq 0$.  Given a graded $K$-vector space $V=\oplus_{h\in H}V_h$, where $V_h$
is finite dimensional $\forall h\in H$,
we will denote by $D(V)$ the subspace of the vector space
$\text{Hom}_K(V,K)$ consisting of the linear forms
$f:V\longrightarrow K$ such that $f(V_h)=0$, for all but finitely
many $h\in H$. Then, $D(V)$ can be identified with the graded
$K$-vector space $\oplus_{h\in H}\text{Hom}_K(V_h,K)$, where
$D(V)_h=\text{Hom}_K(V_{-h},K)$ for all $h\in H$.

\begin{defi}
Let $V=\oplus_{h\in H}V_h$ and $W=\oplus_{h\in H}W_h$  be graded
$K$-vector spaces, where the homogeneous components are finite
dimensional, and let $d\in H$ be any element. A bilinear form
$(-,-):V\times W\longrightarrow K$ is said to be \emph{of degree
$d$} if $(V_h,W_k)\neq 0$ implies that $h+k=d$. Such a form will

be called \emph{nondegenerate} when the induced maps $W\longrightarrow
D(V)$ ($w\rightsquigarrow (-,w)$)) and $V\longrightarrow D(W)$
($v\rightsquigarrow (v,-)$)) are bijective.
\end{defi}

Note that, in the above situation, if $(-,-):V\times
W\longrightarrow K$ is a nondegenerate bilinear form of degree $d$,
then the bijective map $W\longrightarrow D(V)$ ( resp
$V\longrightarrow D(W)$) given above gives an isomorphism of graded
$K$-vector spaces
 $W[d]\stackrel{\cong}{\longrightarrow}D(V)$ (resp. $V[d]\stackrel{\cong}{\longrightarrow}D(W)$).

 The following concept is fundamental for us.

 \begin{defi} \label{defi.graded Nakayama form}
 Let $A=\oplus_{h\in H}A_h$ be a weakly basic graded algebra with
 enough idempotents. A bilinear form $(-,-):A\times A\longrightarrow
 K$ is said to be a \emph{graded Nakayama form} when the following
 assertions hold:

 \begin{enumerate}
 \item $(ab,c)=(a,bc)$, for all $a,b,c\in A$
 \item For each $i\in I$ there is a unique $\nu (i)\in I$ such that
 $(e_iA,Ae_{\nu (i)})\neq 0$ and the assignment $i\rightsquigarrow\nu
 (i)$ defines a bijection $\nu :I\longrightarrow I$. \item There is
 a map $\mathbf{h}:I\longrightarrow H$ such that the
 induced map $(-,-):e_iAe_j\times e_jAe_{\nu (i)}\longrightarrow K$ is a
 nondegenerate graded bilinear form of degree $h_i=\mathbf{h}(i)$,
 for all $i,j\in I$.
 \end{enumerate}
The bijection $\nu$ is called the \emph{Nakayama permutation} and
$\mathbf{h}$ will be called the \emph{degree map}. When $\mathbf{h}$
is a constant map and $\mathbf{h}(i)=h$, we
 will say that $(-,-):A\times A\longrightarrow K$ is a \emph{graded
 Nakayama form of degree $h$}.
 \end{defi}

 Recall that  a Quillen
 exact category $\mathcal{E}$ (e.g. an abelian category) is said to be a
 \emph{Frobenius category} when it has enough projectives and enough
 injectives and the projective and the injective objects are the
 same in $\mathcal{E}$.

 \begin{defi}\label{defi.pseudo-Frobenius graded
algebras} A weakly basic locally finite dimensional
 graded algebra $A$ with enough idempotents is called \emph{graded pseudo-Frobenius} (resp. \emph{graded Quasi-Frobenius})  if it admits a graded Nakayama form $(-,-):A\times A\longrightarrow K$ (resp. both
 $A-Gr$ and $Gr-A$ are Frobenius categories).
\end{defi}

 A graded Quasi-Frobenius algebra A is always graded pseudo-Frobenius and the converse is true whenever $A$ is graded locally Noetherian i.e., whenever $Ae_i$  and $e_iA$ satisfies the ACC on graded
 submodules, for each $i\in I$. Note then that, for a finite dimensional algebra $A$, viewed as a graded algebra concentrated in zero degree,   the notions of self-injective,
 graded pseudo-Frobenius and graded Quasi-Frobenius coincide.

 When $A$ is a graded locally bounded pseudo-Frobenius algebra, its graded Nakayama form $(-,-)$ induces
 an automorphism of (ungraded) algebras $\eta\in Aut(A)$ such that $D(A)\cong {}_{1}A_{\eta}$. This automorphism is unique , up to inner automorphism, and called \emph{Nakayama automorphism}. It is given by the rule $(a,-)=(-,\eta(a))$, for every $a\in A$, and turns out to be an automorphism of
 graded algebras whenever the associated map $\mathbf{h}: I\longrightarrow H$ takes constant value $h$. Indeed, in this latter case
 we get an isomorphism of graded algebras $D(A)\cong {}_{1}A_{\eta}[h]$.

The following result gives a handy criterion, in the locally Noetherian
 case, for $A$ to be graded Quasi-Frobenius.

 \begin{cor} \label{cor.Frobenius=simple socle}
 Let $A=\oplus_{h\in H}A_h$ be a weakly basic locally Noetherian
  graded algebra with enough idempotents. The following
 assertions are equivalent:

 \begin{enumerate}

\item The following two conditions hold:

\begin{enumerate}
 \item For each $i\in I$, $Ae_i$ and $e_iA$ have a simple essential
socle in $A-Gr$ and $Gr-A$, respectively \item There are bijective
maps $\nu ,\nu' :I\longrightarrow I$ such that
$\text{Soc}_{gr}(e_iA)\cong \frac{e_{\nu (i)}A}{e_{\nu
 (i)}J^{gr}(A)}[h_i]$ and  $\text{Soc}_{gr}(Ae_i)\cong
 \frac{Ae_{\nu '
 (i)}}{[J^{gr}(A)e_{\nu '
 (i)}}[h'_i]$, for certain $h_i,h'_i\in H$

\end{enumerate}

\item $A$ is graded
Quasi-Frobenius
 \end{enumerate}
 \end{cor}

The following result shows that if $A$ is a
split graded pseudo-Frobenius algebra, then all possible graded Nakayama
forms for $A$ come in similar way. Recall from \cite{AS} if $A$ is such an algebra, then $\text{Soc}_{gr}({}_A)=\text{Soc}_{gr}(A_A)$. Recall also that if $V=\oplus_{h\in
H}V_h$ is a graded vector space, then its \emph{support}, denoted
$\text{Supp}(V)$,  is the set of $h\in H$ such that $V_h\neq 0$.

\begin{prop} \label{prop.graded Nak-form via basis}
Let $A$ be a split pseudo-Frobenius graded algebra and $(e_i)_{i\in
I}$ a weakly basic distinguished family of orthogonal idempotents.
The following assertions hold:

\begin{enumerate}
\item All graded Nakayama forms for $A$ have the same Nakayama
permutation. It assigns to each $i\in I$ the  unique $\nu (i)\in I$
such that $e_i\text{Soc}_{gr}(A)e_{\nu (i)}\neq 0$.
\item If $h_i\in \text{Supp}(e_i\text{Soc}_{gr}(A))$, then
$\text{dim}(e_i\text{Soc}_{gr}(A))_{h_i})=1$ \item For a bilinear
form $(-,-):A\times A\longrightarrow K$,  the following statements
are equivalent:

\begin{enumerate}
\item $(-,-)$ is a graded Nakayama form for $A$ \item There exists
an element $\mathbf{h}=(h_i)\in\prod_{i\in
I}\text{Supp}(e_i\text{Soc}_{gr}(A))$ and a basis $\mathcal{B}_i$ of
$e_iA_{h_i}e_{\nu (i)}$, for each $i\in I$, such that:

\begin{enumerate}
\item $\mathcal{B}_i$ contains a (unique) element $w_i$ of
$e_i\text{Soc}_{gr}(A)_{h_i}$ \item If $a,b\in\bigcup_{i,j}e_iAe_j$
are homogeneous elements, then $(e_iA_h,A_ke_j)=0$ unless $j=\nu
(i)$ and $h+k=h_i$ \item If $(a,b)\in e_iA_h\times A_{h_i-h}e_{\nu
(i)}$, then $(a,b)$ is the coefficient of $w_i$  in the expression
of $ab$ as a linear combination of the elements of $\mathcal{B}_i$.
\end{enumerate}
\end{enumerate}
\end{enumerate}
\end{prop}

\begin{defi} \label{defi.graded Nakayama form associated to basis}
Let $A=\oplus_{h\in H}A_h$ be a split pseudo-Frobenius graded
algebra, with $(e_i)_{i\in I}$ as weakly basic distinguished family
of idempotents and $\nu :I\longrightarrow I$ as Nakayama
permutation. Given a pair $(\mathcal{B},\mathbf{h})$ consisting of
an element $\mathbf{h}=(h_i)_{i\in I}$ of $\prod_{i\in
I}\text{Supp}(e_i\text{Soc}_{gr}(A))$ and a family
$\mathcal{B}=(\mathcal{B}_i)_{i\in I}$, where $\mathcal{B}_i$ is a
basis of $e_iA_{h_i}e_{\nu (i)}$ containing an element of
$e_iSoc_{gr}(A)$,  for each $i\in I$, we call \emph{graded Nakayama
form associated to $(\mathcal{B},\mathbf{h})$} to the bilinear form
$(-,-):A\times A\longrightarrow K$ determined by the conditions b.ii
and b.iii of last proposition. When $\mathbf{h}$ is constant, i.e.
there is $h\in H$ such that $h_i=h$ for all $i\in I$,  we will call
$(-,-)$ the graded Nakayama form of $A$ of degree $h$ associated to
$\mathcal{B}$.
\end{defi}

We now assume that $G$ is a group
acting on $A$ as a group of graded automorphisms (of degree $0$)
which permutes the $e_i$. That is, if $\text{Aut}^{gr}(A)$ denotes
the group of graded automorphisms of degree $0$ which permute the
$e_i$, then we have a group homomorphism $\varphi
:G\longrightarrow\text{Aut}^{gr}(A)$. We will write $a^g=\varphi
(g)(a)$, for each $a\in A$ and $g\in G$.

The following definition will be needed for our purposes.

\begin{defi}
Let $A=\oplus_{h\in H}A_h$ be a graded pseudo-Frobenius algebra and
$G$ be a group acting on $A$ as graded automorphisms. A graded
Nakayama form $(-,-):A\times A\longrightarrow K$ will be called
\emph{$G$-invariant} when $(a^g,b^g)=(a,b)$, for all $a,b\in A$ and
all $g\in G$.
\end{defi}

We say that a group $G$ as above \emph{acts
freely on objects} when $g(i)\neq i$, for all $i\in I$ and $g\in
G\setminus\{1\}$. In such case one can consider the \emph{orbit category}
$A/G$. The objects of this category are the $G$-orbits $[i]$ of
indices $i\in I$ and the morphisms from $[i]$ to $[j]$ are formal
sums $\sum_{g\in G}[a_g]$, where $[a_g]$ is the $G$-orbit of an
element $a_g\in e_iAe_{g(j)}$. This definition does not depend on
$i,j$, but just on the orbits $[i],[j]$. The anticomposition of
morphisms extends by $K$-linearity the following rule. If
$a,b\in\bigcup_{i,j\in I}e_iAe_j$ and $[a]$, $[b]$ denote the
$G$-orbits of $a$ and $b$, then $[a]\cdot [b]=0$, in case
$[t(a)]\neq [i(b)]$, where $t(a)$ and $i(b)$ denote the terminus vertex of $a$ and the initial vertex of $b$, and $[a]\cdot [b]=[ab^g]$, in case
$[t(a)]=[i(b)]$,  where $g$ is the unique element of $G$ such that
$g(i(b))=t(a)$.

We now recall a result from \cite{AS} concerning the preservation of
the pseudo-Frobenius condition via the canonical projection $\pi: A\longrightarrow A/G$.
with takes $a\rightsquigarrow [a]$.

\begin{prop} \label{prop.gr-Frobenius via pushdown functor}
Let $A=\oplus_{h\in H}A_h$ be a (split basic)  locally
bounded graded pseudo-Frobenius algebra, with $(e_i)_{i\in I}$ as weakly basic distinguished family of
orthogonal homogeneous idempotents,  and let $G$ be a group which
acts on $A$ as graded automorphisms which permute the $e_i$ and
which acts freely on objects. If there exists a $G$-invariant graded Nakayama form
$(-,-):A\times A\longrightarrow K$. Then $\Lambda =A/G$ is a (split basic)  locally bounded
graded pseudo-Frobenius algebra whose
graded Nakayama form is induced from $(-,-)$.
\end{prop}

Under the assumptions of last proposition, it is known that
the functor $\pi$ is a \emph{covering
functor}, that is, it is surjective on objects (i.e., vertices) and induces bijective
maps $\oplus_{i'\in \pi^{-1}(j)}e_iA_he_{i'}\longrightarrow e_{\pi(i)}\Lambda_he_j$
and $\oplus_{i'\in \pi^{-1}(j)}e_{i'}A_he_{i}\longrightarrow
e_{j}\Lambda_he_{\pi(i)}$, for each
$(i,j,h)\in I\times J\times H$. Furthermore, the pushdown functor $\pi_{\lambda}: A-Gr\rightarrow \Lambda-Gr$
which takes $Ae_i\rightsquigarrow \Lambda e_{[i]}$ is exact (see, e.g., \cite{CM} and \cite{Asa2} for further
details).

The following result ensures that, in the split case, $G$-invariant graded Nakayama forms always exist, a fact which allows to apply last proposition.

\begin{cor} \label{cor.G-invariant basis and Nakayama form}
Let $A=\oplus_{h\in H}A_h$ be a split  graded pseudo-Frobenius
algebra with Nakayama permutation $\nu$ and let $G$ be a group of graded automorphisms of $A$ which
permute the $e_i$ and acts freely on objects. The equality $\nu (g(i))=g(\nu (i))$ holds, for all $g\in G$ and $i\in I$.  Moreover, there exist an element
$\mathbf{h}=(h_i)_{i\in I}\in\prod_{i\in I}Supp(e_iSoc_{gr}(A))$ and a
basis $\mathcal{B}_i$ of $e_iA_{h_i}e_{\nu (i)}$, for each $i\in I$, satisfying
the following properties:

\begin{enumerate}
\item $h_i=h_{g(i)}$, for all $i\in I$ \item
$g(\mathcal{B}_i)=\mathcal{B}_{g(i)}$ and $\mathcal{B}_i$ contains an
element of $e_iSoc_{gr}(A)$, for all $i\in I$
\end{enumerate}
In such case, letting $\mathcal{B}=\bigcup_{i\in I}\mathcal{B}_i$, the graded Nakayama form  associated to the pair $(\mathcal{B},\mathbf{h})$
 (see definition \ref{defi.graded Nakayama form associated to basis})
is $G$-invariant.
\end{cor}

\begin{rem}\label{rem. G-invarinat basis} The basis of the previous corollary is constructed as follows.
We fix a subset $I_0\subseteq I$ representing the $G$-orbits of objects and, for each $i\in I_0$,
we fix an $h_i\in
Supp(e_iSoc_{gr}(A))$ and a basis $\mathcal{B}_i$ of $e_iA_{h_i}e_{\nu (i)}$
containing an element $w_i\in e_iSoc_{gr}(A)$. Then, for each $j\in I$, we define
$\mathcal{B}_j=g(\mathcal{B}_i)$ where $i\in I_0$ and $g\in G$ are the unique elements such that $j=g(i)$.

When $A$ is split locally bounded pseudo-Frobenius,  we have that $\eta\circ g= g\circ \eta$, for all $g\in G$,
and hence, the Nakayama automorphism $\bar{\eta}$ of $\Lambda=B/G$ is induced from $\eta$, i.e., $\bar{\eta}([a])= [\eta(a)]$,
for each $a\in \bigcup_{i,j} e_iAe_j$.

\end{rem}

\section{The mesh algebra of a Dynkin quiver}

\subsection{Stable translation quivers}
\label{subsect.stable translation quivers}

Recall that a \emph{quiver} or \emph{oriented graph} is a quadruple
$Q=(Q_0,Q_1,i,t)$, where $Q_0$ and $Q_1$ are sets, whose elements
are called \emph{vertices} and \emph{arrows} respectively, and
$i,t:Q_1\longrightarrow Q_0$ are maps. If $a\in Q_1$ then $i(a)$ and
$t(a)$ are called the \emph{origin} (or $\emph{initial vertex}$) and
the \emph{terminus} of $a$.

Given a quiver $Q$, a \text{path} in $Q$ is a concatenation of
arrows $p=a_1a_2...a_r$ such that $t(a_k)=i(a_{k+1})$, for all
$k=1,...,r$. In such case, we put $i(p)=i(a_1)$ and $t(p)=t(a_r)$
and call them the origin and terminus of $p$. The number $r$ is the
\emph{length} of $p$ and we view the vertices of $Q$ as paths of
length $0$. The \emph{path algebra} of $Q$, denoted by $KQ$,  is the
$K$-vector space with basis the set of paths, where the
multiplication extends by $K$-linearity the multiplication of paths.
This multiplication is defined as $pq=0$, when $t(p)\neq i(q)$, and
$pq$ is the obvious concatenation path, when $t(p)=i(q)$. The
algebra $KQ$ is an algebra with enough idempotents, where $Q_0$ is a
a distinguished family of orthogonal idempotents. If $i\in Q_0$ is a
vertex, we will write it as $e_i$ when we view it as an element of
$KQ$.

A \emph{stable translation quiver} is  a pair $(\Gamma,
\tau)$,  where $\Gamma$ is a \emph{locally finite quiver} (i.e. given any
vertex,  there is a finite number of arrows having it as origin or
terminus) and $\tau: \Gamma_0\rightarrow \Gamma_0$ is a bijective
map such that for any $x,y\in \Gamma_0$, the number of arrows from
$x$ to $y$ is equal to the number of arrows from $\tau(y)$ to $x$.
The map $\tau$ will be called the \emph{Auslander-Reiten
translation}. Throughout the rest of the work, whenever we have a
stable translation quiver, we will also fix a bijection $\sigma:
\Gamma_1(x,y)\rightarrow \Gamma_1(\tau(y),x)$ called a
\emph{polarization} of $(\Gamma, \tau)$. Note that, from the
definition of $\sigma$, one gets that $\tau$ can be extended to a
graph automorphism of $\Gamma$ by setting
$\tau(\alpha)=\sigma^2(\alpha)$ $\forall \alpha\in \Gamma_1$. If
$K\Gamma$ denotes the path algebra of $\Gamma$, then the \emph{mesh
algebra} of $\Gamma$ is $K(\Gamma )=K\Gamma /I$, where $I$ is the
ideal of $K\Gamma$ generated by the so-called \emph{mesh relations}
$r_x$, where $r_x=\sum_{a\in\Gamma_1\text{, }t(a)=x}\sigma (a)a$,
for each $x\in\Gamma_0$.
Therefore $K(\Gamma )$ is canonically a positively
($\mathbb{Z}$-)graded algebra with enough idempotents, where the grading
is induced by the path length, and $\tau$
becomes a graded automorphism of $K(\Gamma )$.

The typical example of stable translation quiver  is the following.
Given a locally finite quiver $\Delta$, the stable translation quiver
$\mathbb{Z}\Delta$ will have as  set of vertices
$\mathbb{Z}\Delta_0:=\mathbb{Z}\times\Delta_0$. Moreover,   for
each arrow $\alpha: x\rightarrow y$ in $\Delta_1$, we have arrows
$(n,\alpha): (n,x)\rightarrow (n,y)$ and $(n,\alpha)':
(n,y)\rightarrow (n+1,x)$ in $\mathbb{Z}\Delta_1$. Finally, we
define $\tau(n,x)=(n-1,x)$, for each $(n,x)\in\mathbb{Z}\Delta_0$,
and $\sigma(n,\alpha)=(n-1,\alpha)'$ and
$\sigma[(n,\alpha)']=(n,\alpha)$.

In general, different quivers $\Delta$ and $\Delta '$ with the same
underlying  graph give non-isomorphic translation quivers
$\mathbb{Z}\Delta$ and $\mathbb{Z}\Delta '$. However, when $\Delta$
is a tree, e.g. when $\Delta$ is any of the Dynkin quivers
$\mathbb{A}_n, \mathbb{D}_{n+1},
\mathbb{E}_{6},\mathbb{E}_{7},\mathbb{E}_{8}$, the isoclass of the
translation quiver $\mathbb{Z}\Delta$ does not depend on the
orientation of the arrows.

A group of automorphism $G$ of a stable translation quiver $(\Gamma
,\tau )$ is a group of automorphism of $\Gamma$ which commute with
$\tau$ and $\sigma$. Such a group is called \emph{weakly admissible}
when $x^+\cap (gx)^+=\emptyset$, for each $x\in\Gamma_0$ and $g\in G\backslash \{1\}$, where
$x^+:=\{y\in\Gamma_0:$ $\Gamma_1(x,y)\neq\emptyset\}$. In such a
case, when $G$ acts freely on objects,  the orbit quiver $\Gamma/G$
inherits a structure of stable translation quiver, with the AR
translation $\bar{\tau}$ mapping $[x]\rightsquigarrow [\tau (x)]$,
for each $x\in\Gamma_0\cup\Gamma_1$. Moreover, the group $G$ can be
interpreted as a group of graded automorphisms of  the mesh algebra
$K(\Gamma )$ and $K(\Gamma)/G$ is canonically isomorphic to the mesh
algebra of $\Gamma /G$.

\subsection{The mesh algebra of a Dynkin quiver. Basic properties}
\label{subsect.definition mesh algebra}

Throughout this section $\Delta$ will be one of the Dynkin quivers
$\mathbb{A}_n$, $\mathbb{D}_{n+1}$ ($n\geq 3$) or $\mathbb{E}_n$
($n=6,7,8$), and  $\mathbb{Z}\Delta$ will be  the associated
translation quiver. Its path algebra will be denoted by
$K\mathbb{Z}\Delta$ and we will put $B=K(\mathbb{Z}\Delta )$ for the
mesh algebra.

When $\Delta =\mathbb{A}_{2n-1}$, $\mathbb{E}_{6}$ or
$\mathbb{D}_{n+1}$, with $n>3$,  the underlying unoriented graph
admits a canonical automorphism $\rho$ of order $2$. Similarly,
$\mathbb{D}_{4}$ admits an automorphism of order $3$. In each case,
the  automorphism $\rho$ extends to an automorphism of
$\mathbb{Z}\Delta$ with the same order. In the case of
$\mathbb{A}_{2n}$ the canonical automorphism of order $2$ of the
underlying graph extends to an automorphism of $\mathbb{Z}\Delta$,
but this automorphism has infinite order. It is still denoted by
$\rho$ and it plays, in some sense, a similar role to the other
cases. This automorphism of \textbf{$\mathbb{Z}\mathbb{A}_{2n}$} is obtained by
applying the symmetry with respect to the central horizontal line and moving
half a unit to the right. Note that we have $\rho^2=\tau^{-1}$.

Although the orientation in $\Delta$ does not change the isomorphism
type of $\mathbb{Z}\Delta$, in order to numbering the vertices of
$\mathbb{Z}\Delta$ we need to fix an orientation in $\Delta$. Below
we fix such an orientation, and then give the corresponding
definition of the automorphism $\rho$ of $\mathbb{Z}\Delta$
mentioned above.

\begin{enumerate}
\item If $\Delta= \mathbb{A}_{2n}:$
$$\xymatrix{1 \ar[r] & 2 \ar[r] & \cdots \ar[r] & 2n},$$ then $\rho (k,i)=(k+i-n,2n+1-i)$.
\item If $\Delta= \mathbb{A}_{2n-1}:$
$$\xymatrix{1 \ar[r] & 2 \ar[r] & \cdots \ar[r] & 2n-1},$$ then $\rho (k,i)=(k+i-n,2n-i)$.
\item $\Delta =\mathbb{D}_{n+1}$:

$$\xymatrix{ 0 & & & \\ & 2\ar[r] \ar[ul] \ar[dl] & \cdots \ar[r] & n
\\ 1 & & &} $$ with $n>3$,  then $\rho (k,0)=(k,1)$, $\rho (k,1)=(k,0)$ and
$\rho$ fixes all vertices $(k,i)$, with $i\neq 0,1$.

\item If $\Delta =\mathbb{D}_4$:

$$
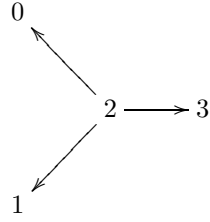
\xymatrix{ 0 & &  \\ & 2\ar[r] \ar[ul] \ar[dl] & 3
\\ 1 & & } $$

then $\rho$ fixes the vertices $(k,2)$ and, for $k$ fixed, it
applies the $3$-cycle $(013)$ to the second component of each vertex
$(k,i)$.

\item  If $\Delta=\mathbb{E}_6$:

$$
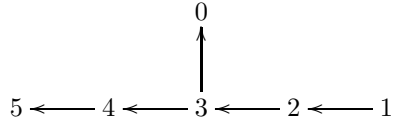
\xymatrix{& & 0 & & \\ 5 & 4 \ar[l] & 3 \ar[l] \ar[u] & 2 \ar[l] & 1\ar[l]}
$$

then $\rho (k,i)=(k+i-3,6-i)$ for all $i\neq 0$ and $\rho(k,0)=(k,0)$.

\item  If $\Delta=\mathbb{E}_7$:

$$
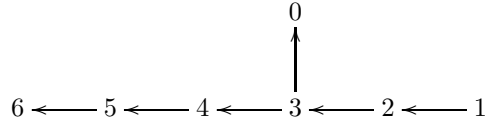
\xymatrix{& & & 0 & & \\ 6  & 5 \ar[l] & 4 \ar[l] & 3 \ar[l] \ar[u] & 2 \ar[l] & 1\ar[l]}
$$

\item  If $\Delta=\mathbb{E}_8$:

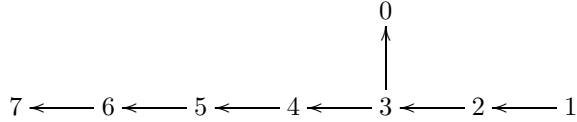
$$\xymatrix{& & & & 0 & & \\ 7 & 6 \ar[l] & 5 \ar[l] & 4 \ar[l] & 3 \ar[l] \ar[u] & 2 \ar[l] & 1\ar[l]}
$$

\end{enumerate}

The following facts are well-known (cf. \cite{BLR}[Section 1.1] and
\cite{G}[Section 6.5]).

\begin{prop}\label{prop:Coxeter-Nakayama} Let $\Delta$ be a Dynkin quiver, $\bar{\Delta}$ be its associated graph,
 $c_{\Delta}$ be its Coxeter number and  $B=K(\mathbb{Z}\Delta)$
be the mesh algebra of the stable translation quiver $\mathbb{Z}\Delta$.
The following assertions hold:

\begin{enumerate}
\item Each path of length $>c_{\Delta }-2$
in $\mathbb{Z}\Delta$ is zero in $B$. \item For each $(k,i)\in
\mathbb{Z}\Delta_0$, there is a unique vertex $\nu (k,i)\in
(\mathbb{Z}\Delta )_0$ for which there is a path $(k,i)\rightarrow
...\rightarrow\nu (k,i)$ in $\mathbb{Z}\Delta$ of length $c_\Delta
-2$  which is nonzero in $B$. This path is unique, up to sign in
$B$.
\item If $(k,i)\rightarrow ...\rightarrow (m,j)$ is a nonzero path
then there is a path $q:(m,j)\rightarrow ...\rightarrow\nu (k,i)$
such that $pq$ is a nonzero path (of length $c_\Delta -2$)
\item The assignment $(k,i)\rightsquigarrow\nu (k,i)$ gives a
bijection $\nu :(\mathbb{Z}\Delta
)_0\longrightarrow(\mathbb{Z}\Delta )_0$, called the \emph{Nakayama
permutation}.

\item The vertex $\nu (k,i)$ is given as follows:

\begin{enumerate}
\item If $\Delta=\mathbb{A}_{r}$, with $r=2n$ or $2n-1$, (hence $c_{\Delta}= r+1$), then
  $\nu
(k,i)=\rho\tau^{1-n}(k,i)=(k+i-1,r+1-i)$
 \item If
$\Delta =\mathbb{D}_{n+1}$ (hence $c_{\Delta}= 2n$), then

\begin{enumerate}
 \item $\nu
(k,i)=\tau^{1-n} (k,i)=(k+n-1,i)$, in case $n+1$ is even
\item $\nu (k,i)=\rho\tau^{1-n}(k,i)$, in case $n+1$ is odd.
\end{enumerate}
\item If $\Delta=\mathbb{E}_6$
 (hence $c_\Delta$=12),  then
$\nu (k,i)=\rho\tau^{-5}(k,i)$.
\item If $\bar{\Delta}=\mathbb{E}_7$ (hence $c_\Delta =18$), with any
orientation, then $\nu (k,i)=\tau^{-8}(k,i)=(k+8,i)$ \item If
$\bar{\Delta}=\mathbb{E}_8$ (hence $c_\Delta =30$), with any
orientation, then $\nu (k,i)=\tau^{-14}(k,i)=(k+14,i)$.
\end{enumerate}
\end{enumerate}
\end{prop}

\begin{cor} \label{cor.simple socle of a projective}
$B$ is a split basic graded Quasi-Frobenius algebra admitting a
graded Nakayama form of degree $\mathit{l}=c_\Delta -2$.
\end{cor}
\begin{proof}
By last proposition, we know that $Be_{(k,i)}$ and $e_{(k,i)}B$ are
finite-dimensional graded $B$-modules. In particular, both are
Noetherian, so that $B$ is a locally Noetherian graded algebra. Note
that $e_{(k,i)}Be_{(k,i)}\cong K$, for each vertex
$(k,i)\in\Gamma_0$, and that $J^{gr}(B)=J(B)$ is the vector subspace
generated by the paths of length $>0$. Therefore $B$ is clearly
split basic. On the other hand, if $\nu$ is the Nakayama permutation
and we fix a nonzero path $w_{(k,i)}:(k,i)\rightarrow...\rightarrow
\nu (k,i)$ of length $\mathit{l}=c_\Delta -2$, then last proposition
says that $w_{(k,i)}$ is in the (graded and ungraded) socle of
$e_{(k,i)}B$.

By conditions 2 and 3 of Proposition \ref{prop:Coxeter-Nakayama}, we
have that  $\text{dim}(\text{Soc}(e_{(k,i)}B))=1$ and that
$\text{Soc}(e_{(k,i)}B)$ is an essential (graded and ungraded)
submodule of $e_{(k,i)}B)$.  Note that $B^{op}$ is the mesh algebra
of the opposite Dynkin quiver $\Delta^{op}$, which is again Dynkin
of the same type. Then also $Be_{(k,i)}$ has essential simple
(graded and ungraded) socle, which is isomorphic to
$S_{\nu^{-1}(k,i)}[\mathit{l}]$ as graded left $B$-module. Then all
conditions of Corollary \ref{cor.Frobenius=simple socle} are
satisfied, with $\nu ' =\nu^{-1}$.

Moreover, by Proposition \ref{prop.graded Nak-form via basis},  we know
that $B$ admits a graded Nakayama form with constant degree function $\mathbf{h}$ such that
$\mathbf{h}(k,i)=\mathit{l}$, for all $(k,i)\in\Gamma_0$.
\end{proof}

\subsection{m-fold mesh algebras}

 When $\Gamma
=\mathbb{Z}\Delta$, with $\Delta$ a Dynkin quiver, it is known that
each weakly admissible group $G$ of automorphisms is infinite cyclic (see
\cite{Ri}, \cite{A}) and below is the list of the resulting stable
translation quivers $\mathbb{Z}\Delta /G$ that appear, where a
generator of $G$ is given in each case (see \cite{D2}). We will denote by
$\langle - \rangle$ the `subgroup generated by' and, in each
case, the following automorphism $\rho$ is always that of the list
preceding Proposition \ref{prop:Coxeter-Nakayama}:

\vspace{1cm}

\hspace{0.5cm}$\bullet$ $\Delta^{(m)}= \mathbb{Z}\Delta/\langle
\tau^m \rangle$, for $\Delta= \mathbb{A}_n, \mathbb{D}_n,
\mathbb{E}_n$.

\hspace{0.5cm}$\bullet$ $\mathbb{B}_n^{(m)}=
\mathbb{Z}\mathbb{A}_{2n-1}/\langle \rho\tau^m \rangle$.

\hspace{0.5cm}$\bullet$ $\mathbb{C}_n^{(m)}=
\mathbb{Z}\mathbb{D}_{n+1}/\langle \rho\tau^m \rangle$.

\hspace{0.5cm}$\bullet$ $\mathbb{F}_4^{(m)}=
\mathbb{Z}\mathbb{E}_{6}/\langle \rho\tau^m \rangle$.

\hspace{0.5cm}$\bullet$ $\mathbb{G}_2^{(m)}=
\mathbb{Z}\mathbb{D}_{4}/\langle \rho\tau^m \rangle$.

\hspace{0.5cm}$\bullet$ $\mathbb{L}_n^{(m)}=
\mathbb{Z}\mathbb{A}_{2n}/\langle \rho\tau^m \rangle$.

 \vspace{1cm}

As shown by Dugas (see \cite{D2}[Section 3]), they are the only
stable translation quivers with finite-dimensional mesh algebras. These
mesh algebras are isomorphic to $\Lambda =B/G$ in each case, where
$B$ is the mesh algebra of $\mathbb{Z}\Delta$. Abusing of notation,
we will simply write $\Lambda =\mathbb{Z}\Delta /\langle\varphi
\rangle$. These algebras are called \emph{m-fold mesh algebras} and are
known to be self-injective, a fact that can be easily seen by
applying Proposition \ref{prop.gr-Frobenius via pushdown functor}
since the cyclic group $G$ acts freely on objects, i.e., on
$(\mathbb{Z}\Delta )_0$. They are also periodic (see \cite{BBK}).

Note that, except for $\mathbb{L}_n^{(m)}$,  each generator of the
group $G$ in the above list is of the form $\rho\tau^m$, where
$\rho$ is an automorphism of order $1$ (i.e. $\rho
=id_{\mathbb{Z}\Delta}$), $2$ or $3$. This leads us to introduce the
following concept, which will be used later on in the paper.

\begin{defi} \label{defi.extended type of a mesh algebra}
Let $\Lambda =\mathbb{Z}\Delta /\langle\rho\tau^m\rangle$ be an $m$-fold mesh
algebra of a Dynkin quiver, possibly with $\rho
=id_{\mathbb{Z}\Delta}$. The \emph{extended type} of $\Lambda$ will
be the triple  $(\Delta ,m,t)$, where $t$ is the order of $\rho$, in
case $\Lambda\neq\mathbb{L}_n^{(m)}$, and $t=2$ when $\Lambda
=\mathbb{L}_n^{(m)}$.
\end{defi}

We warn the reader that the commonly used type for
 stable Auslander algebras of representation-finite self-injective algebras (see \cite{Asa},
\cite{D2},\cite{IV}) does not coincide with the here defined
extended type.

\subsection{A change of presentation}
\label{subsec. change of relations}
For calculation purposes, it is convenient to modify the mesh
relations. We want that if $(k,i)\in(\mathbb{Z}\Delta )_0$ is a
vertex which is the end of exactly two arrows, then the
corresponding mesh relation changes from a sum to a difference. When
$\Delta =\mathbb{D}_{n+1}$ and we consider the three paths
$(k,2)\rightarrow (k,i)\rightarrow (k+1,2)$ ($i=0,1,3$), we want
that the path going through $(k,3)$ is the sum of the other two.
Finally, when $\Delta =\mathbb{E}_n$ ($n=6,7,8$) and we consider the
three paths $(k,3)\rightarrow (k,i)\rightarrow (k+1,3)$ ($i=0,4$) and $(k,3)\rightarrow (k+1,2)\rightarrow (k+1,3)$, we want
that the one going through $(k,0)$ is the sum of the other two. This
can be done by selecting an appropriate subset $X\subset
(\mathbb{Z}\Delta )_1$ and applying the automorphism of
$K\mathbb{Z}\Delta$ which fixes the vertices and all the arrows not
in $X$ and changes the sign of the arrows in $X$. But we want the
same phenomena to pass from $B$ to $\Lambda = B/G$, for any weakly
admissible group of automorphisms $G$ of $\mathbb{Z}\Delta$. This
forces us to impose the condition that $X$ be $G$-invariant, i.e.,
that $g(X)=X$ for each $g\in G$.

\begin{prop} \label{prop.admissible change of variables}
Let $\Delta$ be a Dynkin quiver, $K\mathbb{Z}\Delta$ be the path
algebra of $\mathbb{Z}\Delta$, let $I$ be the ideal of
$K\mathbb{Z}\Delta$ generated by the mesh relations  and let
$\hat{G}$ be the group of automorphisms of $\mathbb{Z}\Delta$
generated by $\rho$ and $\tau$, whenever $\rho$ exists, and just by
$\tau$ otherwise. Let $X\subset (\mathbb{Z}\Delta )_1$ be the set of
arrows constructed as follows:

\begin{enumerate}
\item If $\Delta\neq\mathbb{A}_{2n-1},\mathbb{D}_{4}$ and $X'$ is the
set of arrows given in the following list, then  $X$ is the union of
the $\hat{G}$-orbits of elements of $X'$:

\begin{enumerate}
\item When $\Delta =\mathbb{A}_{2n}$, $X'=\{(0,i)\rightarrow (0,i+1):$ $1\leq i\leq n-1\text{ and }i\not\equiv n\text{ (mod
2)}\}$. \item When $\Delta =\mathbb{D}_{n+1}$, with $n>3$,
$X'=\{(0,i)\rightarrow (0,i+1):$ $2\leq i\leq n-2\text{ and }i\equiv
0\text{ (mod 2)}\}$. \item When $\Delta =\mathbb{E}_{6}$,
$X'=\{(0,2)\rightarrow (0,3)\}$.

\item When $\Delta =\mathbb{E}_n$ ($n=7,8$),
 $X'=\{(0,2)\rightarrow (0,3),\text{ }(0,4)\rightarrow
(1,3),\text{ }(0,6)\rightarrow (1,5)\}$.
\end{enumerate}

\item If $\Delta =\mathbb{D}_4$ and $G=\langle\tau^m\rangle$, then $X$ is the
union of the $\langle\tau \rangle$-orbits of the arrows $(0,2)\rightarrow (0,3)$.

\item If $\Delta =\mathbb{A}_{2n-1}$, then:

\begin{enumerate}
\item When $G=\langle\tau^m\rangle$, $X$ is the union of the $\langle\tau \rangle$-orbits of
arrows in the set $X'=\{(0,i)\rightarrow (0,i+1):$ $1\leq i\leq
2n-3\text{ and }i\not\equiv 0\text{ (mod 2)}\}$. \item When
$G=\langle\rho\tau^m\rangle$, with $m$ odd, $X$ is the union of all
$\langle\rho\tau\rangle$-orbits of arrows in the set $X'=\{(0,i)\rightarrow
(0,i+1):$ $1\leq i\leq n-1\}$. \item When $G=\langle\rho\tau^m\rangle$, with $m$
even, $X$ is the union of the $\langle\rho ,\tau^2\rangle$-orbits of arrows in
the set $X'_1=\{(0,i)\rightarrow (0,i+1):$ $1\leq i\leq n-2\}$ and
the $G$-orbits of arrows in the set $X'_2=\{(2r,i)\rightarrow
(2r,i+1):$ $0\leq 2r<m\text{ and }i=n-1,n\}$.
\end{enumerate}
\end{enumerate}

When $\Delta\neq\mathbb{A}_{2n-1},\mathbb{D}_{4}$, the given set $X$
is $G$-invariant, for all choices of the weakly admissible group of
automorphisms $G$. When $\Delta=\mathbb{A}_{2n-1}$ or $\mathbb{D}_4$, $X$ is
$G$-invariant for the respective group $G$.

Moreover, let $s:X\longrightarrow\mathbb{Z}_2$ be the
\emph{signature map}, where $s(a)=1$ exactly when $a\in X$, and let
$\varphi :K\mathbb{Z}\Delta\longrightarrow K\mathbb{Z}\Delta$ be the
unique graded algebra automorphism which fixes the vertices and maps
$a\rightsquigarrow (-1)^{s(a)}a$, for each $a\in (\mathbb{Z}\Delta
)_1$. Then $\varphi (I)$ is the ideal of $K\mathbb{Z}\Delta$
generated by the relations mentioned in the paragraph preceding this
proposition.
\end{prop}
\begin{proof}
The $G$-invariance of $X$ is clear. In order to prove that  $\varphi
(I)$ is as indicated, note that the mesh relation
$\sum_{t(a)=(k,i)}\sigma (a)a$ is mapped onto
$\sum_{t(a)=(k,i)}(-1)^{s(\sigma (a)a)}\sigma (a)a$, with the
signature $s(p)$ of a path defined as the sum of the signature of
its arrows. The result will follow from the verification of the
following facts, which are routine:

\begin{enumerate}
\item[i)] If $(k,i)$  is the terminus of exactly two arrows $a$ and $b$,
then the set
 $X\cap\{a,b,\sigma (a),\sigma (b)\}$ has only one element.
\item[ii)] When $\Delta =\mathbb{D}_{n+1}$, with $n>3$, and $a:(k-1,3)\rightarrow (k,2)$, $b:(k-1,0)\rightarrow
(k,2)$  and $c:(k-1,1)\rightarrow (k,2)$ are the three arrows ending
at $(k,2)$, then $X\cap\{a,b,c,\sigma (a),\sigma (b),\sigma
(c)\}=\{\sigma (a)\}$ \item[iii)] When $\Delta =\mathbb{E}_n$
($n=6,7,8$) and $a:(k,2)\rightarrow (k,3)$, $b:(k-1,0)\rightarrow
(k,3)$  and $c:(k-1,4)\rightarrow (k,3)$ are the three arrows ending
at $(k,2)$, then $s(\sigma (b)b)\neq 1=s(\sigma (a)a)=s(\sigma
(c)c)$.
\end{enumerate}
\end{proof}

\begin{cor} \label{cor.new-presentation-mesh-algebra}
With the terminology of the previous proposition, the mesh algebra
is isomorphic as a graded algebra to $B':=K\mathbb{Z}\Delta /\varphi
(I)$ and, in each case, the ideal $\varphi (I)$ is $G$-invariant. In
particular, $G$ may be viewed as group of graded automorphisms of
$B'$ and $\varphi$ induces an isomorphism
$B/G\stackrel{\cong}{\longrightarrow}B'/G$.
\end{cor}
\begin{proof}
Since $\varphi$ is a graded automorphism of $K\mathbb{Z}\Delta$ it
induces an isomorphism
$B=K\mathbb{Z}\Delta/I\stackrel{\cong}{\longrightarrow}K\mathbb{Z}\Delta/\varphi
(I)=B'$. If we view $G$ as a group of graded automorphisms of
$K\mathbb{Z}\Delta$, then the fact that $X$ is $G$-invariant implies
that $\varphi\circ g=g\circ\varphi$, for each $g\in G$. From this
remark the rest of the corollary is clear.
\end{proof}

\begin{rem} \label{rem.exceptionality of D4}
When $\Delta =\mathbb{D}_4$ and $G=\langle\rho\tau^m\rangle$, one cannot find a
$G$-invariant set of arrows $X$ as in the above proposition
guaranteeing that, for each $k\in\mathbb{Z}$, the path
$(k-1,2)\rightarrow (k-1,3)\rightarrow (k,2)$ is the sum of the
other two paths from $(k-1,2)$ to $(k,2)$. This is the reason for
the following convention.
\end{rem}

\begin{convention} \label{convention}
From now on in this paper, the term `mesh algebra' will denote the
algebra $K\mathbb{Z}\Delta/\varphi (I)$ given by Corollary
\ref{cor.new-presentation-mesh-algebra}, or just
$K\mathbb{Z}\mathbb{D}_4/I$ in case $(\Delta
,G)=(\mathbb{D}_4,\langle\rho\tau^m\rangle)$. This `new' mesh algebra will be
still denoted by $B$.
\end{convention}

\section{The Nakayama automorphism. Symmetric $m$-fold mesh algebras}

\subsection{The Nakayama automorphism of the mesh algebra of a Dynkin quiver}

The quiver   $\mathbb{Z}\Delta$ does not have double arrows and,
hence, if $a:x\rightarrow y$ is an arrow, then there exists exactly
one arrow $\nu (x)\rightarrow\nu (y)$, where $\nu$ is the Nakayama
permutation. This allows us to extend $\nu$ to an automorphism of
the stable translation quiver $\mathbb{Z}\Delta$ and, hence, also to an
automorphism of the path algebra $K\mathbb{Z}\Delta$. Moreover, due
to the (new) mesh relations (see Proposition \ref{prop.admissible
change of variables} and the paragraph preceding it), we easily see
that if $I'$ is the ideal of $K\mathbb{Z}\Delta$ generated by those
mesh relations, then $\nu (I')=I'$. Note also from Proposition
\ref{prop:Coxeter-Nakayama} that, as an automorphism of the quiver
$\mathbb{Z}\Delta$, we have that $\nu =\tau^k$ or $\nu=\rho\tau^k$,
for a suitable natural number $k$. It follows that if $G$ is any
weakly admissible automorphism of $\mathbb{Z}\Delta$, then $\nu\circ
g=g\circ\nu$ for all $g\in G$. All these comments prove:

\begin{lema} \label{lem.nu compatible with mesh relations}
Let $\Delta$ be a Dykin quiver, $B=B(\Delta)$ be its associated mesh algebra
and $G$ be a weakly admissible group of automorphisms of $\mathbb{Z}\Delta$.
The Nakayama permutation $\nu$ extends to a graded automorphism $\nu
:B\longrightarrow B$ such that $\nu\circ g=g\circ\nu$, for all $g\in
G$.
\end{lema}

%In order to know the explicit formula for the Nakayama automorphism of each $m$-fold mesh algebra, we will provide a suitable %Nakayama automorphism of its Galois cover.

The following result is fundamental for us.

\begin{teor} \label{teor.G-invariant Nakayama automorphism}
Let $\Delta$ be a Dynkin quiver with the labeling of vertices
and the orientation of the arrows of subsection
\ref{subsect.definition mesh algebra},  and let $G=\langle\varphi
\rangle$ be a weakly admissible group of automorphisms of $\mathbb{Z}\Delta$. If
$\eta$ is the graded automorphism of $B$ which acts as the Nakayama
permutation on the vertices and acts on the arrows as indicated in
the following list, then $\eta$ is a Nakayama automorphism of $B$
such that $\eta\circ g=g\circ\eta$, for all $g\in G$.

\begin{enumerate}
\item When $\Delta =\mathbb{A}_n$ and $\varphi$ is arbitrary,  $\eta (\alpha )=\nu (\alpha )$ for all $\alpha\in (\mathbb{Z}\Delta )_1$

\item When $\Delta =\mathbb{D}_{n+1}$:

\begin{enumerate}
\item If $n+1\geq 4$ and $\varphi =\tau^m$ then:

\begin{enumerate}
\item $\eta (\alpha )=-\nu (\alpha )$, whenever $\alpha:(k,i)\longrightarrow (k,i+1)$ is an upward arrow  with
$i\in\{2,...,n-1\}$.

\item $\eta (\alpha )=\nu (\alpha )$, whenever $\alpha:(k,i)\longrightarrow (k+1,i-1)$ is downward arrow  with
$i\in\{3,...,n\}$.

\item $\eta (\varepsilon_i)=(-1)^i\nu (\varepsilon_i)$, for the arrow $\varepsilon_i: (k,2)\longrightarrow (k,i)$ ($i=0,1$),

\item $\eta (\varepsilon'_i)=(-1)^{i+1}\nu (\varepsilon'_i)$, for the arrow $\varepsilon'_i: (k,i)\longrightarrow (k+1,2)$ ($i=0,1$).
\end{enumerate}

\item If $n+1>4$ and $\varphi =\rho\tau^m$  then:

\begin{enumerate}
\item $\eta (\alpha )=-\nu (\alpha )$, whenever $\alpha$  is an upward arrow  as above or  $\alpha:(k,i)\longrightarrow (k+1,i-1)$ is downward arrow  as above such that $k\equiv -1\text{ (mod m)}$.

\item $\eta (\alpha )=\nu (\alpha )$, whenever $\alpha:(k,i)\longrightarrow (k+1,i-1)$ is downward arrow  such that  $k\not\equiv -1\text{(mod m)}$

\item For  the remaining arrows, if $q$ and $r$ are the quotient and remainder of dividing $k$ by $m$, then

$\eta(\varepsilon_i)=(-1)^{q+i} \nu(\varepsilon_i)$ ($i=0,1$).

$\eta(\varepsilon^{'}_i)=(-1)^{q+i+1} \nu(\varepsilon^{'}_i)$,  when
$r\neq m-1$, and $\eta(\varepsilon^{'}_i)=(-1)^{q+i}
\nu(\varepsilon^{'}_i)$ otherwise

\end{enumerate}

\item If $n+1=4$ and $\varphi =\rho\tau^m$ (see Convention \ref{convention}), then:

\begin{enumerate}

 \item $\eta (\varepsilon_i)=\nu (\varepsilon_i)$, whenever $\varepsilon_i:(k,2)\rightarrow
 (k,i)$ ($i=0,1,3$) \item $\eta (\varepsilon'_i)=-\nu
 (\varepsilon'_i)$, whenever $\varepsilon'_i:(k,i)\rightarrow
 (k+1,2)$  ($i=0,1,3$).
\end{enumerate}

\end{enumerate}

\item When $\Delta =\mathbb{E}_6$:

\begin{enumerate}

\item If $\varphi =\tau^m$ then:

\begin{enumerate}

\item $\eta(\alpha)= \nu(\alpha)$  and   $\eta(\alpha^{'})= -\nu(\alpha^{'})$, where  $\alpha:(k,1)\rightarrow
(k,2)$ and $\alpha^{'}:(k,2)\rightarrow (k+1,1)$.

\item $\eta(\beta)= \nu(\beta)$   and  $\eta(\beta^{'})= -\nu(\beta^{'})$, where  $\beta:(k,2)\rightarrow
(k,3)$ and  $\beta^{'}:(k,3)\rightarrow (k+1,2)$.

\item $\eta(\gamma)= \nu(\gamma)$  and  $\eta(\gamma^{'})= -\nu(\gamma^{'})$, where  $\gamma:(k,3)\rightarrow
(k,4)$ and  $\gamma^{'}:(k,4)\rightarrow (k+1,3)$.

\item $\eta(\delta)= -\nu(\delta)$  and  $\eta(\delta^{'})= \nu(\delta^{'})$, where $\delta:(k,4)\rightarrow
(k,5)$ and  $\delta^{'}:(k,5)\rightarrow (k+1,4)$.

\item $\eta(\varepsilon)= -\nu(\varepsilon)$ and   $\eta(\varepsilon^{'})= \nu(\varepsilon^{'})$,  where $\varepsilon:(k,3)\rightarrow
(k,0)$ and  $\varepsilon^{'}:(k,0)\rightarrow (k+1,3)$.

\end{enumerate}

\item If $\varphi =\rho\tau^m$,  $(k,i)$ is the origin of the given arrow,  $q$ and $r$ are the quotient and remainder of dividing $k$ by $m$,   then:

\begin{enumerate}
\item  $\eta(\alpha)= \nu(\alpha)$.

\item  $\eta(\alpha^{'})= -\nu(\alpha^{'})$.

\item $\eta(\beta)= (-1)^{q}\nu(\beta)$

 \item $\eta(\beta^{'})= (-1)^{q+1}\nu(\beta^{'})$

 \item $\eta(\gamma)= (-1)^q\nu(\gamma)$

 \item  $\eta(\gamma^{'})= \nu(\gamma^{'})$, when either $q$ is odd and $r\neq m-1$ or $q$ is even and $r=m-1$, and
  $\eta(\gamma^{'})= -\nu(\gamma^{'})$ otherwise.

 \item $\eta(\delta)= -\nu(\delta)$

 \item  $\eta(\delta^{'})= \nu(\delta^{'})$.

 \item  $\eta(\varepsilon)= -\nu(\varepsilon)$

\item  $\eta(\varepsilon^{'})=-\nu(\varepsilon^{'})$, when $r=m-1$,
and $\eta(\varepsilon^{'})= \nu(\varepsilon^{'})$  otherwise.

\end{enumerate}

\end{enumerate}

\item When $\Delta =\mathbb{E}_7$, $\varphi=\tau^m$, and then:

\begin{enumerate}[i]

\item $\eta(a)$ is given as in $3.(a)$ for any arrow $a$ contained in the copy of $\mathbb{E}_6$.

\item $\eta(\zeta)= \nu(\zeta)$  and  $\eta(\zeta^{'})= -\nu(\zeta^{'})$, where
$\zeta:(k,5)\rightarrow
(k,6)$ and  $\zeta^{'}:(k,6)\rightarrow (k+1,5)$.

\end{enumerate}

\item When $\Delta =\mathbb{E}_8$, $\varphi= \tau^m$, and then:

\begin{enumerate}[i]

\item $\eta(a)$ is given as in $4$ for any arrow $a$ contained in the copy of $\mathbb{E}_7$.

\item $\eta(\theta)= \nu(\theta)$  and  $\eta(\theta^{'})= -\nu(\theta^{'})$, where  $\theta:(k,6)\rightarrow
(k,7)$ and  $\theta^{'}:(k,7)\rightarrow (k+1,6)$.

\end{enumerate}

\end{enumerate}
\end{teor}
\begin{proof}
Let $\nu$ be the Nakayama permutation of $\mathbb{Z}\Delta$ (see
Proposition \ref{prop:Coxeter-Nakayama}). By the proof of Corollary
\ref{cor.simple socle of a projective}, we know that
$Soc_{gr}(e_{(k,i)}B)=Soc(e_{(k,i)}B)$ is one-dimensional and
concentrated in degree $\mathit{l}=c_\Delta -2$, for each $(k,i)\in
\mathbb{Z}\Delta_0$. By  applying  Proposition \ref{prop.graded Nak-form via basis},
after taking a nonzero element
$w_{(k,i)}\in e_{(k,i)}Soc_{gr}(B)$, for each $(k,i)\in
(\mathbb{Z}\Delta )_0$, we can take the graded Nakayama form
$(-,-):B\times B\longrightarrow K$ of degree $\mathit{l}$ associated
to $\mathcal{B}=(\mathcal{B}_{(k,i)})_{(k,i)\in\mathbb{Z}\Delta_0}$
(see definition \ref{defi.graded Nakayama form associated to
basis}), where $\mathcal{B}_{(k,i)}=\{w_{(k,i)}\}$ is a basis of
$e_{(k,i)}B_le_{\nu (k,i)}$, for each $(k,i)\in\mathbb{Z}\Delta_0$.
It is clear that the so obtained graded Nakayama form will be
$G$-invariant whenever
$\mathcal{B}=\bigcup_{(k,i)\in\mathbb{Z}\Delta_0}\mathcal{B}_{(k,i)}$
is $G$-invariant. The existence of a
$G$-invariant basis is guaranteed by Corollary \ref{cor.G-invariant basis and Nakayama form}.
Moreover, in such case, recall that the associated Nakayama
automorphism $\eta$ satisfies that $\eta\circ g=g\circ\eta$, for
all $g\in G$ (see Remark \ref{rem. G-invarinat basis}).
The canonical way of constructing such a $G$-invariant basis $\mathcal{B}$ is also given in that remark. Namely, we select a
set $I'$ of representatives of the $G$-orbits of vertices and a
element $0\neq w_{(k,i)}\in e_{(k,i)}Soc_{gr}(B)$, for each
$(k,i)\in I'$. Then $\mathcal{B}=\{g(w_{(k,i)}):$ $g\in G\text{,
}(k,i)\in I'\}$ is a $G$-invariant basis as desired. However, note
that if we choose $\mathcal{B}$ to be $\tau$-invariant, then it is
$G$-invariant for $G=\langle\tau^m\rangle$. So, in order to construct
$\mathcal{B}$,  we will only need to consider the cases $\varphi
=\tau$ and $\varphi =\rho\tau^m$

Once the $G$-invariant basis $\mathcal{B}$ of $Soc_{gr}(B)=Soc(B)$
has been described, the strategy to identify the action of the
associated Nakayama automorphism $\eta$ on the arrows is very
simple. Given an arrow $\alpha$, we take a path $q:t(\alpha
)\rightarrow ...\rightarrow\nu (i(\alpha ))$ of length
$\mathit{l}-1$ such that $\alpha q$ is a nonzero path. Then we have
$\alpha q=(-1)^{u(\alpha )}w_{i(\alpha )}$, so that, by definition
of the graded Nakayama form associated to $\mathcal{B}$, we have an
equality  $(\alpha ,q)=(-1)^{u(\alpha )}$. Since the quiver
$\mathbb{Z}\Delta$ does not have double arrows we know that $\eta
(\alpha )=\lambda (\alpha )\nu (\alpha )$, for some $\lambda (\alpha
)\in K^*$. In particular we know that $q\nu (\alpha )$ is a nonzero
path (of length $\mathit{l}$) because $(q,\eta (\alpha ))=(\alpha
,q)\neq 0$. If we have an equality $q\nu (\alpha )=(-1)^{v(\alpha
)}w_{t(\alpha )}$ in $B$, then it follows that $(-1)^{u(\alpha
)}=(\alpha ,q)=(q,\eta (\alpha ))=\lambda (\alpha )(q,\nu (\alpha
))=\lambda (\alpha )(-1)^{v(\alpha )}$. Then we get $\lambda (\alpha
)=(-1)^{u(\alpha )-v(\alpha)}$ and the task is reduced to calculate
the exponents $u(\alpha )$ and $v(\alpha )$ in each case. Taking
into account that we have $\eta\circ g=g\circ\eta$, for each $g\in
G$, it is enough to calculate $u(\alpha )$ and $v(\alpha )$ just for
the arrows starting at a vertex of $I'$.

 To construct $\mathcal{B}$ when $\Delta =\mathbb{A}_n$ has no
problem, for all paths of length $\mathit{l}=c_\Delta -2$ from
$(k,i)$ to $\nu (k,i)$ are equal in $B$. So in this case the choice
of $w_{(k,i)}$ will be the element of $B$ represented by a path from
$(k,i)$ to $\nu (k,i)$ and $\mathcal{B}=\{w_{(k,i)}:$ $(k,i)\in
(\mathbb{Z}\Delta )_0\}$ is $G$-invariant for any choice of
$\varphi$.

On what concerns the explicit calculations for the cases when $\Delta$ is either
$\mathbb{D}_{n+1}$ or $\mathbb{E}_r$ ($r=6,7,8$), they can be found in the appendix given at the end
of this paper.

\end{proof}

\subsection{Some important auxiliary results}

Recall that a \emph{walk} in a quiver $Q$ between the vertices $i$
 and $j$ is a finite sequence $i=i_0\leftrightarrow i_1\leftrightarrow ...i_{r-1}\leftrightarrow
 i_r=j$, where each edge $i_{k-1}\leftrightarrow i_k$ is either an
 arrow $i_{k-1}\rightarrow i_k$ or an arrow $i_{k}\rightarrow
 i_{k-1}$. We write such a walk as
 $\alpha_1^{\varepsilon_1}...\alpha_r^{\varepsilon_r}$, where the $\alpha_i$
 are arrows and $\varepsilon_i$ is $1$ or $-1$, depending on whether
 the  corresponding edge is an arrow pointing to the right or to the
 left.

 We will need the following concept from \cite{GA-S}:

\begin{defi} \label{defi.acyclic character}
Let $Q$ be a (not necessarily finite) quiver. An \emph{acyclic
character} on $Q$ (over the field $K$) is a map $\chi
:Q_1\longrightarrow K^*$ such that if
$p=\alpha_1^{\varepsilon_1}...\alpha_r^{\varepsilon_r}$ and
$q=\beta_1^{\varepsilon '_1}...\beta_s^{\varepsilon '_s}$ are two walks of
length $>0$ between any given vertices $i$ and $j$, then
$\prod_{1\leq i\leq r}\chi (\alpha_i)^{\varepsilon_i}=\prod_{1\leq
j\leq s}\chi (\beta_j)^{\varepsilon '_j}$.
\end{defi}

Notice that if $A$ is a graded algebra with enough idempotents, $G\subseteq Aut^{gr}(A)$ is a group acting freely on objets such that $A/G$ is finite dimensional and $f: A\longrightarrow A$ is a graded automorphism commuting with the elements in $G$, then the assignment $[a]\rightsquigarrow [f(a)]$, with $a\in\bigcup_{i,j\in
I}e_iAe_j$, determines a graded automorphism $\bar{f}$ of $A/G$.

The following general result will be very useful.

\begin{lema} \label{lem.criterion of inner automorphism}
Let $A=\oplus_{n\geq 0}A_n$ be a  basic positively
$\mathbb{Z}$-graded pseudo-Frobenius algebra with enough idempotents
such that $e_iA_0e_i\cong K$, for each $i\in I$,  let $G$ be a group
of graded automorphisms of $A$ acting freely on objects such that
$\Lambda =A/G$ is finite dimensional and let $f:A\longrightarrow
A$ be a graded automorphism that fixes all idempotents $e_i$ and satisfies that $f\circ g=g\circ f$, for all $g\in G$.

Then the following assertions are equivalent:

\begin{enumerate}
\item $\bar{f}$ is an inner automorphism of $\Lambda$.

\item There is a map $\lambda :I\longrightarrow K^*$ such that
$f(a)=\lambda (i(a))^{-1}\lambda (t(a))a$, for all
$a\in\bigcup_{i,j\in I}e_iAe_j$, and $\lambda\circ g_{|I}=\lambda$,
for all $g\in G$.
\end{enumerate}
\end{lema}
\begin{proof}

$1)\Longrightarrow 2)$ Let $\lambda :I\longrightarrow K^*$ be any
map and $\chi_\lambda:A\longrightarrow A$ be the (graded) automorphism which
is the identity on objects and maps $a\rightsquigarrow
\lambda(i(a))^{-1}\lambda (t(a))a$, for each
$a\in\bigcup_{i,j}e_iAe_j$.

If now $f$ is as in the statement and $\bar{f}$ is inner, the goal is to find a
map $\lambda$ as in the previous paragraph such that
$\chi_\lambda=f$ and
$\lambda\circ g_{|I}=\lambda$, for all $g\in G$.

We know from Proposition \ref{prop.gr-Frobenius via pushdown functor}
that $\Lambda$ is a split basic graded algebra. So it is given by a
finite quiver with relations whose set of vertices is (in
bijection with) the set $I/G=\{[i]:i\in I\}$ of $G$-orbits of
elements of $I$. From \cite{GA-S}[Proposition 10 and Theorem 12]  we
get a map $\bar{\lambda}:I/G\longrightarrow K^*$ such that the
assignment
$[a]\rightsquigarrow\bar{\lambda}([i(a)])^{-1}\bar{\lambda}([t(a)])[a]$,
where $a\in\bigcup_{i,j\in I}e_iAe_j$, is a (graded) inner
automorphism $u$ of $\Lambda$ such that $u^{-1}\circ\bar{f}$ is the
inner automorphism $\iota =\iota_{1-x}$ of $\Lambda$ defined by an
element of the form $1-x$, where $x\in J(\Lambda )$. In our
situation, the equality $J(\Lambda )=\oplus_{n>0}\Lambda_n$ holds,
so that $x$ is a sum of homogeneous elements of degree $>0$. But
 $\iota =u^{-1}\circ\bar{f}$ is also a graded automorphism, so that we
have that $\iota (\Lambda_n)=(1-x)\Lambda_n(1-x)^{-1}=\Lambda_n$. If $y\in
\Lambda_n$ then the $n$-th homogeneous component of
$(1-x)y(1-x)^{-1}$ is $y$. It follows that $\iota$ is the identity
on $\Lambda_n$, for each $n\geq 0$. Therefore we have $\iota =id_\Lambda$,
so that $\bar{f}=u$.

Let now $\pi :A\longrightarrow\Lambda =A/G$ be the $G$-covering
functor and let $\lambda$ be the composition map
$I\stackrel{\pi}{\longrightarrow}I/G\stackrel{\bar{\lambda
}}{\longrightarrow}K^*$. By definition, we have that $\lambda\circ
g=\lambda$, for all $g\in G$. As a consequence, the associated
automorphism $\chi_\lambda :A\longrightarrow A$ defined above has
the property that $[\chi_\lambda (a)]=u([a])=\bar{f}([a])=[f(a)]$,
for each $a\in\bigcup_{i,j}e_iAe_j$. Since $f$ is the identity on
objects we immediately get that $f=\chi_\lambda$ as desired.

$2)\Longrightarrow 1)$ The map $\lambda$ of the hypothesis satisfies
that $\chi_\lambda=f$. It then follows that
$\bar{\chi}_\lambda=\bar{f}$, where $\bar{\chi}_\lambda
:\Lambda\longrightarrow\Lambda$ maps $[a]\rightsquigarrow\lambda
(i(a))^{-1}\lambda (t(a))[a]$, for each $a\in\bigcup_{i,j}e_iAe_j$.
Note that $\bar{\chi}_\lambda$ is well defined because $\lambda\circ
g_{| I}=\lambda$, for all $g\in G$.  It turns out that
$\bar{\chi}_\lambda$ is the inner automorphism of $\Lambda$ defined
by the element $\sum_{[i]\in I/G}\lambda (i)^{-1}e_{[i]}$.

\end{proof}

In the rest of the paper, for any m-fold mesh algebra $\Lambda$, we shall denote by $\text{Inn}^{gr}(\Lambda )$ the subgroup of $\text{Aut}^{gr}(\Lambda )$ consisting of those graded automorphisms which are inner.

\begin{prop} \label{prop.tau nu eta in center}
Let $\Delta$ be a Dynkin quiver, let $G$ be a weakly admissible group of automorphisms of $\mathbb{Z}\Delta$, let $\Lambda =B(\Delta )/G$ be the associated m-fold mesh algebra and let $\eta$ be any $G$-invariant graded Nakayama automorphism of $B=B(\Delta )$.  The images  of $\bar{\tau}$, $\bar{\nu}$ and $\bar{\eta}$ by the canonical group homomorphism $\text{Aut}^{gr}(\Lambda )\longrightarrow \text{Aut}^{gr}(\Lambda )/ \text{Inn}^{gr}(\Lambda )$ are all in the center of $ \text{Aut}^{gr}(\Lambda )/ \text{Inn}^{gr}(\Lambda )$.
\end{prop}

\begin{proof}
Due to the fact that the quiver of $\Lambda$ (i.e. $Q=\mathbb{Z}\Delta/G$) does not have double arrows,  each graded automorphism $\varphi$ of $\Lambda$ induces an automorphism $\tilde{\varphi}$ of $Q$. This automorphism extends to an automorphism of the path algebra $KQ$ which respects the mesh relations. Therefore we can  look at $\tilde{\varphi}$  as a graded automorphism of $\Lambda$ as well. Since $\tilde{\varphi}$ and $\bar{\tau}$ commute when viewed as automorphisms of the quiver $Q$, we get that they also commute as graded automorphisms of $\Lambda$. On the other hand, if $(-,-):\Lambda\times\Lambda\longrightarrow K$ is a  graded Nakayama form whose associated Nakayama automorphism is $\bar{\eta}$, then the map $[-,-]=\Lambda\times\Lambda\longrightarrow K$ given by $[a,b]=(\varphi (a),\varphi (b))$ is a graded Nakayama form whose associated Nakayama automorphism is $\varphi^{-1}\circ\bar{\eta}\circ\varphi$. But all Nakayama automorphisms of $\Lambda$ are equal, up to composition by an inner automorphism. We conclude that $\varphi^{-1}\circ\bar{\eta}\circ\varphi\circ\bar{\eta}^{-1}\in\text{Inn}^{gr}(\Lambda )$, so that the statement about $\bar{\eta}$ is proved. Moreover,  the Nakayama permutation of $Q_0$ associated to $\varphi^{-1}\circ\bar{\eta}\circ\varphi$ ``is''  $\tilde{\varphi}^{-1}\circ\bar{\nu}\circ\tilde{\varphi}$, which, by Proposition \ref{prop.graded Nak-form via basis}, implies that $\bar{\nu}\circ\tilde{\varphi}=\tilde{\varphi}\circ\bar{\nu}$, an equality that may be seen as an equality of (graded) automorphisms of $\Lambda$.

From the previous paragraph we also deduce that each $\varphi\in\text{Aut}^{gr}(\Lambda )$ decomposes as $\varphi =\tilde{\varphi}\circ\bar{\chi}$, where $\bar{\chi}$ is an automorphism of $\Lambda$ which is the identity on vertices of $Q$. This means that we also have a map $\bar{\xi}:Q_1=\frac{\mathbb{Z}\Delta_1}{G}\longrightarrow K^*$ such that $\bar{\chi} ([a])=\bar{\xi} ([a])[a]$, for each $a\in\mathbb{Z}\Delta_1$. Denoting by $\xi$ the composition $\mathbb{Z}\Delta_1\stackrel{\pi}{\longrightarrow}:Q_1=\frac{\mathbb{Z}\Delta_1}{G}\stackrel{\bar{\xi}}{\longrightarrow}K^*$, we have that $\xi\circ g_{| \mathbb{Z}\Delta_1}=\xi$, for all $g\in G$, and $\bar{\chi}([a])=\xi (a)[a]$, for all  $a\in\mathbb{Z}\Delta_1$. We then get a $G$-invariant graded  automorphism $\chi =\chi^\xi$ of $B=B(\Delta )$ which fixes the vertices,  maps $a\rightsquigarrow\xi (a)a$, for all $a\in\mathbb{Z}\Delta_1$, and has the property that $\bar{\chi}([a])=[\chi (a)]$, for all $a\in\mathbb{Z}\Delta_1$.

By the first paragraph of this proof, we know that $\bar{\tau}$ and $\bar{\nu}$ commute with the (graded) automorphisms of $\Lambda$ given by quiver automorphisms of $Q$. In order to end the proof, we just need to check that they commute, up to composition with an inner automorphism,  with all automorphisms of $\Lambda$ of the form $\bar{\chi}$, where $\chi =\chi^\xi$, for some map $\xi :\mathbb{Z}\Delta_1\longrightarrow K^*$ such that $\xi\circ g_{| \mathbb{Z}\Delta_1}=\xi$ for all $g\in G$.
Moreover, by Proposition \ref{prop:Coxeter-Nakayama}, we know that $\nu=\rho^k\circ\tau^{l}$, where $k\in\{0,1\}$ and $l>0$. Our task is hence reduced to check that $\bar{\psi}\circ\bar{\chi}\circ\bar{\psi}^{-1}\circ\bar{\chi}^{-1}$ is inner, when $\psi\in\{\rho ,\tau\}$ and $\chi =\chi^\xi$, for some map $\xi$ as above. Note that  $\psi\circ\chi\circ\psi^{-1}\circ\chi^{-1}$ is an automorphism of $B$ which fixes the vertices, commutes with all elements of $G$ and maps $a\rightsquigarrow\xi (a)^{-1}\xi (\psi^{-1}(a))a$, for each $a\in\mathbb{Z}\Delta_1$. Then Lemma \ref{lem.criterion of inner automorphism} can be used. Note that the compatibility of $\chi$ with the mesh relations implies that, for each $(k,i)\in\mathbb{Z}\Delta_0$, the product $\xi (\sigma (a))\xi (a)$ is constant on the set of arrows ending at $(k,i)$. It follows easily that the map  $\xi':\mathbb{Z}\Delta_1\longrightarrow K^*$, which takes $a\rightsquigarrow\xi (a)^{-1}\xi (\psi^{-1}(a))$,   is an acyclic character of $\mathbb{Z}\Delta$.  By  \cite{GA-S}[Proposition 10] and its proof, there is a map $\lambda =\mathbb{Z}\Delta_0\longrightarrow K^*$  such that $\lambda_{i(a)}^{-1}\lambda_{t(a)}=\xi (a)^{-1}\xi (\psi^{-1}(a))$, for all $a\in\mathbb{Z}\Delta_1$.

By Lemma \ref{lem.criterion of inner automorphism}, our task is reduced to prove that such a map $\lambda$ has the property that $\lambda\circ g_{| \mathbb{Z}\Delta_0}=\lambda$, for all $g\in G$. We claim that if $G=<\varphi>$ and we have that $(\lambda\circ\varphi)(r,j)=\lambda (r,j)$ for  one vertex $(r,j)$, then $\lambda\circ g_{\mathbb{Z}\Delta_0}=\lambda$, for all $g\in G$. Indeed, suppose that this is the case.  By definition of $\lambda$,  if $\lambda\circ\varphi$ and $\lambda$ act the same on the origin or the terminus of a given arrow $a$, then they act the same both on $i(a)$ and $t(a)$. Then, by the connectedness of $\mathbb{Z}\Delta$, we get that $\lambda\circ\varphi_{| \mathbb{Z}\Delta_0}=\lambda$ as desired. We then pass to prove, for all possibilities of the extended type $(\Delta ,m,t))$,  that there is a  $(r,j)\in\mathbb{Z}\Delta_0$ such that $(\lambda\circ\varphi )(r,j)=\lambda(r,j)$, where $\varphi$ is either $\tau^m$, when $t=1$, or $\rho\circ\tau^m$, when $t=2,3$.
Indeed, except when $(\Delta ,m,t) =(\mathbb{A}_{2n},m,2)$, we always have  a  $j\in\Delta_0$ such that $\rho(k,j)=(k,j)$, for all $k\in\mathbb{Z}$ (here we consider also the case when $t=1$, and hence $\rho =id_{\mathbb{Z}\Delta }$). In particular, we have $\varphi (0,j)=\tau^m(0,j)=(-m,j)$. We fix such a vertex and consider the canonical path $\varphi (0,j)=(-m,j)\stackrel{\sigma^{2m-1}(a)}{\longrightarrow}\bullet\stackrel{\sigma^{2m-2}(a)}{\longrightarrow}(-m+1,j)\longrightarrow ...(-1,j)\stackrel{\sigma (a)}{\longrightarrow}\bullet\stackrel{a}{\longrightarrow}(0,j)$. By definition of $\lambda$, we then have an equality

\begin{center}
$\lambda_{(0,j)}=\lambda_{\varphi (0,j)}\xi(\sigma^{2m-1}(a))^{-1}\xi (\psi (\sigma^{2m-1}(a)))\xi(\sigma^{2m-2}(a))^{-1}\xi (\psi (\sigma^{2m-2}(a)))\cdot ...\cdot\xi(\sigma(a))^{-1}\xi (\psi (\sigma (a)))\xi(a)^{-1}\xi (\psi (a))=\lambda_{\varphi (0,j)}\cdot (\prod_{1\leq k\leq 2m-1}\xi (\sigma^k(a)))^{-1}\cdot (\prod_{1\leq k\leq 2m-1}\xi (\psi (\sigma^k(a)))).$
\end{center}

We need to prove that the equality $(\prod_{0\leq k\leq 2m-1}\xi (\sigma^k(a)))^{-1}\cdot (\prod_{0\leq k\leq 2m-1}\xi (\psi (\sigma^k(a))))=1$ holds. Note that, acting on arrows, both $\rho$ and $\tau$ commute with $\sigma$. For $\psi =\tau$, we then get that

\begin{center}
$(\prod_{0\leq k\leq 2m-1}\xi (\sigma^k(a)))^{-1}\cdot (\prod_{0\leq k\leq 2m-1}\xi (\psi (\sigma^k(a))))=[\xi (\sigma (a))\xi (a)]^{-1}\cdot [\xi(\sigma^{2m+1}(a))\xi(\sigma^{2m}(a))]=[\xi (\sigma (a))\xi (a)]^{-1}\cdot [\xi(\sigma(\tau^m(a)))\xi(\tau^m(a))].$ \hspace*{1cm} (*)
\end{center}
When $G=<\tau^m>$, this expression is equal to $1$ since $\xi\circ g_{| \mathbb{Z}\Delta_1}=\xi$, for all $g\in G$. On the other hand, if $G=<\rho\tau^m>$ (and hence $\Delta\neq\mathbb{A}_{2n}$), then  $\rho (\sigma^{2r}(a))$ and $\sigma^{2r}(a)$ are arrows ending at $(-r,j )$, for each integer $r\geq 0$, due to the choice of $j$.
 It follows that $\xi (\sigma^{2r+1}(\rho (a)))\xi (\sigma^{2r}(\rho (a)))=\xi (\sigma^{2r+1}(a))\xi (\sigma^{2r}(a))$, for each $r\geq 0$ (see the third paragraph of this proof). We then get $(\prod_{0\leq k\leq 2m-1}\xi (\sigma^k(a)))^{-1}\cdot (\prod_{0\leq k\leq 2m-1}\xi (\rho (\sigma^k(a))))=1$ also in this case.

We finally consider the case when $(\Delta ,t)=(\mathbb{A}_{2n},2)$. Similarly to the other cases, we consider the canonical path

\begin{center}
$\varphi (0,n)=\rho\tau^m(0,n)=(-m,n+1)\stackrel{\sigma^{2m-2}(a)}{\longrightarrow}(-m+1,n)\stackrel{\sigma^{2m-3}(a)}{\longrightarrow}(-m+1,n+1)\longrightarrow ...\longrightarrow (-1,n)\stackrel{\sigma(a)}{\longrightarrow}(-1,n+1)\stackrel{a}{\longrightarrow}(0,n)$.
\end{center}
With the argument used above, we need to check that

$(\prod_{0\leq k\leq 2m-2}\xi (\sigma^k(a)))^{-1}\cdot (\prod_{0\leq k\leq 2m-2}\xi (\psi (\sigma^k(a))))=1$, for $\psi=\rho ,\tau$.

Note that we have $\rho (\sigma^k(a))=\sigma^{k-1}(a)$, for each $k\in\mathbb{Z}$. As a consequence, we get:

\begin{center}
$(\prod_{0\leq k\leq 2m-2}\xi (\sigma^k(a)))^{-1}\cdot (\prod_{0\leq k\leq 2m-2}\xi (\tau (\sigma^k(a))))=\xi (\sigma(a))^{-1}\cdot\xi (a)^{-1}\cdot\xi (\sigma^{2m}(a))\cdot\xi (\sigma^{2m-1}(a))=\xi (\sigma(a))^{-1}\cdot\xi (a)^{-1}\cdot\xi(\rho\tau^m)(\sigma (a))\cdot\xi (\rho\tau^m(a))$,
\end{center}
 and

\begin{center}
$(\prod_{0\leq k\leq 2m-2}\xi (\sigma^k(a)))^{-1}\cdot (\prod_{0\leq k\leq 2m-2}\xi (\rho (\sigma^k(a))))=\xi (\sigma^{2m-2}(a))^{-1}\cdot\xi (\sigma^{-1}(a))=\xi (\rho\tau^m(\sigma^{-1}(a)))^{-1}\cdot\xi (\sigma^{-1}(a))$.
\end{center}
Both expressions are equal to $1$ since $\xi\circ g_{\mathbb{Z}\Delta_1}=\xi$, for all $g\in G=<\rho\tau^m>$.
\end{proof}

 The following is the identification of a subgroup of the integers
which is crucial for our purposes.

\begin{prop} \label{prop.subgroup of Z associated}
Let $\Lambda$ be the $m$-fold mesh algebra of extended type $(\Delta
,m,t)$ and let  $H(\Delta ,m,t)$ be the set of integers $s$ such
that $\bar{\eta}^{s}\bar{\nu}^{-s}$ is an inner automorphism of
$\Lambda$. Then $H(\Delta ,m,t)$ is a subgroup of $\mathbb{Z}$ and
the following assertions hold:

\begin{enumerate}
\item If $\text{char}(K)=2$ or $\Delta =\mathbb{A}_r$, then $H(\Delta ,m,t)=\mathbb{Z}$.
\item If $\text{char}(K)\neq 2$ and $\Delta\neq\mathbb{A}_r$, then $H(\Delta
,m,t)=\mathbb{Z}$, when $m+t$ is odd, and $H(\Delta
,m,t)=2\mathbb{Z}$ otherwise.
\end{enumerate}
\end{prop}
\begin{proof}
Theorem \ref{teor.G-invariant Nakayama automorphism} gives a map $\xi :\mathbb{Z}\Delta_1\longrightarrow K^*$, where $\xi (a)=(-1)^{u(a)}$, for some $u(a)\in\mathbb{Z}_2$, such that $\xi\circ g_{\mathbb{Z}\Delta_1}=\xi$ and if $\chi =\chi^\xi$, with the terminology of the proof of Proposition \ref{prop.tau nu eta in center}, then $\eta:=\nu\circ\chi$ is a $G$-invariant graded Nakayama automorphism of $B$. By the last mentioned proposition, we conclude that $H(\Delta ,m,t)=\{s\in\mathbb{Z}:$ $\bar{\chi}^s\in\text{Inn}(\Lambda )\}$. It is clearly a subgroup of $\mathbb{Z}$, which contains $2$ since $\chi\circ\chi =id_B$. Moreover, when $\text{char}(K)=2$ or $\Delta =\mathbb{A}_r$, we have $\chi =id_B$ so that $H(\Delta ,m,t)=\mathbb{Z}$ in this case.
Our task reduces to determine, for the remaining cases,  when $\bar{\chi}\in\text{Inn}(\Lambda )$.

Using Lemma \ref{lem.criterion of inner automorphism} and arguing as in the proof of Proposition \ref{prop.tau nu eta in center}, we have a map $\lambda :\mathbb{Z}\Delta_0\longrightarrow K^*$ such that $\lambda_{i(a)}^{-1}\lambda_{t(a)}=\xi (a)=(-1)^{u(a)}$, for all $a\in\mathbb{Z}\Delta_1$, and we need to determine when $\lambda\circ\varphi_{| \mathbb{Z}\Delta_0}=\lambda$, where $\varphi$ is the canonical generator of $G$. We emulate the proof of the previous proposition
and fix a $j\in\Delta_0$ such that $\rho (0,j)=(0,j)$. By the proof of the mentioned proposition, it is enough to show that $\lambda (\varphi (0,j))=\lambda (0,j)$. We have a path

\begin{center}
$\varphi (0,j)=(-m,j)\stackrel{\sigma^{2m-1}(a)}{\longrightarrow}\bullet\stackrel{\sigma^{2m-2}(a)}{\longrightarrow}(-m+1,j)\longrightarrow ...\longrightarrow (-1,j)\stackrel{\sigma (a)}{\longrightarrow}\bullet\stackrel{a}{\longrightarrow}(0,j)$,
\end{center}
which, by the definition of $\lambda$, implies that

$$\lambda_{(0,j)}=\prod_{0\leq k\leq 2m-1}(-1)^{u(\sigma^k(a))}\lambda_{\varphi (0,j)}=(-1)^{\sum_{0\leq k\leq 2m-1}u(\sigma^k(a)}\lambda_{\varphi (0,j)}$$

When $t=1$, so that $\varphi =\tau^m$, Theorem \ref{teor.G-invariant Nakayama
automorphism} gives that, for all choices of $\Delta$, we have  $u(b)\neq u(\sigma (b))$, for all $b\in\mathbb{Z}\Delta_1$. It follows that $\sum_{0\leq k\leq 2m-1}u(\sigma^k(a)=m$. Similarly,  when $t=3$ (whence $\Delta =\mathbb{D}_4$), we have $j=2$ and, looking at Theorem \ref{teor.G-invariant Nakayama automorphism},   one  gets that, for any choice of the arrow $a$ ending at $(0,2)$, the equality  $\sum_{0\leq k\leq 2m-1}u(\sigma^k(a)=m$ also holds. In both cases, we conclude that the equality $\lambda_{\varphi (0,j)}=\lambda_{(0,j)}$ holds exactly when $m$ is even.

Finally, suppose that  $t=2$, so that $\varphi =\rho\tau^m$. We take as arrow $a$ ending at $(0,j)$ the following:

\begin{enumerate}
\item When $\Delta =\mathbb{D}_{n+1}$, with $n+1>4$, we have $j=2$ and take  $a:(-1,3)\rightarrow (0,2)$;
\item When $\Delta =\mathbb{E}_6$, we have $j=3$ and take $a=\epsilon':(-1,0)\rightarrow (0,3)$.
\end{enumerate}
Using Theorem \ref{teor.G-invariant Nakayama automorphism}, we see that in the family of exponents $\{u(\sigma^r(a))\text{: }r=0,1,...,2m-1\}$ there are exactly $m+1$ which are equal to $1$. It follows that  $\lambda_{\varphi (0,j)}=\lambda_{(0,j)}$ holds exactly when $m$ is odd.

\end{proof}

\subsection{Symmetric and weakly symmetric $m$-fold mesh algebras}

The only result of this subsection identifies all the weakly
symmetric and symmetric $m$-fold mesh algebras.

\begin{teor} \label{teor.weakly symmetric m-fold algebras}
Let $\Lambda$ be an $m$-fold mesh algebra of extended type $(\Delta
,m,t)$. If $\Lambda$ is weakly symmetric then $t=1$ or $t=2$ and,
when $\text{char}(K)=2$ or $\Delta =\mathbb{A}_r$, such an algebra
is also symmetric. Moreover, the following assertions hold:

\begin{enumerate}
\item When $t=1$, $\Lambda$ is weakly symmetric if, and only if, $\Delta$ is $\mathbb{D}_{2r}$, $\mathbb{E}_{7}$ or
$\mathbb{E}_{8}$ and $m$ is a divisor of $\frac{c_\Delta}{2}-1$.
When $\text{char}(K)\neq 2$, such an algebra is symmetric if, and
only if,  $m$ is even.
\item When $t=2$ and $\Delta\neq\mathbb{A}_{2n}$, $\Lambda$ is
weakly symmetric if, and only if, $m$ divides $\frac{c_\Delta}{2}-1$
and, moreover,  the quotient of the division is odd, in case $\Delta
=\mathbb{A}_{2n-1}$, and even, in case $\Delta =\mathbb{D}_{2r}$.
When $\text{char}(K)\neq 2$, such an algebra is symmetric if, and
only if, $\Delta =\mathbb{A}_{2n-1}$ or $m$ is odd.
\item When $(\Delta,m ,t)=(\mathbb{A}_{2n},m,2)$, i.e. $\Lambda
=\mathbb{L}_n^{(m)}$, the algebra is (weakly) symmetric if, and only
if, $2m-1$ divides $2n-1$.
\end{enumerate}
\end{teor}
\begin{proof}
The algebra $\Lambda$ is weakly symmetric if, and only if, the
automorphism $\bar{\nu}:\Lambda\longrightarrow\Lambda$ induced by
$\nu$ is the identity on vertices. We identify the vertices of the
quiver of $\Lambda$ as $G$-orbits of vertices of
$\mathbb{Z}\Delta_0$, where $G$ is the weakly admissible group of
automorphism considered in each case. If we take care to choose a
vertex $(k,i)$ which is not fixed by $\rho$, then the equality
$\bar{\nu}([(k,i)])=[(k,i)]$ holds exactly when there is a $g\in G$
such that $\nu (k,i)=g(k,i)$. But if $\hat{G}$ denotes the group of
automorphisms generated by $\rho$ and $\tau$, then $\hat{G}$  acts
freely on the vertices not fixed by $\rho$. Since $G\subset\hat{G}$
and $\nu\in\hat{G}$ (see Proposition \ref{prop:Coxeter-Nakayama})
the equality $\nu (k,i)=g(k,i)$ implies that $\nu =g$. Therefore the
algebra $\Lambda$ is weakly symmetric if, and only if, $\nu$ belongs
to $G$.

On the other hand, $\Lambda$ is symmetric if, and only if,
$\bar{\eta}:\Lambda\longrightarrow\Lambda$ is an inner automorphism.
By Lemma \ref{lem.criterion of inner automorphism}, this is
equivalent to saying that $\Lambda$ is weakly symmetric and
$\bar{\eta}\circ\bar{\nu}^{-1}$ is an inner automorphism of
$\Lambda$. That is,  $\Lambda$ is symmetric if, and only if,
$\Lambda$ is weakly symmetric and $H(\Delta ,m,t)=\mathbb{Z}$. As a
consequence, once the weakly symmetric $m$-fold mesh algebras have
been identified, the part of the theorem referring to symmetric
algebras follows directly from   Proposition \ref{prop.subgroup of Z
associated}.

 If $t=3$, then
$\Delta=\mathbb{D}_{4}$,   $G=\langle\rho\tau^m\rangle$, with $\rho$ acting on
vertices as the $3$-cycle $(013)$, and $\nu=\tau^{-2}$. It is
impossible to have $\tau^{-2}\in G$ and therefore $\Lambda$ is never
weakly symmetric in this case.

If $t=1$, then $G=\langle\tau^m\rangle$. If we assume that
$\Delta\neq\mathbb{D}_{2r},\mathbb{E}_{7},\mathbb{E}_{8}$ then $\nu
=\rho\tau^{1-n}$, for some integer $n$. Again it is impossible that
$\nu\in G$ and, hence, $\Lambda$ cannot be weakly symmetric. On the
contrary,  suppose that $\Delta$ is one of
the Dynkin quivers $\mathbb{D}_{2r},\mathbb{E}_{7},$ or $\mathbb{E}_{8}$. Then
$\nu=\tau^{1-n}$, with $n=\frac{c_\Delta}{2}$, and $\nu$ belongs to
$G$ if, and only if, there is an integer $r$  such that
$\tau^{1-n}=(\tau^m)^r$, which is equivalent to saying that
$n-1=-mr$ since $\tau$ has infinite order. Then $\Lambda$ is weakly
symmetric in this case if, and only if, $m$ divides $n-1$.

Suppose now that $t=2$ and $\Delta\neq\mathbb{A}_{2n}$. Then
$G=\langle\rho\tau^m\rangle$ and, except when $\Delta =\mathbb{D}_{2r}$, we have
that  $\nu=\rho\tau^{1-n}$, where $n=\frac{c_\Delta}{2}$. Assume
that $\Delta\neq\mathbb{D}_{2r}$. Then $\nu$ is in $G$ if, and only
if, there is an integer $r$ such that
$\rho\tau^{1-n}=(\rho\tau^m)^r$. Note that then $r$ is necessarily
odd.  If follows that $\Lambda$ is weakly symmetric if, and only if,
$m$ divides $n-1$ and the quotient $\frac{n-1}{m}$ is an odd number.
But the condition on $\frac{n-1}{m}$ to be odd is superfluous when
$\Delta =\mathbb{D}_{2r+1}$ or $\mathbb{E}_{6}$ because $n$ is even
in both cases.

Consider now the case in which $(\Delta ,t)=(\mathbb{D}_{2r},2)$.
Then $\nu=\tau^{1-n}$, where $n=\frac{c_\Delta}{2}=2r-1$. Then $\nu$
is in $G$ if, and only if, there is an integer $s$ such that
$\tau^{1-n}=(\rho\tau^m)^s$. This forces $s$ to be even. We then get
that $\Lambda$ is weakly symmetric if, and only if, $m$ divides
$n-1$ and the quotient $\frac{n-1}{m}$ is even.

Finally, let us consider the case when the extended type is
$(\mathbb{A}_{2n},m,2)$. In this case $\rho^2=\tau^{-1}$ and $\nu
=\rho\tau^{1-n}$. Then $\nu$ is in $G$ if, and only if, there is an
integer $r$ such that $\rho\tau^{1-n}=(\rho\tau^{m})^r$. This forces
$r=2s+1$ to be odd, and then
$\rho\tau^{-s+m(2s+1)}=(\rho\tau^m)^{2s+1}=\rho\tau^{1-n}$. Then
$\Lambda$ is weakly symmetric if, and only if, there is an integer
$s$ such that $(2m-1)s=1-m-n$. That is, if and only if $2m-1$
divides $m+n-1$, which is equivalent to saying that  $2m-1$ divides
$2(m+n-1)-(2m-1)=2n-1$.
\end{proof}

\section{The   period and the stable Calabi-Yau dimension of an $m$-fold mesh algebra}

\subsection{The minimal projective resolution of the regular bimodule}

\begin{lema} \label{lem.G-invarian basis}
Let $\Delta$ be a Dynkin quiver and $B=B(\mathbb{Z}\Delta)$ be its associated mesh
algebra. For any weakly admissible group of automorphisms $G$ of
$\mathbb{Z}\Delta$, there is a basis $\mathcal{B}$ of $B$ consisting
of paths which is $G$-invariant (i.e. $g(\mathcal{B})=\mathcal{B}$
for all $g\in G$).
\end{lema}
\begin{proof}
The way of constructing the basis $\mathcal{B}$ is
analogous to the way in which a $G$-invariant basis of
$\text{Soc}(B)$ was constructed (see the initial  paragraphs of the
proof of Theorem \ref{teor.G-invariant Nakayama automorphism}). The
task is reduced to find, for each vertex $(k,i)$ in a fixed subset $I'\subset\mathbb{Z}\Delta_0$ of representatives of the $G$-orbits,  a basis of $e_{(k,i)}B$ consisting of paths.
Since the existence of this basis is clear the result follows.
\end{proof}

Suppose that $(-,-):B\times B\longrightarrow K$ is a $G$-invariant
graded Nakayama form for $B$. Given a basis $\mathcal{B}$ as in last
lemma, its \emph{(right) dual basis} with respect to $(-,-)$ will be
the basis $\mathcal{B}^*=\bigcup_{(k,i)\in (\mathbb{Z}\Delta
)_0}\mathcal{B}^*e_{\nu (k,i)}$, where $\mathcal{B}^*e_{\nu (k,i)}$,
is the (right) dual basis of $e_{(k,i)}\mathcal{B}$ with respect to
the induced graded bilinear form $(-,-):e_{(k,i)}B\times Be_{\nu
(k,i)}\longrightarrow K$. By the graded condition of this bilinear
form, $\mathcal{B}^*$ consists of homogeneous elements. By the
$G$-invariance of $(-,-)$ and $\mathcal{B}$, we immediately get that
$\mathcal{B}^*$ is $G$-invariant. On what concerns the minimal
projective resolution of $B$ as a bimodule, we will need to fix a
basis $\mathcal{B}$ as given by last lemma and use it and its dual
basis to give the desired resolution.

By a classical argument for unital algebras, also valid here (see, e.g., \cite{BES} or \cite{D2}),
we know that if $B'$ is the original mesh algebra,
i.e., $K\mathbb{Z}\Delta/I$ where $I$ is the ideal generated by $r_{(k,i)}= \sum_{t(a)=(k,i)}
\sigma(a)a$ with $(k,i)\in \mathbb{Z}\Delta_0$, then the initial part of the
minimal projective resolution of $B'$ as a graded $B'$-bimodule has
the following shape

\begin{center}

$Q^{-2}\stackrel{R'}{\longrightarrow}
Q^{-1}\stackrel{\delta}{\longrightarrow} Q^0 \stackrel{u}{\longrightarrow}
B'\longrightarrow 0$

\end{center}

where $Q^0=(\oplus_{(k,i)\in (\mathbb{Z}\Delta )_0}B'e_{(k,i)}\otimes
e_{(k,i)}B')[0]$, $Q^{-1}=(\oplus_{a\in (\mathbb{Z}\Delta
)_1}B'e_{i(a)}\otimes e_{t(a)}B')[-1]$, $Q^{-2}=(\oplus_{(k,i)\in
(\mathbb{Z}\Delta )_0}B'e_{\tau (k,i)}\otimes e_{(k,i)}B')[-2]$, and

\begin{enumerate}

\item $u$ is the multiplication map,

\item $\delta$ is the only homomorphism of $B'$-bimodules such that, for all $a\in(\mathbb{Z}\Delta )_1$,

$$\delta (e_{i(a)}\otimes e_{t(a)})=a\otimes e_{t(a)}-e_{i(a)}\otimes
a;$$

\item $R'$ is the only homomorphism of $B'$-bimodules such that, for all $(k,i)\in
(\mathbb{Z}\Delta )_0$,

$$R'(e_{\tau (k,i)}\otimes e_{(k,i)})=
\sum_{t(a)=(k,i)} [\sigma (a)\otimes
e_{(k,i)}+e_{\tau (k,i)}\otimes a]$$

\end{enumerate}

We will slightly modify this resolution bearing in mind that,
for simplification purposes, we are working with the mesh algebra $B$ given
by the new relations as in Section \ref{subsec. change of relations}.
We point out that this change only accounts for the difference in
the description $R'$. Indeed, considering the canonical
algebra isomorphism $\varphi
=\varphi^{-1}:K\mathbb{Z}\Delta\stackrel{\cong}{\longrightarrow}K\mathbb{Z}\Delta$,
given in Proposition \ref{prop.admissible change of variables}, and
the induced isomorphism of graded algebras
$B\stackrel{\cong}{\longrightarrow}B'$, it is routine to check that, up to
isomorphism, the initial part of the minimal graded projective resolution of $B$ as a $B$-bimodule
is given by

\begin{center}
$Q^{-2}\stackrel{R}{\longrightarrow}Q^{-1}\stackrel{\delta}{\longrightarrow}Q^0\stackrel{u}{\longrightarrow}B\rightarrow
0$,
\end{center}

where $R$ is the only homomorphism of $B$-bimodules such that, for all $(k,i)\in
(\mathbb{Z}\Delta )_0$,

$$R(e_{\tau (k,i)}\otimes e_{(k,i)})=
\sum_{t(a)=(k,i)} (-1)^{s(\sigma (a)a)}[\sigma (a)\otimes
e_{(k,i)}+e_{\tau (k,i)}\otimes a]$$ where $s: \mathbb{Z}\Delta_1\longrightarrow \mathbb{Z}_2$ is the associated
 signature map given in Proposition \ref{prop.admissible
change of variables}, which we assume to be the empty set when
$(\Delta ,G)=(\mathbb{D}_4,\langle\rho\tau^m\rangle)$, and the signature of a path is the sum of
the signatures of its arrows.

Next we identify the elements generating $Ker(R)$.

\begin{prop} \label{prop:deformed projective resolution}
Let $\Delta$ be a Dynkin quiver and let $B$ be the associated mesh algebra.
Denote by $\tau'$ the graded automorphism
of $B$ which acts as $\tau$ on vertices and maps $a\rightsquigarrow
(-1)^{s(a)+s(\tau (a))}a$, for each $a\in (\mathbb{Z}\Delta )_1$. If for
each $(k,i)\in(\mathbb{Z}\Delta )_0$ we consider the homogeneous elements of $Q^{-2}$ given
by  $$\xi_{(k,i)}=\sum_{x\in
e_{(k,i)}\mathcal{B}}(-1)^{\text{deg}(x)}\tau' (x)\otimes x^*,$$  then
  $\oplus_{(k,i)\in\mathbb{Z}\Delta_0}B\xi
_{(k,i)}=\text{Ker}(R)=\oplus_{(k,i)\in\mathbb{Z}\Delta_0}\xi
_{(k,i)}B$.

\end{prop}
\begin{proof}
Let us denote by $h$ the induced isomorphism of graded algebras
$B\stackrel{\cong}{\longrightarrow}B'$ and by $f$ its inverse. We
put $\mathcal{B}'=h(\mathcal{B})$, where $\mathcal{B}$ is the
$G$-invariant basis of $B$ given by the previous lemma. The
mentioned classical arguments also show that the elements
$\xi'_{(k,i)}=\sum_{x\in
e_{(k,i)}\mathcal{B}'}(-1)^{\text{deg}(x)}\tau (x)\otimes x^*$, with
$(k,i)\in(\mathbb{Z}\Delta )_0$, are in $\text{Ker}(R')$ (see [11]).

%Note that
%the argument which proves for unital algebras that the $\xi_{(k,i)}$
%generate $\text{Ker}(R')$ cannot be adapted in a straightforward
%way.

From the equalities $f(\tau (x))=\tau '(f(x))$ and $f(x^*)=f(x)^*$,
for all $x\in\mathcal{B}'$, and the fact that
$f(\mathcal{B}')=\mathcal{B}$ we immediately get that
$\xi_{(k,i)}=f(\xi'_{(k,i)})=\sum_{y\in
e_{(k,i)\mathcal{B}}}(-1)^{\text{deg}(y)}\tau' (y)\otimes y^*$.
Therefore the $\xi_{(k,i)}$ are elements of $L:=Ker(R)$.

If $S_{(m,j)}=Be_{(m,j)}/J(B)e_{(m,j)}$ is the simple graded left
module concentrated in degree zero associated to the vertex $(m,j)$,
then the induced sequence

\begin{center}
$Q^{-2}\otimes_BS_{(m,j)}\stackrel{R\otimes
1}{\longrightarrow}Q^{-1}\otimes_BS_{(m,j)}\stackrel{\delta\otimes
1}{\longrightarrow}Q^0\otimes_BS_{(m,j)}\longrightarrow
S_{(m,j)}\rightarrow 0$
\end{center}
is the initial part of the minimal projective resolution of
$S_{(m,j)}$. It is easy to see that the pushdown functor $\pi_\lambda
:B-Gr\longrightarrow \Lambda-Gr$ preserves and reflects simple
objects. When applied to the last resolution, we then get the minimal
projective resolution of the simple $\Lambda$-module $S_{[(m,j)]}$,
where $\Lambda$ is viewed as the orbit category $B/G$ and where $[(m,j)]$ denotes the $G$-orbit
of $(m,j)$. But we know that $\Omega_\Lambda^3(S_{[(m,j)]})$ is a
simple $\Lambda$-module (see, e.g., \cite{D2}). It follows that
$\Omega_B^3(S_{(m,j)})$ is a graded simple left $B$-module. Moreover we
have an isomorphism $Q^{-2}\otimes_BS_{(m,j)}\cong Be_{\tau
(m,j)}[-2]$ in $B-Gr$. By definition of the Nakayama permutation, we
have that $\text{Soc}_{gr}(Be_{\tau(m,j)})\cong S_{\nu^{-1} \tau
(m,j)}[-c_\Delta +2]$. Then we have an isomorphism
$\Omega_B^3(S_{(m,j)})\cong S_{\nu^{-1} \tau(m,j)}[-c_\Delta ]$, for all
$(m,j)\in\mathbb{Z}\Delta_0$. Considering the decomposition
$B/J(B)=\oplus_{(m,j)\in\mathbb{Z}\Delta_0}S_{(m,j)}$, we then get
that $L/LJ(B)\cong L\otimes_B\frac{B}{J(B)}$ is isomorphic to
$B/J(B)[-c_\Delta ]$ as a graded left $B$-module. Due to the fact
that $J(B)=J^{gr}(B)$ is nilpotent, we know that every left or right
graded $B$-module has a projective cover. By taking projective
covers in $B-Gr$ and bearing in mind that $L$ is projective on the
left and on the right, we then get that $L_B\cong B_B[-c_\Delta ]$.
With a symmetric argument, one also gets that ${}_BL\cong
{}_BB[-c_\Delta]$. In particular, ${}_BL={}_B\Omega_{B^e}^3(B)$
(resp. $L_B=\Omega_{B^e}^3(B)_B$) decomposes as a direct sum of
indecomposable projective graded $B$-modules, all of them with
multiplicity $1$.

Note now that we have equalities $e_{\tau\nu^{-1}(k,i)}\xi
_{\nu^{-1}(k,i)}=\xi_{\nu^{-1}(k,i)}=\xi
_{\nu^{-1}(k,i)}e_{(k,i)}$, for all $(k,i)\in\mathbb{Z}\Delta_0$.
This gives surjective homomorphisms
$Be_{\tau\nu^{-1}(k,i)}[-c_\Delta
]\stackrel{\rho}{\twoheadrightarrow}B\xi_{\nu^{-1}(k,i)}$ and
$e_{(k,i)}B[-c_\Delta ]\stackrel{\lambda}{\twoheadrightarrow}\xi
_{\nu^{-1}(k,i)}B$ of graded left and right $B$-modules given by
right and left multiplication by $\xi_{\nu^{-1}(k,i)}$. But $\rho$
and $\lambda$ do not vanish on
$\text{Soc}_{gr}(Be_{\tau\nu^{-1}(k,i)})$ and
$\text{Soc}_{gr}(e_{(k,i)}B)$, which are simple graded modules,
respectively. It follows that $\rho$ and $\lambda$ are injective
and, hence, they are isomorphisms. We then get that
$N:=\oplus_{(k,i)\in\mathbb{Z}\Delta_0}B\xi
_{\nu^{-1}(k,i)}=\oplus_{(k,i)\in\mathbb{Z}\Delta_0}B\xi_{(k,i)}$
is a graded submodule of $_BL$ isomorphic to ${}_BB\cong {}_BL$ and,
hence, it is injective in $B-Gr$ since this category is Frobenius.
We then get that $N$ is a direct summand of $_BL$ which is
isomorphic to ${}_BL$. Since $\text{End}_{B-Gr}(Be_{(k,i)})\cong K$
for each vertex $(k,i)$, Azumaya's theorem applies (see
\cite{AF}[Theorem 12.6]) and we can conclude that
$L=N=\oplus_{(k,i)\in\mathbb{Z}\Delta_0}B\xi_{(k,i)}$ for
otherwise the decomposition of ${}_BL\cong {}_BB$ as a direct sum of
indecomposables would contain summands with multiplicity $>1$. By a
symmetric argument, we get that
$L=\oplus_{(k,i)\in\mathbb{Z}\Delta_0}\xi_{(k,i)}B$.

\end{proof}

\begin{prop} \label{prop.automorphism mu general}
Let $\Delta$ be a Dynkin quiver, let $G$ be a weakly admissible
group of automorphisms of $B$ and fix a $G$-invariant graded Nakayama
form and its associated Nakayama automorphism $\eta$ (see Theorem
\ref{teor.G-invariant Nakayama automorphism}). Assume that  $X$ is
the $G$-invariant set of arrows given in Proposition
\ref{prop.admissible change of variables}, which we assume to be the
empty set when $(\Delta ,G)=(\mathbb{D}_4,\langle\rho\tau^m\rangle)$ and with
respect to which we calculate the signature of arrows. Finally, let
$\kappa$ and $\vartheta$ be the graded automorphisms of $B$ which
fix the vertices and act on arrows as:

\begin{enumerate}
\item $\kappa (a)=-a$
\item $\vartheta (a)=(-1)^{s(\tau^{-1}(a))+s(a)}a$,
\end{enumerate}
for all $a\in (\mathbb{Z}\Delta )_1$. If we put $\xi= \sum_{(k,i)\in \mathbb{Z}\Delta_0}\xi_{(k,i)}$ and $\mu'
=\kappa\circ\eta\circ\tau^{-1}\circ\vartheta$, then we have an equality

$$b\xi= \xi\mu'(b)$$

\noindent for each element $b\in \bigcup_{(k,i),(m,j)\in \mathbb{Z}\Delta_0}e_{(k,i)}Be_{(m,j)}$. Moreover,
$\mu'\circ g= g\circ \mu'$ for all $g\in G$, and there exists an isomorphism of
graded $B$-bimodules $\Omega_{B^e}^3(B)\cong {}_{\mu'} B_1[-c_\Delta
]$.
\end{prop}
\begin{proof}

First note that, for any of the choices
of the set $X$, the sum $s(\sigma^{-1}(a))+s(\sigma
(a))+s(\tau^{-1}(a))+s(a)$ in $\mathbb{Z}_2$ is constant when $a$
varies on the set of arrows ending at a given vertex $(k,i)\in
(\mathbb{Z}\Delta )_0$. This implies that $\vartheta$ either
preserves the relation $\sum_{t(a)=(k,i)}(-1)^{s(\sigma (a)a)}\sigma
(a)a$ or multiplies it by $-1$. Then $\vartheta$ is a well-defined
automorphism of $B$. Moreover,  the $G$-invariant condition of the
set of arrows $X$ implies that the sum $s(\tau^{-1}(a))+s(a)$ in
$\mathbb{Z}_2$ is $G$-invariant. This shows that  $\vartheta\circ
g=g\circ\vartheta$, for all $g\in G$. This implies that $\mu'\circ
g=g\circ\mu'$ since we have $\kappa\circ g=g\circ\kappa$, for all
$g\in G$.

 All throughout the rest of the proof, a $G$-invariant basis $\mathcal{B}$ of
$B$ consisting of paths in $\mathbb{Z}\Delta$ is fixed, with respect
to which the $\xi_{(k,i)}$ are calculated. We  shall prove that
$a\xi_{\tau^{-1}(t(a))}=\xi_{\tau^{-1}(i(a))}\mu'(a)$, for all
$a\in\mathbb{Z}\Delta_1$. Once this is proved, one easily shows by
induction on $\text{deg}(b)$ that if
$b\in\bigcup_{(k,i),(m,j)\in\mathbb{Z}\Delta_0}e_{(k,i)}Be_{(m,j)}$
is a homogeneous element with respect to the length grading, then
the equality $b\xi_{\tau^{-1}(t(b))}=\xi_{\tau^{-1}(i(b))}\mu' (b)$
holds. It follows from this that the assignment $b\rightsquigarrow
b\xi_{\tau^{-1}(t(b))}$ extends to an isomorphism of $B$-bimodules
${}_1B_{\mu'^{-1}}\stackrel{\cong}{\longrightarrow} L$, which
actually induces an isomorphism of graded $B$-bimodules ${}_{\mu'}
B_1[-c_\Delta ]\cong\Omega_{B^e}^3(B)$, when we view
$\Omega_{B^e}^3(B)$ as a graded sub-bimodule of
$Q^{-2}=(\otimes_{(k,i)\in(\mathbb{Z}\Delta )_0}Be_{\tau
(k,i)}\otimes e_{(k,i)}B)[-2]$.

We have an equality:

\begin{center}
$a\xi_{\tau^{-1}(t(a))}=\sum_{x\in
e_{\tau^{-1}(t(a))}\mathcal{B}}(-1)^{\text{deg}(x)}a\tau'(x)\otimes
x^*.$
\end{center}
But we have $\tau '(\tau^{-1}(a))=(-1)^{s(\tau^{-1}(a))+s(a)}a$, so
that

\begin{center}
$a\tau '(x)=(-1)^{s(\tau^{-1}(a))+s(a)}\tau '(\tau^{-1}(a))\tau
'(x)=(-1)^{s(\tau^{-1}(a))+s(a)}\tau '(\tau^{-1}(a)x)$.
\end{center}
Note that we have $\tau^{-1}(a)x=\sum_{y\in
e_{\tau^{-1}(i(a))}\mathcal{B}}(\tau^{-1}(a)x,y^*)y$ from which we
get the equality

\begin{center}
$a\xi_{\tau^{-1}(t(a))}=\sum_{x\in
e_{\tau^{-1}(t(a))}\mathcal{B}}\sum_{y\in
e_{\tau^{-1}(i(a))}\mathcal{B}}(-1)^{\text{deg}(x)}(-1)^{s(\tau^{-1}(a))+s(a)}(\tau^{-1}(a)x,y^*)\tau
'(y)\otimes x^*$ \hspace*{1cm} (!)
\end{center}

On the other hand, a direct calculation shows that $\mu'
(a)=(-1)^{1+s(\tau^{-1}(a))+s(a)}(\eta\circ\tau^{-1})(a)$, for each
$a\in (\mathbb{Z}\Delta )_1$. Then we have another equality

\begin{center}
$\xi_{\tau^{-1}(i(a))}\mu' (a)=\sum_{y\in
e_{\tau^{-1}(i(a))}\mathcal{B}}(-1)^{\text{deg}(y)}(-1)^{s(\tau^{-1}(a))+s(a)+1}\tau
'(y)\otimes y^* (\eta\circ\tau^{-1})(a)$.
\end{center}
 But we have an equality

\begin{center}
$y^*(\eta\circ\tau^{-1})(a)=\sum_{x\in
e_{\tau^{-1}(t(a))}\mathcal{B}}(x,y^*(\eta\circ\tau^{-1})(a))x^*=\sum_{x\in
e_{\tau^{-1}(t(a))}\mathcal{B}}(xy^*,\eta
(\tau^{-1}(a)))x^*=\sum_{x\in
e_{\tau^{-1}(t(a))}\mathcal{B}}(\tau^{-1}(a),xy^*)x^*=\sum_{x\in
e_{\tau^{-1}(t(a))}\mathcal{B}}(\tau^{-1}(a)x,y^*)x^*$,
\end{center}
using that $(-,-)$ is a graded Nakayama form and that $\eta$ is its
associated Nakayama automorphism. We then get

\begin{center}
$\xi_{\tau^{-1}(i(a))}\mu' (a)=\sum_{y\in
e_{\tau^{-1}(i(a))}\mathcal{B}}\sum_{x\in
e_{\tau^{-1}(t(a))}\mathcal{B}}(-1)^{\text{deg}(y)}(-1)^{s(\tau^{-1}(a))+s(a)+1}(\tau^{-1}(a)x,y^*)\tau
'(y)\otimes x^*$ \hspace*{1cm} (!!)
\end{center}
Bearing in mind that $deg(y)=deg(\tau^{-1}(a)x)=deg(x)+1$
whenever $(\tau^{-1}(a)x,y^*)\neq 0$ we readily see that the second
members of the equalities (!) and (!!) are equal. We then get
$a\xi_{\tau^{-1}(t(a))}=\xi_{\tau^{-1}(i(a))}\mu' (a)$, as desired.

\end{proof}

\begin{rem}\label{rem. mu}
Note that, except when $(\Delta ,G)
=(\mathbb{A}_{2n-1},\langle\rho\tau^m\rangle)$, the automorphism $\vartheta$ of
last proposition is the identity since $X=\tau (X)$.

Also, notice that whenever $(\Delta, G)\neq (\mathbb{A}_{2n},\langle\rho\tau^m\rangle)$, it is always po\-ssible to choose a map $\lambda
:(\mathbb{Z}\Delta )_0\longrightarrow K^*$, taking values in
$\{-1,1\}$, such that $\lambda_{i(a)}=-\lambda_{t(a)}$, for all
$a\in (\mathbb{Z}\Delta )_1$ and $\lambda\circ
g_{\mathbb{Z}\Delta_0}=\lambda$, for all $g\in G$. Indeed, if
$\Delta\neq\mathbb{D}_{n+1}$, we define $\lambda (k,i)=(-1)^i$ for
each $(k,i)\in (\mathbb{Z}\Delta )_0$, and if
$\Delta=\mathbb{D}_{n+1}$, we put $\lambda (k,i)=(-1)^i$, when
$i\neq 0$, and $\lambda (k,0)=-1$. Then, the map
$\psi :B\longrightarrow B$ taking
$b\rightsquigarrow\lambda_{i(b)}b$, for any homogeneous element
$b\in\bigcup_{(k,i),(m,j)\in (\mathbb{Z}\Delta
)_0}e_{(k,i)}Be_{(m,j)}$, clearly defines an isomorphism of $B$-bimodules
$\psi :{}_{\kappa}B_1\longrightarrow {}_{1}B_1$.
This means that, in this case, the automorphism $\eta\tau^{-1}\vartheta$ is another possibility for the twist in $\Omega^3_{B^e}(B)$.

\end{rem}

Crucial for our goals is that what has been done in the last two
propositions is '$G$-invariant', which gives the following
consequence.

\begin{cor} \label{cor. mu for Lambda}
Let $\Delta$ be a Dynkin quiver, $B$ the corresponding  mesh
algebra, $G$ a weakly admissible group of
automorphisms of $\mathbb{Z}\Delta$ and let $\Lambda =B/G$ be the associated $m$-fold mesh
algebra. We denote by $\mu$ the graded automorphism of $B$  given by $\kappa \eta\tau^{-1}$ when $(\Delta,G)= (\mathbb{A}_{2n},\langle \rho\tau^m\rangle)$, $\eta\tau^{-1}\vartheta$ when $(\Delta,G)= (\mathbb{A}_{2n-1},\langle \rho\tau^m\rangle)$ and $\eta\tau^{-1}$ otherwise. Then $\mu$ induces a graded automorphism $\bar{\mu}:\Lambda\longrightarrow\Lambda$ of $\Lambda$ and there is an
isomorphism of graded $\Lambda$-bimodules
$\Omega_{\Lambda^e}^3(\Lambda )\cong
{}_{\bar{\mu}}\Lambda_1[-c_\Delta]$, where $c_\Delta$ is the Coxeter
number.
\end{cor}
\begin{proof}
We fix a $G$-invariant basis of $B$ as in Lemma \ref{lem.G-invarian
basis} and a $G$-invariant graded Nakayama form $(-,-):B\times
B\longrightarrow K$. If we interpret $\Lambda =B/G$ as the orbit
category and $[x]$ denotes the $G$-orbit of $x$, for each
$x\in\bigcup_{(k,i),(m,j)}e_{(k,i)}Be_{(m,j)}$, note that the
$G$-orbits of elements of $\mathcal{B}$ form a basis
$\bar{\mathcal{B}}$ of $\Lambda$ consisting of homogeneous elements
in
$\bigcup_{[(k,i)],[(m,j)]\in\mathbb{Z}\Delta_0/G}e_{[(k,i)]}\Lambda
e_{[(m,j)]}$. Moreover, if $\mathcal{B}^*$ is the right dual basis
of $\mathcal{B}$ with respect $(-,-)$, then
$\bar{\mathcal{B}}^*=\{[x^*]:$ $[x]\in\bar{\mathcal{B}}\}$ is the
right dual basis of $\bar{\mathcal{B}}$ with respect to the graded
Nakayama form $<-,->:\Lambda\times\Lambda\longrightarrow K$ induced
from $(-,-)$ (see Proposition \ref{prop.gr-Frobenius via pushdown functor}).

By taking into account the change of presentation of $\Lambda$ and
\cite{D2}[Section 4], we see that the initial part of the minimal
projective resolution of $\Lambda$ as a  graded $\Lambda$-bimodule
is of the form

\begin{center}
$P^{-2}\stackrel{\bar{R}}{\longrightarrow}P^{-1}\longrightarrow
P^0\longrightarrow\Lambda\rightarrow 0$,
\end{center}
where $P^{-2}=\oplus_{[(k,i)]\in\mathbb{Z}\Delta_0/G}e_{[\tau
(k,i)]}\Lambda e_{[(k,i)]}$ and we have equalities
$\oplus_{[(k,i)]\in\mathbb{Z}\Delta_0/G}\Lambda\bar{\xi}
_{[(k,i)]}=\text{Ker}(\bar{R})=\oplus_{[(k,i)]\in\mathbb{Z}\Delta_0/G}\bar{\xi}_{[(k,i)]}\Lambda$,
where $\bar{\xi}_{[(k,i)]}=\sum_{[x]\in
e_{[(k,i)]}\bar{\mathcal{B}}}(-1)^{deg(x)}[\tau '(x)]\otimes [x^*]$,
for each $[(k,i)]\in\mathbb{Z}\Delta_0/G$.

On the other hand, by Proposition \ref{prop.automorphism mu general},
we get that the automorphism $\mu$ defined as above satisfies that
$\mu\circ g=g\circ\mu$, for all $g\in G$, and hence, induces a graded automorphism
$\bar{\mu}:\Lambda\longrightarrow\Lambda$ which maps
$[x]\rightsquigarrow [\mu (x)]$. In case $\mu
=k\circ\eta\circ\tau^{-1}\circ\vartheta$, we  get the equality
$[b]\bar{\xi}_{[\tau'(i(b))]}=\bar{\xi}_{[\tau^{-1}(i(b))]}\bar{\mu}([b])$,
for each homogeneous element
$[b]\in\bigcup_{[(k,i)],[(m,j)]\in\mathbb{Z}\Delta_0/G}e_{[(k,i)]}\Lambda
e_{[(m,j)]}$, from the corresponding equality in the proof of the
previous proposition, just by replacing the homogeneous elements of
$B$ by their orbits. We leave to the reader the routine
verification. It then follows that the assignment
$[b]\rightsquigarrow [b]\bar{\xi}_{[\tau'(i(b))]}$ gives an
isomorphism of graded $\Lambda$-bimodules
$\Omega_{\Lambda^e}^3(\Lambda )\cong
{}_1\Lambda_{\bar{\mu}^{-1}}[-c_\Delta ]\cong
{}_{\bar{\mu}}\Lambda_1[-c_\Delta ]$.

The different descriptions for the automorphism $\mu$ of $\Lambda$ given in the
statement are valid using Remark \ref{rem. mu}.

%When $(\Delta ,G)\neq (\mathbb{A}_{2n},<\rho\tau^m>)$ and we take
%$\mu '=\eta\circ\tau^{-1}\circ\vartheta$, we have seen in the proof
%of the last proposition that there is a map $\lambda
%:\mathbb{Z}\Delta_0\longrightarrow K^*$ such that $\lambda\circ
%g_{\mathbb{Z}\Delta_0}=\lambda$, for all $g\in G$, and such that
%$\mu (a)=\lambda_{i(a)}^{-1}\lambda_{t(a)}\mu'(a)$, for all
%homogeneous elements $a\in\bigcup_{(k,i),(m,j)}e_{(k,i)}Be_{(m,j)}$.
%We then get from lemma \ref{lem.criterion of inner automorphism}
%that $\bar{\mu}^{-1}\bar{\mu}'$ is an inner automorphism of
%$\Lambda$, so that also $\Omega_{\Lambda^e}^3(\Lambda )\cong
%{}_{\bar{\mu}'}\Lambda_1$.

\end{proof}

\subsection{The period of an $m$-fold mesh algebra}

This section is devoted to compute the $\Omega$-\emph{period} of an
$m$-fold mesh algebra $\Lambda$. That is, the smallest of the
positive integers $r$ such that $\Omega_{\Lambda^e}^r(\Lambda )$ is
isomorphic to $\Lambda$ as a $\Lambda$-bimodule. We need to separate
the case of Loewy length $2$, for which the period has
already been computed (see, e.g., \cite{EH}), from the rest. We  point out that the
only connected self-injective algebras of Loewy length $2$ are the $m$-fold mesh algebras $\mathbb{A}_2^{(m)}$ and
$\mathbb{L}_1^{(m)}$ and that  each of these is precisely the path algebra of
a cyclic quiver modulo paths of length $2$.

\begin{prop} \label{prop.periods of Nakayama algebras}
Let $\Lambda$ be a connected selfinjective algebra of Loewy length
$2$. The following assertions hold:

\begin{enumerate}
\item If $\text{char}(K)=2$ or
$\Lambda =\mathbb{A}_2^{(m)}$, i.e.  $|Q_0|$ is even,  then the
period of $\Lambda$ is $|Q_0|$. \item If $\text{char}(K)\neq 2$ and
 $\Lambda =\mathbb{L}_1^{(m)}$, i.e. $|Q_0|$ is odd,  then the period
of $\Lambda$ is $2|Q_0|$.
\end{enumerate}
\end{prop}
%\begin{proof}
%By proposition \ref{prop.2-nilpotent selfinjective algebras}, we
%know that $\Omega_{\Lambda^e}(\Lambda )\cong
%{}_{\bar{\mu}}\Lambda_1$, where $\bar{\mu}$ is the automorphism
%which acts on vertices as the $n$-cycle $(12...n)$ and on arrows by
%$a_i\rightsquigarrow -a_{i+1}$. The period of $\Lambda$ is then the
%smallest of the integers $r>0$ such that  $\bar{\mu}^r$ is inner.
%But since inner automorphisms fix the vertices each such $r$ is
%multiple of $n=|Q_0|$. When $\text{char}(K)=2$ or $n$ is even, we
%have that $\bar{\mu}^n$ fixes the vertices an maps
%$a_i\rightsquigarrow (-1)^na_i=a_i$, for each $i\in Q_0$. Then
%$\bar{\mu}^n=id_\Lambda$ and the period of $\Lambda$ is $n$.
%However, when $\text{char}(K)\neq 2$ and $n$ is odd, we have that
%$\bar{\mu}^n$ is not inner (see proposition \ref{prop.2-nilpotent
%selfinjective algebras}) but $\bar{\mu}^{2n}=id_\Lambda$. It follows
%that the period of $\Lambda$ is $2n$ in this case.
%\end{proof}

For the remaining cases we will need the following:

\begin{lema} \label{lem.commutativity of eta y tau}
Let $\Lambda=B/G$ be an $m$ fold mesh algebra, with
$\Delta\neq\mathbb{A}_{1},\mathbb{A}_{2}$, and let $r\geq 0$ be an
integer. The equality  $\text{dim}(\Omega_{\Lambda^e}^r(\Lambda
))=\text{dim}(\Lambda)$ holds if, and only if, $r\in 3\mathbb{Z}$.
\end{lema}
\begin{proof}
 The `if' part is well-known. For the `only if' part, note that we
have the following formulas for the dimensions of the syzygies:

\begin{enumerate}
\item $\text{dim}(\Omega_{\Lambda^e}^{r}(\Lambda
))=\text{dim}(\oplus_{i\in Q_0}\Lambda e_i\otimes
e_i\Lambda)-\text{dim}(\Lambda)=\sum_{i\in Q_0}\text{dim}(\Lambda
e_i) (\text{dim}(e_i\Lambda )-1)$, whenever $r\equiv 1\text{ (mod
3)}$

\item $\text{dim}(\Omega_{\Lambda^e}^{r}(\Lambda
))=\text{dim}(\oplus_{i\in Q_0}\Lambda e_{\tau (i)}\otimes
e_i\Lambda)-\text{dim}({}_{\bar{\mu}}\Lambda_1)=\sum_{i\in
Q_0}\text{dim}(\Lambda e_{\tau (i)}) (\text{dim}(e_i\Lambda
)-1)=\sum_{i\in Q_0}\text{dim}(\Lambda e_i) (\text{dim}(e_i\Lambda
)-1)$, whenever $r\equiv 2\text{ (mod 3)}$
\end{enumerate}
For $r\equiv 1,2\text{ (mod 3)}$ the equality
$\text{dim}(\Omega_{\Lambda^e}^{r}(\Lambda ))=\text{dim}(\Lambda )$
can occur if, and only if, $\text{dim}(e_i\Lambda )=2$, for each
$i\in Q_0$. But this can only happen when the Loewy length is $2$,
which is discarded.

\end{proof}

By the previous lemma,  we know that
$\text{dim}(\Omega_{\Lambda^e}^r(\Lambda ))\neq\text{dim}(\Lambda )$
whenever $r\not\in 3\mathbb{Z}$.  Due to the existence of an
automorphism $\bar{\mu}$ of $\Lambda$ satisfying that
$\Omega_{\Lambda^e}^3(\Lambda)\cong _{\bar{\mu}}\Lambda _1$ as
$\Lambda$-bimodules (see Proposition \ref{prop.automorphism mu
general}), in order to calculate the ($\Omega$-)period of $\Lambda$,
we just need to control the positive integers $r$ such that
$\bar{\mu}^r$ is inner. For the sake of simplicity, we shall divide
the problem into two steps. We begin by identifying the smallest
$u\in \mathbb{N}$ such that $(\bar{\nu}\circ
\bar{\tau}^{-1})^u=Id_{\Lambda}$ , that is, the smallest $u$ such
that $\bar{\mu}^u$ acts as the identity on vertices. This is the
content of the next result.

\begin{lema} \label{lem.candidatos periodo de Lambda}
Let $\Lambda=K(\mathbb{Z}\Delta)/\langle\varphi\rangle$ be an $m$-fold mesh algebra
of extended type $(\Delta ,m,t)$ and let us put $u:=\text{min}\{r\in
\mathbb{Z}^+ \mid (\bar{\nu}\circ \bar{\tau}^{-1})^r= Id_{\Lambda}
\}$. The following assertions hold:

\begin{enumerate}
\item If $t=1$ then:

\begin{enumerate}
\item $u=\frac{2m}{\text{gcd}(m,c_\Delta )}$,  whenever $\Delta$ is
$\mathbb{A}_r$, $\mathbb{D}_{2r-1}$ or $\mathbb{E}_{6}$;
\item $u=\frac{m}{\text{gcd}(m,\frac{c_\Delta}{2})}$, whenever
$\Delta$ is $\mathbb{D}_{2r}$, $\mathbb{E}_{7}$ or $\mathbb{E}_{8}$.
\end{enumerate}
\item If $t=2$ then:

\begin{enumerate}
\item $u=\frac{2m}{\text{gcd}(2m,m+\frac{c_\Delta}{2})}$,  whenever $\Delta$ is
$\mathbb{A}_{2n-1}$, $\mathbb{D}_{2r-1}$ or $\mathbb{E}_{6}$;
\item $u=\frac{2m}{\text{gcd}(2m,\frac{c_\Delta}{2})}$, whenever
$\Delta$ is $\mathbb{D}_{2r}$;
\item $u=\frac{2m-1}{\text{gcd}(2m-1,2n+1)}$, when $\Delta =\mathbb{A}_{2n}$.
\end{enumerate}
\item If $t=3$ (hence $\Lambda =K(\mathbb{Z}\mathbb{D}_4)/\langle\rho\tau^m\rangle$), then $u=m$.
\end{enumerate}

\end{lema}

\begin{proof}
The argument that we used for $\nu$ in the first paragraph of the
proof of Theorem \ref{teor.weakly symmetric m-fold algebras} is also
valid for $(\nu\circ\tau^{-1})^r$. Then
$(\bar{\nu}\circ\bar{\tau}^{-1})^r=id_\Lambda$ if, and only if,
$(\nu\circ\tau^{-1})^r\in G$.

  When $\Delta$ is $\mathbb{A}_{2n-1}$, $\mathbb{D}_{n+1}$,
with $n+1$ odd, or $\mathbb{E}_{6}$, the Nakayama permutation is
$\nu =\rho\tau^{1-n}$, where $n=\frac{c_\Delta}{2}$. Then
$(\nu\tau^{-1})^r=\rho^r\tau^{-nr}$. If $t=1$ this automorphism is
in $G$ if, and only if, $r=2r'$ is even and
$\tau^{-nr}=\tau^{-2nr'}$ is equal to $(\tau^m)^v=\tau^{mv}$, for
some $v\in\mathbb{Z}$. This happens exactly when $2nr'\in
m\mathbb{Z}$ and the smallest $r'$ satisfying this is
$u'=\frac{m}{\text{gcd}(m,2n)}$. We then
get that
$u=2u'=\frac{2m}{\text{gcd}(m,2n)}=\frac{2m}{\text{gcd}(m,c_\Delta
)}$. Suppose that $t=2$. Then $(\nu\tau^{-1})^r=\rho^r\tau^{-nr}$ is
in $G=\langle\rho\tau^m\rangle$ if, and only if, there is $v\in\mathbb{Z}$ such
that $v\equiv r\text{ (mod 2)}$ and $-nr=mv$. This is equivalent to
saying that there is $k\in\mathbb{Z}$ such that $-nr=m(r+2k)$ or,
equivalently, that $(m+n)r\in 2m\mathbb{Z}$. The smallest $r$
satisfying this property is
$u=\frac{2m}{\text{gcd}(2m,m+n)}=\frac{2m}{\text{gcd}(2m,m+\frac{c_\Delta}{2})}$.
This proves 1.a, except for $\Delta =\mathbb{A}_{2n}$,  and 2.a.

Suppose next that $\Delta$ is $\mathbb{D}_{n+1}$, with $n+1$ even,
$\mathbb{E}_{7}$ or $\mathbb{E}_{8}$. Then $\nu =\tau^{1-n}$, where
$n=\frac{c_\Delta}{2}$,  so that $(\nu\tau^{-1})^r=\tau^{-nr}$. When
$t=1$, this automorphism is in $G=\langle\tau^m\rangle$ if, and only if, $nr\in
m\mathbb{Z}$. The smallest $r$ satisfying this property is
$u=\frac{m}{\text{gcd}(m,n)}=\frac{m}{\text{gcd}(m,\frac{c_\Delta}{2})}$.
On the other hand, if $t=2$ then $\tau^{-nr}$ is in $G=\langle\rho\tau^m\rangle$
if, and only if, there is $v=2v'\in 2\mathbb{Z}$ such that
$-nr=mv=2mv'$. The smallest $r$ satisfying this property is
$u=\frac{2m}{\text{gcd}(2m,n)}=\frac{2m}{\text{gcd}(2m,\frac{c_\Delta}{2})}$.
This proves 1.b and 2.b.

Let us now take $\Delta =\mathbb{A}_{2n}$. Then $\nu =\rho\tau^{1-n}$,
so that $(\nu\tau^{-1})^r=\rho^r\tau^{-nr}$. If $t=1$, this
automorphism is in $G=\langle\tau^m\rangle$ if, and only if, $r=2r'$ is even and
there exists $v\in\mathbb{Z}$ such that
$\rho^{2r'}\tau^{-2nr'}=\tau^{-(2n+1)r'}$ is equal to $\tau^{mv}$.
This is equivalent to saying that $(2n+1)r'\in m\mathbb{Z}$. The
smallest $r'$ satisfying this property is
$u'=\frac{m}{\text{gcd}(m,2n+1)}$. We then get
$u=\frac{2m}{\text{gcd}(m,2n+1)}=\frac{2m}{\text{gcd}(m,c_\Delta)}$,
which completes 1.a. When $t=2$, the automorphism $\rho^r\tau^{-nr}$
is in $G=\langle\rho\tau^m\rangle$ if, and only if, there exists
$v\in\mathbb{Z}$ such that $v\equiv r\text{ (mod 2)}$ and
$\rho^r\tau^{-nr}=\rho^v\tau^{mv}$. This is in turn equivalent to
the existence of an integer $k$ such that
$\rho^r\tau^{-nr}=\rho^{r+2k}\tau^{m(r+2k)}=\rho^r\tau^{-k}\tau^{mr+2mk}$.
That is, if and only if $-nr=(2m-1)k+mr$. This happens exactly when
$(m+n)r\in (2m-1)\mathbb{Z}$. The smallest $r$ satisfying this
property is $u=\frac{2m-1}{\text{gcd}(m+n,2m-1)}$. But we have that
$\text{gcd}(m+n,2m-1)=\text{gcd}(2m-1,2n+1)$, so that 2.c holds.

Finally, if $t=3$, and hence $\Delta =\mathbb{D}_4$, then
$\nu=\tau^{-2}$, so that $(\nu\tau^{-1})^r=\tau^{-3r}$ is in
$G=\langle\rho\tau^m\rangle$ if, and only if, there is $v=3v'\in 3\mathbb{Z}$
such that $-3r=3mv'$. This happens exactly when $r\in m\mathbb{Z}$,
which implies that  $u=m$ in this case.
\end{proof}

\begin{lema} \label{lem.crucial for the periods in Bn and Ln}
Let $\Lambda$ be an $m$-fold algebra of extended type
$(\mathbb{A}_r,m,2)$ and let $T$ be the subgroup of $\mathbb{Z}$
consisting of the integers $s$ such that $\bar{\mu}^s$ and
$(\bar{\nu}\circ\bar{\tau}^{-1})^s$ are equal, up to composition by
an inner automorphism of $\Lambda$. Then  $T=2\mathbb{Z}$, when
$\text{char}(K)\neq 2$, and $T=\mathbb{Z}$, when $\text{char}(K)=
2$.
\end{lema}
\begin{proof}
We can take $\eta =\nu$ in this case (see Theorem \ref{teor.G-invariant Nakayama automorphism}). By Corollary \ref{cor. mu for Lambda}, we can assume that there is  a $G$-invariant involutive automorphism  $h$ of $B$ which fixes the vertices and acts on arrows as $h(a)=(-1)^{u(a)}a$, for each $a\in\mathbb{Z}\Delta_1$, where $u(a)\in\mathbb{Z}_2$, such that either $\mu =h\circ\nu\circ\tau^{-1}$ or $\mu =\nu\circ\tau^{-1}\circ h$. Indeed the first situation, with $h=\kappa$, appears when $\Delta=\mathbb{A}_{2n}$, and the second situation appears, with $h=\vartheta$, when $\Delta=\mathbb{A}_{2n-1}$.
Then, using Proposition \ref{prop.tau nu eta in center}, we get that $T$ consists of the $s\in\mathbb{Z}$ such that $\bar{h}^s\in\text{Inn}(\Lambda )$, where $\bar{h}$ is the automorphism of $\Lambda =B/G$ induced by $h$. The involutive condition of $h$ implies that $2\in T$ and, when $\text{char}(K)=2$, also $1\in T$. The proof hence reduces to check that $\bar{h}\not\in\text{Inn}(\Lambda )$, when $\text{char}(K)\neq 2$.
Proceeding as in the proof of Proposition \ref{prop.tau nu eta in center}, we have a map $\lambda :\mathbb{Z}\Delta_0\longrightarrow K^*$ such that $\lambda_{i(a)}^{-1}\lambda_{t(a)}a=h(a)$ (equivalently, $\lambda_{i(a)}^{-1}\lambda_{t(a)}=(-1)^{u(a)}$), for all $a\in\mathbb{Z}\Delta_1$. Due to Lemma \ref{lem.criterion of inner automorphism},  we just need to prove $\lambda\circ \rho\tau^m_{| \mathbb{Z}\Delta_0}\neq\lambda$, where $G=<\rho\tau^m>$. We study the possible  situations:

\vspace*{0.3cm}

a) \emph{When $\Delta =\mathbb{A}_{2n}$}: Here we have $h=\kappa$, so  that $\lambda_{t(a)}=-\lambda_{i(a)}$, for all $a\in\mathbb{Z}\Delta_1$. It follows that $\lambda_{\tau(k,i)}=\lambda_{(k,i)}=-\lambda_{(k,i+1)}$,
for all $(k,i)\in\mathbb{Z}\Delta_0$ with $i<2n$. But then $\lambda_{\rho\tau^m(k,n)}=\lambda_{\rho (k-m,n)}=\lambda_{(k-m,n+1)}=\lambda_{(k,n+1)}=-\lambda_{(k,n)}$.

\vspace*{0.3cm}

b) \emph{When $\Delta =\mathbb{A}_{2n-1}$ and $m$ is odd}: (see Proposition \ref{prop.admissible change of variables}): Here $h=\vartheta$, so that $\lambda_{t(a)}=(-1)^{s(\tau^{-1}(a))+s(a)}\lambda_{i(a)}$, for all $a\in\mathbb{Z}\Delta_1$. But $(-1)^{s(\tau^{-1}(a))+s(a)}$ is equal to $-1$, when $a$ is an upward arrow in the 'south' hemisphere or a downward arrow in the 'north' hemisphere, and it is equal to $1$ otherwise. We then get that $\lambda_{(k,n)}=\lambda_{(k,n+j)}$ and $\lambda_{(k,n)}=\lambda_{(k+j,n-j)}$, for all $j=0,1,...,n-1$, so that $\lambda_{\rho (k,i)}=\lambda_{(k,i)}$, for all  $(k,i)\in\mathbb{Z}\Delta_0$. It also follows that $\lambda_{\tau (k,i)}=-\lambda_{(k,i)}$, for all $(k,i)\in\mathbb{Z}\Delta_0$. We then get that $\lambda_{\rho\tau^m(k,i)}=\lambda_{\tau^m(k,i)}=(-1)^m\lambda_{(k,i)}=-\lambda_{(k,i)}$.

\vspace*{0.3cm}

c) \emph{When $\Delta =\mathbb{A}_{2n-1}$ and $m$ is even}: We again have $\lambda_{t(a)}=(-1)^{s(\tau^{-1}(a))+s(a)}\lambda_{i(a)}$, for all $a\in\mathbb{Z}\Delta_1$. We now consider the arrow $a:(0,n)\rightarrow (0,n+1)$. Using Proposition \ref{prop.admissible change of variables}, we see that we have the following formulas for the integers $k\in\{0,1,...,m-1\}$:

\begin{enumerate}
\item $s(\sigma^{-2k}(a))+s(\tau^{-1}(\sigma^{-2k}(a)))$ is equal to $1$ exactly when $k\neq m-1$;
\item $s(\sigma^{-2k-1}(a))+s(\tau^{-1}(\sigma^{-2k-1}(a)))=1$ exactly when $k=m-2,m-1$.
\end{enumerate}
Considering the path
$$(0,n)\stackrel{a}{\longrightarrow} (0,n+1)\stackrel{\sigma^{-1}(a)}{\longrightarrow} (1,n)\longrightarrow ...\longrightarrow (m-1,n)\stackrel{\sigma^{-2(m-1)}(a)}{\longrightarrow} (m-1,n+1)$$

$$\stackrel{\sigma^{-2(m-1)-1}(a)}{\longrightarrow}(m,n)=\rho\tau^{-m}(0,n)$$ and arguing as in the proof of Proposition \ref{prop.tau nu eta in center}, we conclude that $\lambda_{(0,n)}=(-1)^{m+1}\lambda_{\rho\tau^{-m}(0,n)}=-\lambda_{(0,n)}$. The proof is finished since $\rho\tau^{-m}=(\rho\tau^m)^{-1}\in G$.

\end{proof}

We are now ready to describe explicitly the period of any $m$-fold
mesh algebra.

\begin{teor} \label{teor.period of Lambda}
Let $\Lambda$ be an $m$-fold mesh algebra of extended type $(\Delta
,m,t)$, where $\Delta\neq\mathbb{A}_{1},\mathbb{A}_{2}$,  let $\pi
=\pi (\Lambda)$ denote the period of $\Lambda$ and, for each
positive integer $k$, denote by $O_2(k)$ the biggest of the natural
numbers $r$ such that $2^r$ divides $k$. If $\text{char}(K)=2$ then
$\pi =3u$, where $u$ is the positive integer of Lemma
\ref{lem.candidatos periodo de Lambda}. When $\text{char}(K)\neq 2$,
the period of $\Lambda$ is given as follows:

\begin{enumerate}
\item If $t=1$ then:

\begin{enumerate}
\item When $\Delta$ is $\mathbb{A}_{r}$, $\mathbb{D}_{2r-1}$ or
$\mathbb{E}_{6}$, the period is $\pi
=\frac{6m}{\text{gcd}(m,c_\Delta )}$.
\item When $\Delta$ is  $\mathbb{D}_{2r}$,
$\mathbb{E}_{7}$ or $\mathbb{E}_{8}$,  the period is  $\pi
=\frac{3m}{\text{gcd}(m,\frac{c_\Delta}{2} )}$, when $m$ is even,
and $\pi =\frac{6m}{\text{gcd}(m,\frac{c_\Delta}{2} )}$, when $m$ is
odd.
\end{enumerate}
\item If $t=2$ then:

\begin{enumerate}
\item When $\Delta$ is $\mathbb{A}_{2n-1}$, $\mathbb{D}_{2r-1}$ or
$\mathbb{E}_{6}$, the period is
$\frac{6m}{\text{gcd}(2m,m+\frac{c_\Delta}{2} )}$, when $O_2(m)\neq
O_2(\frac{c_\Delta}{2})$, and $\pi
=\frac{12m}{\text{gcd}(2m,m+\frac{c_\Delta}{2} )}$ otherwise.
\item When $\Delta =\mathbb{D}_{2r}$, the period is $\frac{6m}{\text{gcd}(2m,\frac{c_\Delta}{2}
)}=\frac{6m}{\text{gcd}(2m,2r-1)}$.
\item When $\Delta =\mathbb{A}_{2n}$, i.e. $\Lambda =\mathbb{L}_n^{(m)}$,  the period
is $\pi =\frac{6(2m-1)}{\text{gcd}(2m-1,2n+1)}$
\end{enumerate}
\item If $t=3$ then $\pi =3m$, when $m$ is even, and $6m$, when $m$
is odd.
\end{enumerate}
\end{teor}
\begin{proof}
Let $u>0$ be the integer of Lemma \ref{lem.candidatos periodo de
Lambda}. Then $u\mathbb{Z}$ consists of the integers $r$ such that
$\bar{\nu}^r=\bar{\tau}^r$, or equivalently
$(\bar{\nu}\circ\bar{\tau}^{-1})^r=id_\Lambda$, as automorphisms of
$\Lambda$. If $\pi$ is the period of $\Lambda$, then, by Lemma
\ref{lem.commutativity of eta y tau},   we know that $\pi =3v$,
where $v$ is the smallest of the positive integers $s$ such that
$\bar{\mu}^s\in\text{Inn}(\Lambda )$. These integers $s$ obviously
form a subgroup $S=S(\Delta ,m,t)$ of $\mathbb{Z}$, and then
$v\mathbb{Z}=S$. This subgroup is the intersection of $u\mathbb{Z}$
with the subgroup $T$ consisting of the integers $r$ such that
$\bar{\mu}^r$ and $(\bar{\nu}\circ\bar{\tau}^{-1})^r$ are equal, up
to composition by an inner automorphism of $\Lambda$. When $(\Delta
,m,t)=(\mathbb{A}_{r},m,2)$, by Lemma \ref{lem.crucial for the
periods in Bn and Ln}, we get that $v\mathbb{Z}=u\mathbb{Z}\cap
2\mathbb{Z}$, when $\text{char}(K)\neq 2$, and
$v\mathbb{Z}=u\mathbb{Z}\cap\mathbb{Z}=u\mathbb{Z}$, when
$\text{char}(K)=2$. This automatically gives 2.c and the part of
characteristic $2$ in this case. We claim that it also gives the
formula in 2.a for $\Delta =\mathbb{A}_{2n-1}$. Indeed, by Lemma
\ref{lem.candidatos periodo de Lambda}, we have
$u=\frac{2m}{\text{gcd}(2m,m+n)}$ in this case. But the biggest
power of $2$ which divides $2m$ is a divisor of $\text{gcd}(2m,m+n)$
if, and only if, $O_2(m)=O_2(n)$. Then the equality 2.a for
$\mathbb{A}_{2n-1}$ follows automatically.

When  $(\Delta ,m,t)\neq(\mathbb{A}_r,m,2)$, by Proposition
\ref{prop.automorphism mu general} and the subsequent remark, we can
take $\bar{\mu}=\bar{\eta}\circ\bar{\tau}^{-1}$. Then Proposition \ref{prop.tau nu eta in center} implies that $S$ consists of
the integers $s$ such that $\bar{\eta}^s$ and $\bar{\tau}^s$ are
equal, up to composition by an inner automorphism of $\Lambda$. We
then get that $S=u\mathbb{Z}\cap H(\Delta ,m,t)$ (see Proposition
\ref{prop.subgroup of Z associated}). Therefore, Proposition
\ref{prop.subgroup of Z associated} tells us that $v=u$, when either
$H(\Delta ,m,t)=\mathbb{Z}$ or $u$ is even, and $v=2u$ otherwise. This completes the assertion for
characteristic $2$.  We
next check that, together with Proposition
\ref{prop.subgroup of Z associated}, it also gives all the remaining formulas
of the theorem.

For the quivers $\Delta$ in 1.a we always have that $H(\Delta
,m,t)=\mathbb{Z}$ when $\Delta =\mathbb{A}_r$, and also in the other
two cases when $m$ is even. But if $m$ is odd, then automatically
$u=\frac{2m}{\text{gcd}(m,c_\Delta )}$ is even.

For the quivers in 1.b, we always have that $n=\frac{c_\Delta}{2}$
is odd. Therefore $u$ is even exactly when $m$ is even. The reader should note
that the case when $t=1$ is also covered in Section 5 of \cite{D}.

For the quivers in 2.a which are not $\mathbb{A}_{2n-1}$, we have
that $H(\Delta ,m,t)=\mathbb{Z}$ exactly when $m$ is odd. But
$\frac{c_\Delta}{2}$ is even, so that $O_2(m)\neq
O_2(\frac{c_\Delta}{2})$ in that case. As we did above in the case
$(\Delta ,m,t)=(\mathbb{A}_{2n-1},m,2)$, in case $m$ even, we have
that $u=\frac{2m}{\text{gcd}(2m,m+\frac{c_\Delta}{2})}$ is odd if,
and only if, $O_2(m)=O_2(\frac{c_\Delta}{2})$. Then the formula in
2.a is true also for the cases different from $\mathbb{A}_{2n-1}$.

For 2.b, we have that $\frac{c_\Delta}{2}$ is odd, which implies
that $u$ is always even, and then the formula in 2.b is true.

Finally, when $t=3$, we have that $H(\Delta ,m,t)=\mathbb{Z}$,
exactly when $m$ is even, and then the formula in 3) is automatic.

\end{proof}

\subsection{Inner and stably inner automorphisms}

From now on, in this paper, we will denote by ${\Lambda}-mod$ the category of finite
 dimensional $\Lambda$-modules and, given an automorphism $\phi$ of an algebra $\Lambda$,
the functor $:{}_{\phi}\Lambda_{1}\otimes_{\Lambda}-: \Lambda-mod\longrightarrow
\Lambda-mod$ will be denoted by ${}_{\phi}(-)$.
Recall from \cite{IV} that an automorphism $\sigma$ of $\Lambda$ is
\emph{stably inner} if the functor ${}_\sigma (-):\Lambda-\underline{\text{mod}}\longrightarrow\Lambda-\underline{\text{mod}}$ is
naturally isomorphic to the identity functor. In particular, each
inner automorphism is stably inner.

\begin{lema} \label{lem.inner-versus-stably inner}
Let $\Lambda =KQ/I$ be a finite dimensional selfinjective algebra,
where $I$ is an admissible ideal of $KQ$ which is homogeneous with respect to the grading
by path length, and consider the induced grading on $\Lambda$.
Suppose that the Loewy length of $\Lambda$ is greater or equal than
$4$. A graded automorphism of $\Lambda$ is inner if, and only if, it
is stably inner.
\end{lema}
\begin{proof}
Let $\varphi$ be a stably inner graded automorphism of $\Lambda$.
Let $\mathit{l}$ be the Loewy length of $\Lambda$. If $J=J(\Lambda
)=J^{gr}(\Lambda )$ is the Jacobson radical and
$\text{Soc}^n(\Lambda )=\text{Soc}^n_{gr}(\Lambda )$ is the
$n$-socle of $\Lambda$ (i.e. $\text{Soc}^0(\Lambda )=0$  and
$\text{Soc}^{n+1}(\Lambda )/\text{Soc}^n(\Lambda )$ is the socle of
$\Lambda /\text{Soc}^n(\Lambda )$, for all $n\geq 0$), then we have
$J^n=\text{Soc}^{l-n}(\Lambda )=\oplus_{k\geq n}\Lambda_k$, for all
$n\geq 0$.

We then have $\text{Soc}^2(\Lambda )\subseteq J^2$ since
$\mathit{l}\geq 4$. By Corollary 2.11 of \cite{IV}, we have a map
$\lambda :Q_0\longrightarrow K^*$ such that $\varphi
(a)-\lambda_{i(a)}^{-1}\lambda_{t(a)}a\in J(A)^2$, for all $a\in
Q_1$. If we define $\chi_\lambda :\Lambda\longrightarrow\Lambda$ as
in the proof of Lemma \ref{lem.criterion of inner automorphism}, we
get that $\chi_\lambda$ is an inner automorphism of $\Lambda$ such
that $(\varphi\circ\chi_\lambda ^{-1})(a)-a\in J(A)^2$, for all
$a\in Q_1$. But $\varphi\circ\chi_\lambda ^{-1}$ is a graded
automorphism since so are $\varphi$ and $\chi_\lambda$. It then
follows that $(\varphi\circ\chi_\lambda )(a)=a$, for all $a\in Q_1$,
which implies that $\varphi\circ\chi_\lambda =id_\Lambda$, and so
$\varphi =\chi_\lambda$ is inner.
\end{proof}

 Recall that $\Lambda$ is a \emph{Nakayama algebra}
if each left or right indecomposable projective $\Lambda$-module is
uniserial. We will need the following properties of self-injective
algebras of Loewy length $2$.

\begin{prop} \label{prop.2-nilpotent selfinjective algebras}
Let $\Lambda=KQ/KQ_{\geq 2}$ be  selfinjective algebra such  that
$J(\Lambda )^2=0$ and suppose that $\Lambda$ does not have any
semisimple summand as an algebra. The following assertions hold:

\begin{enumerate}
\item $\Lambda$ is a Nakayama algebra and $Q$ is a
disjoint union of oriented cycles, with relations all the paths of
length $2$.

\item $\Lambda$ is a finite direct product of $m$-fold mesh algebras of
Dynkin graph $\Delta=\mathbb{A}_2$.

\item A graded automorphism $\varphi$ of $\Lambda$  is stably inner if, and only if, it fixes the
vertices.

\item $\varphi$ is inner if, and only if, it  fixes the
vertices and there exists an acyclic character $\chi
:Q_1\longrightarrow K^*$ of $Q$ such that $\varphi (a)=\chi (a)a$, for each arrow $a\in Q_1$.

\item If the quiver $Q$ is connected with $n$ vertices (whence an oriented cycle with $Q_0=\mathbb{Z}_n$), then
$\Omega_{\Lambda^e}(\Lambda )$ is isomorphic to the
$\Lambda$-bimodule ${}_{\bar{\mu}}\Lambda_1$, where $\bar{\mu}$ is
the
 automorphism acting on vertices as the $n$-cycle $(12...n)$ and
 on arrows as $\bar{\mu}(a_i)=-a_{i+1}$, where $a_i:i\rightarrow
 i+1$ for each $i\in\mathbb{Z}_n$.
\end{enumerate}
\end{prop}
\begin{proof}
Assertion 1  is folklore. But
$\mathbb{A}_2^{(m)}=\mathbb{Z}\mathbb{A}_2/\langle\tau^m\rangle$ is the
connected Nakayama algebra of Loewy length $2$ with $2m$ vertices
while $\mathbb{L}_1^{(m)}=\mathbb{Z}\mathbb{A}_2/\langle\rho\tau^m\rangle$ is
the one with $2m-1$ vertices. Then assertion 2 is clear.

The only indecomposable objects in the stable category $\Lambda
-\underline{\text{mod}}$ are the simple modules, all of which have
endomorphism algebra isomorphic to $K$. It follows that each
additive self-equivalence $F:\Lambda
-\underline{\text{mod}}\stackrel{\cong}{\longrightarrow}\Lambda
-\underline{\text{mod}}$ such that $F(S)\cong S$, for each simple module
$S$, is naturally isomorphic to the identity. Since each
automorphism $\varphi$ of $\Lambda$ induces the self-equivalence
$F={}_\varphi (-)$, assertion 3 is clear.

Assertion 4 follows directly from \cite{GA-S}[Theorem 12], taking
into account that the only inner graded automorphism induced by an
element $1-x$, with $x\in J$, is the identity (see the proof of
Lemma \ref{lem.criterion of inner automorphism}).

Suppose now that $Q$ is connected and has $n$ vertices, so that
$\Lambda$ is either an $m$-fold mesh algebra of type $\mathbb{A}_2^{(m)}$,
and then $n=2m$, or $\mathbb{L}_1^{(m)}$, and then $n=2m-1$. By the
explicit definition of the minimal projective resolution of
$\Lambda$ as a bimodule (see \cite{D2}), we get that
$\Omega_{\Lambda^e}(\Lambda )$ is generated as a $\Lambda$-bimodule
by the elements $x_i=a_i\otimes e_{i+1}-e_i\otimes a_i$
($i\in\mathbb{Z}_n$). But we have $\oplus_{i\in\mathbb{Z}_n}\Lambda
x_i=\Omega_{\Lambda^e}(\Lambda
)=\oplus_{i\in\mathbb{Z}_n}x_i\Lambda$. Moreover, if $\bar{\mu}$ is
the automorphism mentioned in assertion 5 and
$x=\sum_{i\in\mathbb{Z}_n}x_i$, then we have $yx=x\bar{\mu}(y)$,
whenever $y$ is either a vertex or an arrow. It then follows that
the assignment $y\rightsquigarrow yx$ gives an isomorphism of
$\Lambda$-bimodules
$_1\Lambda_{\bar{\mu}^{-1}}\stackrel{\cong}{\longrightarrow}\Omega_{\Lambda^e}(\Lambda
)$.

\end{proof}

\subsection{The stable Calabi-Yau dimension of an $m$-fold mesh algebra}

In case $\Lambda$ is a self-injective algebra,  Auslander formula
(see \cite{ARS}, Chapter IV, Section 4) says that one has a natural
isomorphism $D\underline{Hom}_\Lambda (X,-)\cong Ext_\Lambda
^1(-,\tau X)$, where $\tau: \Lambda-\underline{\text{mod}}\longrightarrow
\Lambda-\underline{\text{mod}}$ is the Auslander-Reiten (AR) translation.
Moreover, $\tau= \Omega^2\mathcal{N}$, where $\mathcal{N}= D
Hom_{\Lambda}(-,\Lambda)\cong D(\Lambda)\otimes _{\Lambda}-:
\Lambda-\text{mod} \longrightarrow \Lambda-\text{mod}$ is the Nakayama
functor (see \cite{ARS}). As a consequence, as shown in \cite{ESk},
the stable category $A-\underline{\text{mod}}$ has CY-dimension $m$ if and
only if $m$ is the smallest natural number such that
$\Omega_\Lambda^{-m-1}\cong \mathcal{N}\cong
{}_{\bar{\eta}^{-1}}(-)$ (equivalently, $\Omega_\Lambda^{m+1}\cong
{}_{\bar{\eta}}(-)$) as triangulated functors
$\Lambda-\underline{\text{mod}}\longrightarrow
\Lambda-\underline{\text{mod}}$, where $\bar{\eta}$ is the Nakayama
automorphism of $\Lambda$. We shall say that $\Lambda$ is
\emph{stably Calabi-Yau} when $\Lambda -\underline{\text{mod}}$ is a
Calabi-Yau triangulated category. The minimal number $m$ mentioned
above will be then called the \emph{stable Calabi-Yau dimension} of
$\Lambda$ and denoted $CY-\text{dim}(\Lambda )$.

Due to the fact the functor $\Omega_\Lambda^d:\Lambda
-\underline{\text{mod}}\longrightarrow \Lambda -\underline{\text{mod}}$ is
naturally isomorphic to the functor
$\Omega_{\Lambda^e}^d(\Lambda)\otimes_\Lambda -$, for all integers
$d$,  a sufficient condition for $\Lambda$ to be stably Calabi-Yau is that $\Omega_{\Lambda^e}^{d+1}(\Lambda )\cong
{}_{\bar{\eta}}\Lambda_1$ as $\Lambda$-bimodules. An algebra
satisfying this last condition is called \emph{Calabi-Yau Frobenius}
in \cite{ES} and the minimal $d$ satisfying this property is called
the \emph{Calabi-Yau Frobenius dimension} of $\Lambda$. We will
denote it here by $CYF-\text{dim}(\Lambda )$. We always have
$CY-\text{dim}(\Lambda )\leq CYF-\text{dim}(\Lambda )$, but, in
general, it is not known if the equality holds. We  discuss now this
problem for $m$-fold mesh algebras.

Note that, by \cite{IV}[Theorem 1.8], the functor
$\Omega_{\Lambda}^{k+1}:\Lambda-\underline{\text{mod}}\longrightarrow\Lambda-\underline{\text{mod}}$
is naturally isomorphic to
${}_{\bar{\eta}}(-):\Lambda-\underline{\text{mod}}\longrightarrow\Lambda-\underline{\text{mod}}$
if, and only if, $\Omega_{\Lambda^e}^{k+1}(\Lambda )$ and
${}_{\varphi\bar{\eta}}\Lambda_1$ are isomorphic
$\Lambda$-bimodules, for some stably inner automorphism $\varphi$ of
$\Lambda$.

We are now able to calculate the stable and Frobenius Calabi-Yau
dimension of self-injective algebras of Loewy length $2$.

\begin{prop} \label{prop.CY-dim and CYF-dim of 2-nilpotent algebras}
Let $\Lambda$ be a connected  self-injective algebra of Loewy length
$2$. Then $\Lambda$ is always a stably Calabi-Yau algebra and the
following equalities hold:

\begin{enumerate}
\item If $\text{char}(K)=2$ or $\Lambda =\mathbb{A}_2^{(m)}$, i.e. $|Q_0|$ is even,  then $CY-\text{dim}(\Lambda )=CYF-\text{dim}(\Lambda
)=0$. \item If $\text{char}(K)\neq 2$ and  $\Lambda
=\mathbb{L}_1^{(m)}$, i.e., $|Q_0|$ odd, then $CY-\text{dim}(\Lambda
)=0$ and $CYF-\text{dim}(\Lambda )=2m-1=|Q_0|$.
\end{enumerate}
\end{prop}
\begin{proof}
By  Proposition \ref{prop.2-nilpotent selfinjective algebras}, we
know that $\Omega_{\Lambda^e}^{k+1}(\Lambda)$ is isomorphic to
${}_{\bar{\mu}^{k+1}}\Lambda_1$, for each $k\geq 0$. Then
$CY-\text{dim}(\Lambda )$ is the smallest of the natural numbers $k$
such that $\bar{\mu}^{k+1}\bar{\eta}^{-1}$ is stably inner, which is
equivalent to saying that $\bar{\mu}^{k+1}\bar{\eta}^{-1}$ fixes the
vertices. Similarly, $CYF-\text{dim}(\Lambda )$ is the smallest of
the $k$ such that $\bar{\mu}^{k+1}\bar{\eta}^{-1}$ is inner. Due to
the fact that $\Lambda$ is an m-fold mesh algebra of type
$\mathbb{A}_2$, a (graded) Nakayama automorphism of $\Lambda$ is
$\bar{\nu} =\bar{\rho}\bar{\tau}^{1-1}=\bar{\rho}$ (see Theorem \ref{teor.G-invariant
Nakayama automorphism} and Proposition \ref{prop:Coxeter-Nakayama}).
It follows that the graded Nakayama automorphism   $\bar{\eta}$ of
$\Lambda$ maps $i\rightsquigarrow i+1$ and $a_i\rightsquigarrow
a_{i+1}$, when we identify $Q_0=\mathbb{Z}_n$. It follows that
$\bar{\mu}\bar{\eta}^{-1}$ fixes the vertices and, hence, it is
stably inner. This shows that $CY-\text{dim}(\Lambda )=0$.

More generally, $\bar{\mu}^{k+1}\bar{\eta}^{-1}$ fixes the vertices
if, and only if, $i+k+1\equiv i+1\text{ (mod n)}$, for each
$i\in\mathbb{Z}_n$. That is, if and only if  $k\in n\mathbb{Z}$.
Suppose that this property holds and consider the map $\chi
:Q_1\longrightarrow K^*$ taking constant value $(-1)^{k+1}$. We
clearly have $\bar{\mu}^{k+1}\bar{\eta}^{-1}(a_i)=(-1)^{k+1}a_i=\chi
(a_i)a_i$, for each $i\in\mathbb{Z}_n$. But $\chi$ is an acyclic
character if, and only if, either $\text{char}(K)=2$ or
$\prod_{1\leq i\leq n}\chi (a_i)=(-1)^{(k+1)n}$ is equal to $1$. So,
when $\text{char}(K)=2$, the automorphism
$\bar{\mu}^{k+1}\bar{\eta}^{-1}$ is inner for any value of $k$. In
particular, $CYF-\text{dim}(\Lambda )=0$ in such case.

Suppose that $\text{char}(K)\neq 2$. By Proposition
\ref{prop.2-nilpotent selfinjective algebras}, we get that
$\bar{\mu}^{k+1}\bar{\eta}^{-1}$ is an inner automorphism if, and
only if,  $(k+1)n$ is even. This is always the case when $n$ is
even, and in such case $CYF-\text{dim}(\Lambda )=0$. If $n=2m-1$ is
odd then $k+1$ should be even and the smallest $k\in n\mathbb{Z}$
satisfying this property is $k=n$. Then $CYF-\text{dim}(\Lambda
)=n=2m-1$ in this case.

\end{proof}

We also have:

\begin{prop} \label{prop.CY-dimension-versus-Eu-Schedler}
Let $\Lambda$ be an $m$-fold mesh algebra of Dynkin type $\Delta$
different from $\mathbb{A}_r$, for $r=1,2,3$. Then $\Lambda$ is
stably Calabi-Yau if, and only if, it is Calabi-Yau Frobenius. In
such case the equality $CY-\text{dim}(\Lambda
)=CYF-\text{dim}(\Lambda )$ holds.
\end{prop}
\begin{proof}
By Corollary \ref{cor.simple socle of a projective}, we know that
the Loewy length of $\Lambda$ is $c_\Delta -1$, where $c_\Delta$ is
the Coxeter number.  The Dynkin graphs $\Delta =\mathbb{A}_r$, with
$r=1,2,3$, are the only ones for which $c_\Delta -1\leq 3$. So
$\Lambda$ has Loewy length $\geq 4$ in our case. Note that if
$\Omega_{\Lambda^e}^{k+1}(\Lambda )$ is isomorphic to a twisted
bimodule $_\varphi\Lambda _1$, then we have
$\text{dim}(\Omega_{\Lambda^e}^{k+1}(\Lambda ))=\text{dim}(\Lambda
)$. By Lemma \ref{lem.commutativity of eta y tau}, we know that then
$k+1\in 3\mathbb{Z}$.

If there is a $k$ such that $\Omega_{\Lambda^e}^{k+1}(\Lambda )\cong
{}_{\varphi\bar{\eta}}\Lambda_1$, for some inner or stably inner
automorphism $\varphi$, then $k=3s-1$, for some integer $s>0$. But
we know that $\Omega_{\Lambda^e}^{3}(\Lambda )\cong
{}_{\bar{\mu}}\Lambda_1$, where $\bar{\mu}$ is a graded automorphism
of $\Lambda$. We then have that $\Omega_{\Lambda^e}^{3s}(\Lambda
)\cong{}_{\varphi\bar{\eta}}\Lambda_1$, for some stably inner (resp.
inner) automorphism $\varphi$ if, and only if,
$\bar{\mu}^s\bar{\eta}^{-1}$ is a stably inner (resp. inner)
automorphism of $\Lambda$.
 The proof is finished using Lemma \ref{lem.inner-versus-stably
inner} since $\bar{\mu}^s\bar{\eta}^{-1}$ is a graded automorphism.
\end{proof}

The proof of last proposition shows that if $\Lambda$ is not of type
$\mathbb{A}_r$ ($r=1,2$), then  the algebra $\Lambda$ will be stably
Calabi-Yau (resp. Calabi-Yau Frobenius) if, and only if, there
exists an integer $s>0$ such that $\bar{\mu}^{s}\bar{\eta}^{-1}$ is
stably inner (resp. inner). A necessary condition for this is that
$\bar{\mu}^{s}\bar{\eta}^{-1}$ fixes the vertices. So, as a first
step to characterize the stably Calabi-Yau (resp. Calabi-Yau
Frobenius) condition of $\Lambda$, we shall identify the positive
integers $s$ such that $\bar{\mu}^s$ and $\bar{\eta}$ have the same
action on vertices.

\begin{defi} \label{defi.Calabi-Yau number}
Let $\Lambda$ be an $m$-fold mesh algebra of type
$\Delta\neq\mathbb{A}_1,\mathbb{A}_2$,  with quiver $Q$. We will
define the following sets of positive integers:

\begin{enumerate}
\item $\mathbb{N}_{CY}(\Lambda )$ consists of the integers $s>0$
such that $\bar{\mu}^s$ and $\bar{\eta}$ have the same action on
vertices. \item  $\hat{\mathbb{N}}_{CY}(\Lambda )$ consists of the
integers $s>0$ such that $\bar{\mu}^s\bar{\eta}^{-1}$ is an inner
automorphism. Equivalently, it is the set of integers $s>0$ such
that $\Omega_{\Lambda^e}^{3s}(\Lambda)$ is isomorphic to
${}_{\bar{\eta}}\Lambda_1$ as a $\Lambda$-bimodule.
\end{enumerate}
\end{defi}

\begin{rem} \label{rem.CY-dim versus CYF-dim}
Under the hypotheses of last definition, we clearly have
$\hat{\mathbb{N}}_{CY}(\Lambda )\subseteq \mathbb{N}_{CY}(\Lambda
)$. Moreover $\Lambda$ is Calabi-Yau Frobenius if, and only if,
$\hat{\mathbb{N}}_{CY}(\Lambda )\neq\emptyset$. In this latter case
we have $CYF-\text{dim}(\Lambda )=3r-1$, where
$r=\text{min}(\hat{\mathbb{N}}_{CY}(\Lambda ))$, and this number is
equal to $CY-\text{dim}(\Lambda)$ when $\Delta\neq\mathbb{A}_3$.
Note also that if $\hat{\mathbb{N}}_{CY}(\Lambda )=
\mathbb{N}_{CY}(\Lambda )\neq\emptyset$ then $CY-\text{dim}(\Lambda
)=CYF-\text{dim}(\Lambda )$ since the fact that
$\bar{\mu}^{s}\bar{\eta}^{-1}$ be stably inner implies that $s\in
\mathbb{N}_{CY}(\Lambda )$.
\end{rem}

We first identify $\mathbb{N}_{CY}(\Lambda )$ for any $m$-fold mesh
algebra of Loewy length $>2$.

\begin{prop} \label{prop.Calabi-Yau set of numbers}
Let $\Lambda$ be an $m$-fold mesh algebra of extended type $(\Delta
,m,t)$, where $\Delta\neq\mathbb{A}_1,\mathbb{A}_2$. The following
assertions hold:

\begin{enumerate}
\item When $t=1$,  the set $\mathbb{N}_{CY}(\Lambda )$ is nonempty
if, and only if, the following condition is true in each case:

\begin{enumerate}
\item $\text{gcd}(m,c_\Delta )=1$, when $\Delta$ is $\mathbb{A}_{r}$, $\mathbb{D}_{2r-1}$ or
$\mathbb{E}_{6}$. Then $\mathbb{N}_{CY}(\Lambda )=\{s=2s'+1: s'\geq 0 \text{ and }$
$c_\Delta s'\equiv -1\text{ (mod $m$)}\}$

\item $\text{gcd}(m,\frac{c_\Delta}{2})=1$, when $\Delta$ is
$\mathbb{D}_{2r}$, $\mathbb{E}_{7}$ or $\mathbb{E}_{8}$. Then
$\mathbb{N}_{CY}(\Lambda )=\{s>0:$ $\frac{c_\Delta}{2}(s-1)\equiv
-1\text{ (mod $m$)}\}$.
\end{enumerate}

\item When $t=2$, the set $\mathbb{N}_{CY}(\Lambda )$ is nonempty
if, and only if, the following condition is true in each case:

\begin{enumerate}
\item $\text{gcd}(2m,m+\frac{c_\Delta}{2})=1$, when $\Delta$ is
$\mathbb{A}_{2n-1}$, $\mathbb{D}_{2r-1}$ or $\mathbb{E}_{6}$. Then
$\mathbb{N}_{CY}(\Lambda )=\{s>0:$
$(m+\frac{c_\Delta}{2})(s-1)\equiv -1\text{ (mod 2m)}\}$, and this
set consists of even numbers.

\item
$\text{gcd}(m,\frac{c_\Delta}{2})=\text{gcd}(m,,2r-1)=1$, when
$\Delta =\mathbb{D}_{2r}$. Then $\mathbb{N}_{CY}(\Lambda )=\{s>0:$
$(2r-1)(s-1)\equiv -1\text{ (mod $2m$)}\}$ and this set consists of
even numbers.

\item $\text{gcd}(2m-1,2n+1)=1$, when $\Delta=\mathbb{A}_{2n}$. Then
$\mathbb{N}_{CY}(\Lambda )=\{s>0:$ $(m+n)(s-1)\equiv -1\text{ (mod
2m-1)}\}$.

\end{enumerate}

\item If $t=3$ (and hence $\Delta =\mathbb{D}_4$), then
$\mathbb{N}_{CY}(\Lambda )=\emptyset$.
\end{enumerate}
\end{prop}
\begin{proof}
 Note
that $\bar{\mu}$ acts on vertices as $\bar{\nu}\bar{\tau}^{-1}$,
where $\nu$ is the Nakayama permutation and $\tau$ the
Auslander-Reiten translation of $B$. Viewing the vertices of the
quiver of $\Lambda$ as $G$-orbits of vertices in $\mathbb{Z}\Delta$,
we get that $s$ is in $\mathbb{N}_{CY}(\Lambda )$ if, and only if,
$(\bar{\nu}\bar{\tau}^{-1})^s([(k,i)])=\bar{\nu}([(k,i)])$,
equivalently $\bar{\nu}^{s-1}\bar{\tau}^{-s}([(k,i)])=[(k,i)]$, for
each $G$-orbit $[(k,i)]$.  Now the argument  in the first paragraph
of the proof of Theorem \ref{teor.weakly symmetric m-fold algebras}
can be applied to the automorphism $\nu^{s-1}\tau^{-s}$. We then get
that $s\in\mathbb{N}_{CY}(\Lambda )$ if, and only if,
$\nu^{s-1}\tau^{-s}\in G$.
 We use this to identify the set
$\mathbb{N}_{CY}(\Lambda )$ for all possible extended types, and the
result will be derived from that.

 If $t=3$ and so $\Delta =\mathbb{D}_4$, then we know that $\nu
=\tau^{-2}$. It follows that $s\in \mathbb{N}_{CY}(\Lambda )$ if,
and only if, $\tau^{-2(s-1)}\tau^{-s}=(\rho\tau^m)^q$, for some
$q\in\mathbb{Z}$, where $\rho$ is the automorphism of order $3$ of
$\mathbb{D}_4$. By the free action of the group $\langle\rho ,\tau \rangle$ on
vertices not fixed by $\rho$,  necessarily $q\in 3\mathbb{Z}$ and
$2-3s=mq$, which is absurd. Then assertion 3 follows.

Suppose first that $\Delta\neq\mathbb{A}_{2n}$. If $\Delta$ is
$\mathbb{A}_{2n-1}$, $\mathbb{D}_{2r-1}$ or $\mathbb{E}_{6}$, then
$\nu =\rho\tau^{1-n}$, where $n=\frac{c_\Delta}{2}$. Then
$\nu^{s-1}\tau^{-s}=\rho^{s-1}\tau^{(1-n)(s-1)}\tau^{-s}=\rho^{s-1}\tau^{-[n(s-1)+1]}$.
When $t=1$, we have that $G=\langle\tau^m\rangle$ and, hence, the automorphism
$\nu^{s-1}\tau^{-s}$ is in $G$ if, and only if, there is
$q\in\mathbb{Z}$ such that
$\rho^{s-1}\tau^{-[n(s-1)+1]}=(\tau^m)^q$. This happens if, and only
if, $s-1=2s'$ is even and there is  $q\in\mathbb{Z}$ such that
$-2ns'-1=-n(s-1)-1$ is equal to $mq$. Therefore $s$ exists if, and
only if, $\text{gcd}(m,c_\Delta )=\text{gcd}(m,2n)=1$. In this case
$\mathbb{N}_{CY}(\Lambda )=\{s=2s'+1>0:$ $2ns'\equiv -1\text{ (mod
m)}\}=\{s=2s'+1:$ $c_\Delta s'\equiv -1\text{ (mod m)}\}$, which
gives 1.a, except for the case $\Delta =\mathbb{A}_{2n}$. On the
other hand, if $t=2$, and hence $G=\langle\rho\tau^m\rangle$, then the
automorphism $\nu^{s-1}\tau^{-s}$ is in $G$ if, and only if, there
is an integer $q$ such that $q\equiv s-1\text{ (mod 2)}$ and
$\rho^{s-1}\tau^{-[n(s-1)+1]}=\rho^q\tau^{mq}$ or, equivalently,
$-n(s-1)-1=mq$. But this happens if, and only if, there is
$k\in\mathbb{Z}$ such that $-n(s-1)-1=m(s-1+2k)$, which is
equivalent to saying that $(m+n)(s-1)+2mk+1=0$. Therefore $s$ exists
if, and only if,
$\text{gcd}(2m,m+\frac{c_\Delta}{2})=\text{gcd}(2m,m+n)=1$. In this
case $\mathbb{N}_{CY}(\Lambda )=\{s>0:$
$(m+\frac{c_\Delta}{2})(s-1)\equiv -1\text{ (mod 2m)}\}$ which
proves 2.a.

Suppose next that $\Delta$ is $\mathbb{D}_{2r}$, $\mathbb{E}_{7}$ or
$\mathbb{E}_{8}$, so that $\nu =\tau^{1-n}$, where
$n=\frac{c_\Delta}{2}$. Then
$\nu^{s-1}\tau^{-s}=\tau^{(1-n)(s-1)}\tau^{-s}=\tau^{-[n(s-1)+1]}$.
When $t=1$, this automorphism is in $G=\langle\tau^m\rangle$ if, and only if,
there is $q\in\mathbb{Z}$ such that $-n(s-1)-1=mq$. Then $s$ exists
if, and only if,
$\text{gcd}(m,\frac{c_\Delta}{2})=\text{gcd}(m,n)=1$. In this case
$\mathbb{N}_{CY}(\Lambda )=\{s>0:$ $\frac{c_\Delta}{2}(s-1)\equiv
-1\text{ (mod m)}$, which proves 1.b. When $t=2$, whence $\Delta
=\mathbb{D}_{2r}$, the automorphism $\nu^{s-1}\tau^{-s}$ is in
$G=\langle\rho\tau^m\rangle$ if, and only if, there is an even integer $q=2q'$
such that $-n(s-1)-1=2mq'$. Then $s$ exists if, and only if,
$\text{gcd}(2m,n)=1$. But $n=2r-1$ is odd in this case. Then
$\text{gcd}(2m,n)=1$ if, and only if,
$\text{gcd}(m,2r-1)=\text{gcd}(m,n)=1$. On the other hand, note that
$s-1$ is necessarily odd, which implies that
$\mathbb{N}_{CY}(\Lambda )\subset 2\mathbb{Z}$. This completes the
proof of 2.b.

Suppose now that $\Delta =\mathbb{A}_{2n}$, so that
$\rho^2=\tau^{-1}$. Here $\nu =\rho\tau^{1-n}$ and
$\nu^{s-1}\tau^{-s}=\rho^{s-1}\tau^{(1-n)(s-1)-s}$ $=\rho^{s-1}\tau^{-[n(s-1)+1]}$.
When $t=1$, this automorphism is in $G=\langle\tau^m\rangle$ if, and only if,
$s-1=2s'$ is even and $\tau^{-s'}\tau^{-(2ns'+1)}=(\tau^m)^q$, for
some integer $q$. That is, $s$ exists if, and only if, there are
$s'\geq 0$ and $q\in\mathbb{Z}$ such that $mq+(2n+1)s'+1=0$.
Therefore $s$ exists if, and only if,
$\text{gcd}(m,c_\Delta)=\text{gcd}(m,2n+1)=1$. In this case
$s=2s'+1$, where $s'\geq 0$ and $c_\Delta s'=(2n+1)s'\equiv -1\text{
(mod m)}$. This completes 1.a. When $t=2$ the automorphism
$\nu^{s-1}\tau^{-s}$ is in $G=\langle\rho\tau^m\rangle$ if, and only if, there
is $q\in\mathbb{Z}$ such that $q\equiv s-1\text{ (mod 2)}$ and
$\rho^{s-1}\tau^{-[n(s-1)+1]}=\rho^q\tau^{mq}$. This is equivalent
to the existence of an integer $k$ such that
$\rho^{s-1}\tau^{-[n(s-1)+1]}=\rho^{s-1+2k}\tau^{m(s-1+2k)}$.
Canceling $\rho^{s-1}$, we see that the condition is equivalent to
the existence of an integer $k$ such that $-n(s-1)-1=m(s-1)+(2m-1)k$
or, equivalently, such that $(m+n)(s-1)+(2m-1)k+1=0$. Then $s$
exists if, and only if, $\text{gcd}(m+n,2m-1)=1$, which is turn
equivalent to saying that $\text{gcd}(2m-1,2n+1)=1$ since
$(2m-1)+(2n+1)=2(m+n)$. This proves 2.c and the proof is complete.
\end{proof}

We now want to identify $\hat{\mathbb{N}}_{CY}(\Lambda )$.  The
following is our crucial tool.

\begin{lema} \label{lem.CYcriterion}
Let $\Delta$ be a Dynkin quiver different from
$\mathbb{A}_1,\mathbb{A}_2$, $B$ be its associated mesh algebra,
$\Lambda=B/G$ be an $m$-fold mesh algebra of extended type
$(\Delta ,m,t)$ and let $\eta$ be a $G$-invariant graded Nakayama
automorphism of $B$. If $s$ is an integer in
$\mathbb{N}_{CY}(\Lambda )$, then the following assertions are
equivalent:

\begin{enumerate}
\item $s$ is in $\hat{\mathbb{N}}_{CY}(\Lambda )$ (see definition \ref{defi.Calabi-Yau
number}).
\item There is a map $\lambda :\mathbb{Z}\Delta_0\longrightarrow
K^*$ such that:

\begin{enumerate}
\item $\mu^{s}(a)=\lambda_{i(a)}^{-1}\lambda_{t(a)}\eta (\nu^{s-1}\tau^{-s}(a))$,
for all $a\in (\mathbb{Z}\Delta )_1$, where $\mu$ is the graded
automorphism of Proposition \ref{prop.automorphism mu general}.
\item $\lambda\circ g_{|\mathbb{Z}\Delta_0}=\lambda$, for all $g\in
G$.
\end{enumerate}
If $(\Delta ,m,t)\neq (\mathbb{A}_{2n-1},m,2)$, then these
conditions are also equivalent to:

\item There is a map  $\lambda :\mathbb{Z}\Delta_0\longrightarrow
K^*$ satisfying condition 2.b and such that
$(-1)^s\eta^{s-1}(a)=\lambda_{i(a)}^{-1}\lambda_{t(a)}\nu^{s-1}(a)$,
for all $a\in (\mathbb{Z}\Delta )_1$.

If $(\Delta ,t)\neq (\mathbb{A}_r,2)$ then the conditions are also
equivalent to

\item $s-1$ is in $H(\Delta ,m,t)$ (see Proposition \ref{prop.subgroup of Z
associated}).

\end{enumerate}
\end{lema}
\begin{proof}
The first paragraph of the proof of Proposition \ref{prop.Calabi-Yau
set of numbers} says that $s\in\mathbb{N}_{CY}(\Lambda )$ if, and
only if, $\nu^{s-1}\tau^{-s}\in G$. The goal is to give necessary
and sufficient conditions on such an integer $s$  so that
$\bar{\mu}^s$ and $\bar{\eta}=\overline{\eta\nu^{s-1}\tau^{-s}}$ are
equal, up to composition by an inner automorphism of $\Lambda$. But
the actions of $\mu^s =(k\circ\eta\circ\tau^{-1}\circ\vartheta )^s$
and $\eta\circ\nu^{s-1}\circ\tau^{-s}$ on $\mathbb{Z}\Delta_0$ are
equal. By Lemma \ref{lem.criterion of inner automorphism}, we then
get that assertions 1 and 2 are equivalent.

When $(\Delta ,m,t)\neq (\mathbb{A}_{2n-1},m,2)$, what we
know is that $\vartheta =id_B$ and, by Proposition \ref{prop.tau nu eta in center}, we know that $\bar{\eta}$ and $\bar{\tau}^{-1}$
commute, up to composition by an inner automorphism of $\Lambda$.
Then $s$ is in $\hat{\mathbb{N}}_{CY}(\Lambda )$ if, and only if,
$\bar{k}^s\bar{\eta}^s\bar{\tau}^{-s}$ and
$\bar{\eta}\bar{\nu}^{s-1}\bar{\tau}^{-s}$  are equal up to
composition by an inner automorphism of $\Lambda$. By Lemma
\ref{lem.criterion of inner automorphism}, this last condition is
equivalent to saying that there is a map $\lambda
:\mathbb{Z}\Delta_0\longrightarrow K^*$ satisfying 2.b such that
$(-1)^s\eta^s(\tau^{-s}(a))=\lambda_{i(a)}^{-1}\lambda_{t(a)}\eta
(\nu^{s-1}\tau^{-s}(a))$, for each $a\in (\mathbb{Z}\Delta )_1$.
Putting $b=\tau^{-s}(a)$ and defining
$\tilde{\lambda}:(\mathbb{Z}\Delta )_0\longrightarrow K^*$ by the
rule $\tilde{\lambda}(i)=\lambda (\tau^s(i))$, we get that
$(-1)^s\eta^{s-1}(b)=\tilde{\lambda}_{i(b)}^{-1}\tilde{\lambda}_{t(b)}\nu^{s-1}(b)$,
for all $b\in(\mathbb{Z}\Delta )_1$. Then assertions 2 and 3 are
equivalent.

Finally, when $(\Delta ,t)\neq (\mathbb{A}_r,2)$, Corollary \ref{cor. mu for Lambda} says that we can choose $\mu
=\eta\circ\tau^{-1}$. Then the
proof of the equivalence of assertions 2 and 3, taken for
$\kappa=id_B$, shows that assertion 2 holds if, and only if, there
is a map  $\lambda :\mathbb{Z}\Delta_0\longrightarrow K^*$
satisfying condition 2.b and such that
$\eta^{s-1}(b)=\lambda_{i(b)}^{-1}\lambda_{t(b)}\nu^{s-1}(b)$, for
all $b\in (\mathbb{Z}\Delta )_1$. This is equivalent to saying that
$s-1\in H(\Delta ,m,t)$.
\end{proof}

The following is now a consequence of Proposition
\ref{prop.Calabi-Yau set of numbers} and the foregoing lemma.

\begin{cor} \label{cor.CY-criterion in characteristic 2}
Let $\Lambda$ be an $m$-fold mesh algebra over a field of
characteristic $2$,  with $\Delta\neq\mathbb{A}_1$. The algebra is
stably Calabi-Yau if, and only if, it is Calabi-Yau Frobenius. When
in addition $\Delta\neq \mathbb{A}_2$, this is in turn equivalent to
saying that $\mathbb{N}_{CY}(\Lambda )\neq\emptyset$. Moreover, the
following assertions hold:
\begin{enumerate}
\item When the Loewy length of $\Lambda$ is $\leq 2$, i.e. $\Delta
=\mathbb{A}_2$,  the algebra is always Calabi-Yau Frobenius and
 $CY-\text{dim}(\Lambda )=CYF-\text{dim}(\Lambda )=0$.

\item When $\Delta\neq\mathbb{A}_2$, we have $CY-\text{dim}(\Lambda )=CYF-\text{dim}(\Lambda
)=3m-1$, where $m=\text{min}(\mathbb{N}_{CY}(\Lambda ))$ (see
Proposition \ref{prop.Calabi-Yau set of numbers}).
\end{enumerate}
\end{cor}
\begin{proof}
The case of Loewy length $2$ is covered by Proposition
\ref{prop.CY-dim and CYF-dim of 2-nilpotent algebras}. So we assume
$\Delta\neq\mathbb{A}_2$ in the sequel. If $\Lambda$ is stably
Calabi-Yau, then $\mathbb{N}_{CY}(\Lambda )\neq\emptyset$. But, when
$\text{char}(K)=2$, the $G$-invariant graded Nakayama automorphism
of Theorem \ref{teor.G-invariant Nakayama automorphism} is $\eta
=\nu$. In addition, the automorphisms $\vartheta$ and $\kappa$ of
Proposition \ref{prop.automorphism mu general} are the identity.
Then, in order to prove the equality $\hat{\mathbb{N}}_{CY}(\Lambda
)=\mathbb{N}_{CY}(\Lambda)$, one only need to prove that if
$s\in\mathbb{N}_{CY}(\Lambda)$ then condition 2 of last lemma holds.
But this is clear, by taking as $\lambda$ any constant map.
\end{proof}

We are now ready to give, for $\text{char}(K)\neq 2$, the precise criterion for an $m$-fold mesh
algebra to be stably Calabi-Yau, and to calculate
$CY-\text{dim}(\Lambda )$ in that case.

\begin{teor} \label{teor.CY-criterion and CY-dimension}
Let us assume that $\text{char}(K)\neq 2$ and let $\Lambda$ be the
$m$-fold mesh algebra of extended type $(\Delta ,m,t)$, where
$\Delta\neq\mathbb{A}_1,\mathbb{A}_2$. We adopt the convention that
if $a,b,k$ are fixed integers,
 then  $av\equiv b\text{ (mod k)}$  means that $v$ is the smallest non-negative integer satisfying the congruence.
The algebra is Calabi-Yau Frobenius if, and only if, it is stably
Calabi-Yau. Moreover,  we have
$CYF-\text{dim}(\Lambda)=CY-\text{dim}(\Lambda )$ and the following
assertions hold:

\begin{enumerate}
\item If $t=1$ then

\begin{enumerate}
\item When $\Delta$ is $\mathbb{A}_{r}$, $\mathbb{D}_{2r-1}$ or
$\mathbb{E}_{6}$, the algebra is stably Calabi-Yau if, and only if,
$\text{gcd}(m,c_\Delta )=1$. Then $CY-\text{dim}(\Lambda )=6u+2$,
where $c_\Delta u\equiv -1\text{ (mod m)}$.
\item  When $\Delta$ is
$\mathbb{D}_{2r}$, $\mathbb{E}_{7}$ or $\mathbb{E}_{8}$, the algebra
is stably Calabi-Yau if, and only if,
$\text{gcd}(m,\frac{c_\Delta}{2})=1$. Then:

\begin{enumerate}
\item $CY-\text{dim}(\Lambda )=3u+2$, where $\frac{c_\Delta}{2}u\equiv -1\text{ (mod m)}$,  whenever $m$ is even;
\item $CY-\text{dim}(\Lambda )=6u+2$, where  $c_\Delta u\equiv -1\text{ (mod
m)}$, whenever $m$ is odd;
\end{enumerate}

\end{enumerate}
\item If $t=2$ then

\begin{enumerate}
\item When $\Delta$ is $\mathbb{A}_{2n-1}$, $\mathbb{D}_{2r-1}$ or
$\mathbb{E}_{6}$, the algebra is stably Calabi-Yau if, and only if,
$\text{gcd}(2m,m+\frac{c_\Delta}{2})=1$. Then $CY-\text{dim}(\Lambda
)=3u+2$, where $(m+\frac{c_\Delta}{2}) u\equiv -1\text{ (mod 2m)}$.

\item When $\Delta=\mathbb{D}_{2r}$, the algebra is stably
Calabi-Yau if, and only if, $\text{gcd}(m,2r-1)=1$ and $m$ is odd.
Then $CY-\text{dim}(\Lambda )=3u+2$, where $(2r-1)u\equiv -1\text{
(mod 2m)}$.

\item When $\Delta =\mathbb{A}_{2n}$, the algebra is stably
Calabi-Yau if, and only if, $\text{gcd}(2m-1,2n+1)=1$. Then
$CY-\text{dim}(\Lambda )=6u-1$, where $(m+n)(2u-1)\equiv -1\text{
(mod 2m-1)}$
\end{enumerate}
\item If $t=3$ then the algebra is not stably Calabi-Yau.
\end{enumerate}
\end{teor}
\begin{proof}
By Proposition \ref{prop.CY-dimension-versus-Eu-Schedler}, we know
that, when $\Delta\neq\mathbb{A}_3$, the algebra $\Lambda$ is stably
Calabi-Yau if, and only if, it is Calabi-Yau Frobenius and the
corresponding dimensions are equal. From our arguments below it will
follow that, when $\Delta =\mathbb{A}_3$, we always have
$\hat{\mathbb{N}}_{CY}(\Lambda )=\mathbb{N}_{CY}(\Lambda )$, and
then $CY-\text{dim}(\Lambda )=CYF-\text{dim}(\Lambda )$ also in this
case (see Remark \ref{rem.CY-dim versus CYF-dim}).

Our arguments will give an explicit identification of
$\hat{\mathbb{N}}_{CY}(\Lambda )$ in terms of
$\mathbb{N}_{CY}(\Lambda )$. Then $CY-\text{dim}(\Lambda )$ will be
$3v-1$, where $v=\text{min}(\hat{\mathbb{N}}_{CY}(\Lambda ))$.

From Propositions \ref{prop.Calabi-Yau set of numbers} and
\ref{prop.CY-dimension-versus-Eu-Schedler},  we know that, when
$t=3$, the algebra is never stably Calabi-Yau. So we assume in the
sequel that $t\neq 3$.

Suppose first that $(\Delta ,m,t)\neq (\mathbb{A}_r,m,2)$. Then
Lemma \ref{lem.CYcriterion} tells us that
$\hat{\mathbb{N}}_{CY}(\Lambda )=\mathbb{N}_{CY}(\Lambda )\cap
(H(\Delta ,m,t)+1)$, where $H(\Delta ,m,t)+1=\{s\in\mathbb{Z}:$
$s-1\in H:=H(\Delta ,m,t)\}$. By Proposition \ref{prop.subgroup of Z
associated}, we get in these cases that the equality
$\hat{\mathbb{N}}_{CY}(\Lambda )=\mathbb{N}_{CY}(\Lambda )$ holds
whenever $m+t$ is odd. We now examine the different cases:

1.a) If $\Delta =\mathbb{A}_r$ then $H=\mathbb{Z}$. When $\Delta$ is
$\mathbb{D}_{2r-1}$ or $\mathbb{E}_{6}$, the Coxeter number
$c_\Delta$ is even. If $\mathbb{N}_{CY}(\Lambda )\neq\emptyset$ then
$\text{gcd}(m,c_\Delta )=1$, so that $m$ is odd and $H=2\mathbb{Z}$.
But then $\hat{\mathbb{N}}_{CY}(\Lambda )=\mathbb{N}_{CY}(\Lambda
)\cap (2\mathbb{Z}+1)$, which is equal to $\mathbb{N}_{CY}(\Lambda
)$ due to Proposition \ref{prop.Calabi-Yau set of numbers}. So
$\Lambda$ is stably Calabi-Yau if, and only if,
$\text{gcd}(m,c_\Delta )=1$. Then $CY-\text{dim}(\Lambda
)=3(2u+1)-1=6u+2$, where $2u+1=\text{min}(\mathbb{N}_{CY}(\Lambda
))$.

1.b) We need to consider the case when  $m$ is odd. In this case
$\hat{\mathbb{N}}_{CY}(\Lambda )=\mathbb{N}_{CY}(\Lambda )\cap
(2\mathbb{Z}+1)$ is properly contained in $\mathbb{N}_{CY}(\Lambda
)$. However, we claim that if $\mathbb{N}_{CY}(\Lambda
)\neq\emptyset$ then $\hat{\mathbb{N}}_{CY}(\Lambda )\neq\emptyset$,
which will prove that $\Lambda$ is stably Calabi-Yau if, and only
if, $\text{gcd}(m,\frac{c_\Delta}{2})=1$ using Proposition
\ref{prop.Calabi-Yau set of numbers}. Indeed, we need to prove that
 if $\text{gcd}(m,\frac{c_\Delta}{2})=1$, then there is an integer $u'\geq
 0$ such that $2u'+1\in\mathbb{N}_{CY}(\Lambda
)$ or, equivalently, that $\frac{c_\Delta}{2}(2u'+1-1)\equiv
-1\text{ (mod m)}$. But this is clear for if $m$ is odd then also
$\text{gcd}(m,c_\Delta)=1$. Now the formulas in 1.b.i) an 1.b.ii)
come directly from putting $s=u+1$ and $s=2u+1$ and using the fact
that $\frac{c_\Delta}{2}(s-1)\equiv -1\text{ (mod m)}$.

2.a) Suppose first that $\Delta$ is $\mathbb{D}_{2r-1}$ or
$\mathbb{E}_{6}$. In this case $\frac{c_\Delta}{2}$ is even. Then
$\text{gcd}(2m,m+\frac{c_\Delta}{2})=1$ implies that $m$ is odd and,
hence, that $H=\mathbb{Z}$. So in this case
$\hat{\mathbb{N}}_{CY}(\Lambda )=\mathbb{N}_{CY}(\Lambda )$ and the
formula for $CY-\text{dim}(\Lambda )$ comes from putting $s=1+u$,
where $(m+\frac{c_\Delta}{2})u\equiv -1\text{(mod 2m)}$.

Suppose next that $(\Delta ,m,t)=(\mathbb{A}_{2n-1},m,2)$, i.e.
$\Lambda=\mathbb{B}_n^{(m)}$. Here $\eta =\nu$. Then condition 2 of
Lemma \ref{lem.CYcriterion} can be rephrased by saying that
$\bar{\mu}^s$ and $(\bar{\nu}\circ\bar{\tau}^{-1})^s$ are equal, up
to composition by an inner automorphism of $\Lambda$. This proves
that $\hat{\mathbb{N}}_{CY}(\Lambda )=\mathbb{N}_{CY}(\Lambda )\cap
2\mathbb{Z}$ due to Lemma \ref{lem.crucial for the periods in Bn and
Ln}. But Proposition \ref{prop.Calabi-Yau set of numbers} tells us
that then  $\hat{\mathbb{N}}_{CY}(\Lambda )=\mathbb{N}_{CY}(\Lambda
)$. The formula for $CY-\text{dim}(\Lambda )$ is calculated as in
the other two quivers of 2.a.

2.b)  If $\mathbb{N}_{CY}(\Lambda )\neq\emptyset$ then
$\text{gcd}(m,2r-1)=1$. If $m$ is odd then $H=\mathbb{Z}$. If $m$ is
even, then $H=2\mathbb{Z}$ which implies that
$\hat{\mathbb{N}}_{CY}(\Lambda )=\mathbb{N}_{CY}(\Lambda )\cap
(2\mathbb{Z}+1)$. But this is  the empty set due to Proposition
\ref{prop.Calabi-Yau set of numbers}. The formula for
$CY-\text{dim}(\Lambda )$ in the case when $m$ is odd follows again
from putting $s-1=u$ and $(2r-1)u\equiv -1\text{ (mod 2m)}$.

2.c) It remains to consider the case $(\Delta
,m,t)=(\mathbb{A}_{2n},m,2)$, i.e. $\Lambda=\mathbb{L}_n^{(m)}$. We
use condition 3 of Lemma \ref{lem.CYcriterion}. If $\lambda
:\mathbb{Z}\Delta_0\longrightarrow K^*$ is any map such that
$(-1)^s\eta^{s-1}(a)=\lambda_{i(a)}^{-1}\lambda_{t(a)}\nu^{s-1}(a)$,
then $\lambda_{i(a)}^{-1}\lambda_{t(a)}=(-1)^s$ since
$\eta^{s-1}(a)=\nu^{s-1}(a)$, for all $a\in (\mathbb{Z}\Delta )_1$.
It follows that $\lambda_{(k,i)}=(-1)^s\lambda_{(k,j)}$, whenever
$i\not\equiv j\text{ (mod 2)}$, and that $\lambda_{\tau
(k,i)}=\lambda_{(k+1,i)}=(-1)^{2s}\lambda_{(k,i)}=\lambda_{(k,i)}$,
for all $(k,i)\in\mathbb{Z}\Delta_0$. We then get that
$\lambda_{\rho\tau^m(k,i)}=\lambda_{\rho
(k+m,i)}=\lambda_{(k+m+i-n,2n+1-i)}=(-1)^s\lambda_{(k,i)}$. As a
consequence, the equality $\lambda\circ
g_{|\mathbb{Z}\Delta_0}=\lambda$ holds, for all $g\in
G=\langle\rho\tau^m\rangle$, if, and only if, $s\in 2\mathbb{Z}$. It follows that
$\hat{\mathbb{N}}_{CY}(\Lambda )=\mathbb{N}_{CY}(\Lambda )\cap
2\mathbb{Z}$. We claim that if $\mathbb{N}_{CY}(\Lambda
)\neq\emptyset$ then $\hat{\mathbb{N}}_{CY}(\Lambda )\neq\emptyset$,
which implies that $\Lambda$ is stably Calabi-Yau exactly when
$\text{gcd}(2m-1,2n+1)=1$, using Proposition \ref{prop.Calabi-Yau
set of numbers}. Indeed, using the description of this last
proposition, we need to see that the diophantic equation
$(m+n)(2x-1)+(2m-1)y+1$ has a solution. But this is clear since
$\text{gcd}(2(m+n),2m-1)=1$. The formula for $CY-\text{dim}(\Lambda
)$ is now clear.
\end{proof}

\vspace*{0.5cm}

\section*{Appendix}
%\begin{app}
EXPLICIT CALCULATIONS FOR THE PROOF OF THEOREM \ref{teor.G-invariant Nakayama automorphism}.
%\end{app}

Following the recipe to construct the $G$-invariant basis for $B=B(\Delta )$ of $\text{Soc}_{gr}(B)$ given in the  proof of Theorem \ref{teor.G-invariant Nakayama automorphism},  for each choice of a Dynkin graph $\Delta\in\{\mathbb{D}_{n+1},\mathbb{E}_r\}$ ($n>3$, $r=6,7,8$) and of $\varphi\in\{\tau ,\rho\tau^m\}$, we shall give a convenient subset $I'\subset\mathbb{Z}\Delta_0$ which is a set  of representatives of the $\varphi$-orbits, and we  will then give a nonzero element $w_{(k,i)}\in e_{(k,i)}\text{Soc}_{gr}(B)$, for each $(k,i)\in I'$. Finally, we will use these elements to find, for each $a\in\mathbb{Z}\Delta_1$,  the exponents $u(a)$ and $v(a)$ needed for the explicit formula of $\eta (a)$, as indicated in the mentioned proof.

\begin{enumerate}

\item When $\Delta= \mathbb{D}_{n+1}$ with $n>3$:

To simplify the notation, we shall denote by $u$, $v$ and $w$,
respectively,  each of the paths of length $2$

\begin{center}
$(r,2)\rightarrow (r,0)\rightarrow (r+1,2)$

$(r,2)\rightarrow (r,1)\rightarrow (r+1,2)$

$(r,2)\rightarrow (r,3)\rightarrow (r+1,2)$,
\end{center}
with no mention to $r$. Then a composition of those paths
$(r,2)\rightarrow (r+1,2)\rightarrow ...\rightarrow (r+i,2)$ will be
denoted as a (noncommutative) monomial in the $u,v,w$.

We will need also to name the paths that we will use. Concretely:

\begin{enumerate}[i)]

\item $\gamma_{(k,i)}$ is the downward path $(k,i)\rightarrow
...\rightarrow (k+i-2,2)$, with the convention that
$\gamma_{(k,2)}=e_{(k,2)}$.

\item $\delta_{(m,j)}$ is the upward path $(m,2)\rightarrow
...\rightarrow (m,j)$, with the convention that
$\delta_{(m,2)}=e_{(m,2)}$.

\item $\varepsilon_{(k,j)}$ is the arrow  $(k,2)\longrightarrow (k,j)$ and $\varepsilon^{'}_{(k,j)}$ is the arrow  $(k,j)\longrightarrow
(k+1,2)$, for $j=0,1$.
\end{enumerate}

\begin{enumerate}

\item If $\varphi= \tau$, we will take the canonical slice
$I':=\{(0,i):$ $i\in\Delta_0\}$.

Our choice of the $w_{(0,i)}$ is then the following:

\begin{enumerate}
\item[(a)] $w_{(0,i)}=\gamma_{(0,i)}uvuv...\delta_{(n-1,i)}$ whenever $i=2,...,n$.
\item[(b)] $w_{(0,0)}=\varepsilon '_{(0,0)}vuvu...\varepsilon_{\nu (0,0)}$
\item[(c)] $w_{(0,1)}=\varepsilon '_{(0,1)}uvuv...\varepsilon_{\nu (0,1)}$
\end{enumerate}
(note that, for $j=0,1$, the vertex  $\nu (0,j)$  depends on whether
$n+1$ is even or odd).

%The desired elements $w_{(0,i)}\in e_{(0,i)}Soc_{gr}(B)$ are the
%paths given below:

\vspace{0.2cm}
Now, we will use the letter $\alpha$ to denote un
upward arrow $(k,i)\rightarrow (k,i+1)$, with $i=2,...,n-1$, and the
letter $\beta$ to denote a downward arrow $(k,i)\rightarrow
(k+1,i-1)$ with $i=3,...,n$. We will also consider the arrows
$\varepsilon_j:=\varepsilon_{(k,j)}:(k,2)\rightarrow (k,j)$ and
$\varepsilon_j':=\varepsilon'_{(k,j)}:(k,j)\rightarrow (k+1,2)$, for
$j=0,1$. In all cases we consider that the origin of each arrow is a
vertex of $I'$. We will now create a table, where, for each of these
arrows $a$, the path $p_a$ will be a path of length $\mathit{l}-1$
from $t(a)$ to $\nu (i(a))$ such that $ap_a\neq 0$ in $B$. Then
$u(a),v(a)$ will be elements of $\mathbb{Z}_2$ such that
$ap_a=(-1)^{u(a)}w_{i(a)}$ and $p_a\nu (a)=(-1)^{v(a)}w_{t(a)}$. The
routine verification of these equalities is left to the reader.

\begin{center}

\begin{tabular}{|c|c|c|c|} \hline
 $a$ & $p_a$  & $u(a)$ & $v(a)$\\ \hline
$\alpha :(0,i)\rightarrow (0,i+1)$ &
$\gamma_{(0,i+1)}vuvu...\delta_{(n-1,i)}$ & 0 & 1\\ \hline $\beta
:(0,i)\rightarrow (1,i-1)$ &
$\gamma_{(1,i-1)}uvuv...\delta_{(n-1,i)}$ & 0 & 0\\ \hline
$\varepsilon'_0:(0,0)\rightarrow (1,2)$ & $vuvu...\varepsilon_{\nu (0,0)}$ & 0 & 1\\
\hline $\varepsilon'_1:(0,1)\rightarrow (1,2)$ &
$uvuv...\varepsilon_{\nu (0,1)})$ & 0 & 0\\ \hline
$\varepsilon_0:(0,2)\rightarrow (1,0)$ & $\varepsilon'_0vuv...$ & 0
& 0\\ \hline $\varepsilon_1:(0,2)\rightarrow (0,1)$ &
$\varepsilon'_1uvu...$ & 1 & 0\\ \hline
\end{tabular}

\end{center}

and assertion 2.a follows.

\item If $\varphi =\rho\tau^m$,  we will take $I'=\{(k,i):$
$i\in\Delta_0\text{ and }0\leq k<m\}$ and we will put
$w_{(k,i)}=\tau^{-k}(w_{(0,i)})$, for each $(k,i)\in I'$, where $w_{(0,i)}$
is defined as in the previous case.

For the arrows $a$ starting
and ending at a vertex of $I'$, we take $p_a$ as in the table above
and $u(a)$ and $v(a)$ take the same values as in that table. In the
corresponding  table for this case, it is enough to give only the
data for the arrows which start at a vertex of $I'$ but end at one
not in $I'$:

\begin{center}

\begin{tabular}{|c|c|c|c|} \hline
 $a$ & $p_a$  & $u(a)$ & $v(a)$\\  \hline $\beta
:(m-1,i)\rightarrow (m,i-1)$ &
$\gamma_{(m,i-1)}uvuv...\delta_{(m+n-2,i)}$ & 0 & 1\\ \hline
$\varepsilon'_0:(m-1,0)\rightarrow (m,2)$ & $vuvu...\varepsilon_{\nu(m-1,0)}$ & 0 & 0\\
\hline $\varepsilon'_1:(m-1,1)\rightarrow (m,2)$ &
$uvuv...\varepsilon_{\nu (m-1,1)})$ & 0 & 1\\ \hline
\end{tabular}

\end{center}

These values come from the fact that
$w_{(m,i)}=\rho\tau^{-m}(w_{(0,i)})=\gamma_{(m,i)}vuvu...\delta_{(m+n-1,i)}$,
for each $i=2,...,n$. It is now clear that assertions 2.b.i and
2.b.ii hold. As for 2.b.iii, put $I'(q)=\{(k,i):$ $qm\leq
k<(q+1)m\text{ and }i\in\Delta_0\}$, i.e., the set of vertices
$(k,i)$ such that the quotient of dividing $k$ by $m$ is $q$. If
$\varepsilon_0:(k,2)\longrightarrow (k,0)$ has origin (and end) in
$I'(q)$, then
$\rho\tau^{-m}(\varepsilon_0)=\varepsilon_1:(k+m,1)\rightarrow
(k+m,2)$. The symmetric equality is true when exchanging the roles
of $0$ and $1$. It follows that $\eta (\varepsilon_0)=\nu
(\varepsilon_0)$ (resp. $\eta (\varepsilon_1)=-\nu (\varepsilon_0)$)
when the origin of $\varepsilon_0$ (resp. $\varepsilon_1$) is in
$I'(q)$, with $q$ even, and $\eta (\varepsilon_0)=-\nu (\varepsilon_0)$
(resp. $\eta (\varepsilon_1)=\nu (\varepsilon_1)$) otherwise. That is,
we have $\eta (\varepsilon_i)=(-1)^{q+i}\nu (\varepsilon_i)$.

A similar argument shows that if $k\not\equiv\text{ -1 (mod
}m\text{)}$ and $\varepsilon'_j:(k,j)\rightarrow (k+1,2)$, then we
have $\eta (\varepsilon'_j)=(-1)^{q+j+1}\nu (\varepsilon'_j)$.
Finally, if $\varepsilon'_j:((q+1)m-1,j)\rightarrow ((q+1)m,2)$ we get
that $\eta (\varepsilon'_j)=(-1)^{q+j}\nu (\varepsilon'_j)$, which
shows that the equalities in 2.b.iii also hold.

\end{enumerate}

\item When $\Delta= \mathbb{D}_4$:

\begin{enumerate}

\item If $\varphi= \tau^m$, the formulas for $\Delta= \mathbb{D}_{n+1}$ with $n>3$ and $\varphi= \tau^m$ are still valid in this case.

\item If $\varphi= \rho\tau^m$

We slightly divert from the previous case (see Convention \ref{convention}). We take
$w_{(0,0)}=\varepsilon'_0\varepsilon_1\varepsilon'_1\varepsilon_0$ and
$w_{(0,2)}=\varepsilon_0\varepsilon'_0\varepsilon_1\varepsilon'_1$.
Due to the fact that all nonzero paths from $(0,2)$ to $\nu
(0,2)=(2,2)$ are equal, up to sign, in $B$ we know that the action of
 $\langle\rho\rangle$ on those paths  is trivial. The base
$\mathcal{B}$ will be the union of the orbits of $w_{(0,0)}$ and
$w_{(0,2)}$ under the action of the group of automorphisms generated
by $\rho$ and $\tau$.

Recall that, in this case, the mesh arrows are the original ones:
$r_{(k,i)}=\sum_{t(a)=(k,i)}\sigma (a)a$. Note that if
$\varepsilon_i:(k,2)\rightarrow (k,i)$ and
$\varepsilon'_i:(k,i)\rightarrow (k+1,2)$, for $i=0,1,3$,  then we
have $w_{(k,i)}=\varepsilon'_i\varepsilon_{\rho
(i)}\varepsilon'_{\rho (i)}\varepsilon_i$ and
$w_{(k,2)}=\varepsilon_i\varepsilon'_i\varepsilon_{\rho
(i)}\varepsilon'_{\rho (i)}=-\varepsilon_{\rho
(i)}\varepsilon'_{\rho (i)}\varepsilon_i\varepsilon'_i$, for all
$i=0,1,3$. The corresponding table is then given as

\begin{center}

\begin{tabular}{|c|c|c|c|} \hline
 $a$ & $p_a$  & $u(a)$ & $v(a)$\\  \hline $\varepsilon'_i$ & $\varepsilon_{\rho
(i)}\varepsilon'_{\rho (i)}\varepsilon_i$ & 0 & 1\\
\hline $\varepsilon_i$ & $\varepsilon'_i\varepsilon_{\rho
(i)}\varepsilon'_{\rho (i)}$ & 0 & 0\\
 \hline
\end{tabular}
\end{center}

\vspace*{0.3cm}

\end{enumerate}

\item When $\Delta= \mathbb{E}_n$ with $n=6,7,8$:

For the sake of simplicity, we will write any path as a composition
of arrows in $\{\alpha, \alpha^{'}, \beta, \beta^{'}, \gamma,
\gamma^{'},$ $\delta, \delta^{'}, \zeta, \zeta', \theta, \theta' \varepsilon, \varepsilon^{'}\}$
whenever they exist and
assuming that each arrow is considered in the appropriate slice so
that the composition makes sense.

Also, we denote by $u$, $w$ and $v$, respectively, each of the paths
of length 2

\begin{center}

$(k,3)\longrightarrow (k,0)\longrightarrow (k+1,3)$

$(k,3)\longrightarrow (k+1,2)\longrightarrow (k+1,3)$

$(k,3)\longrightarrow (k,4)\longrightarrow (k+1,3)$

\end{center}

\noindent with no mention to k. Then  any path $(k,3)\rightarrow ...\rightarrow (k+r,3)$ is equal in
$B$ to a  monomial in $u,v,w$, with the obvious sense of 'monomial'.
Also, $\beta^{'}\beta= w$, $\gamma
\gamma^{'}=v$ and $\varepsilon \varepsilon^{'}=u$. It is important
to keep in mind that $u=v+w$ .

\begin{enumerate}[A.]

\item When $\Delta= \mathbb{E}_6$:

\begin{enumerate}[(a)]

\item If $\varphi= \tau$, we will take the canonical slice
$I':=\{(0,i):$ $i\in\Delta_0\}$.

With the abuse of notation of omi\-tting  $k$ when showing a vertex
$(k,i)$ in the diagrams below, we then take:

\begin{enumerate} [i.]

\item  $w_{(0,0)}$ is  the path

        $$\xymatrix{ 0 \ar[r] & 3 \ar[rrr]^{vwvw}& & & 3 \ar[r] & 0}$$

\item  $w_{(0,1)}$ is the path

        $$
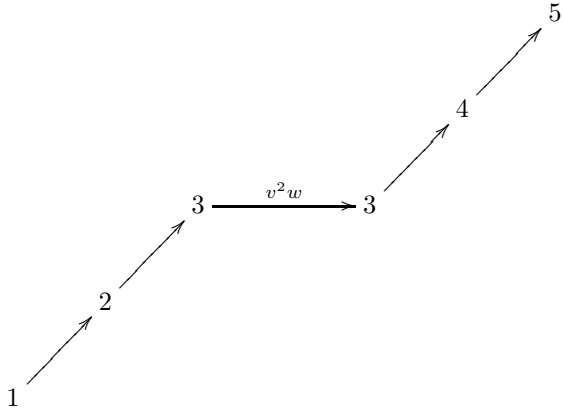
\xymatrix{& & & & & & 5\\
        & & & & & 4\ar[ru] & \\
        & & 3 \ar[rr]^{v^2w} & & 3 \ar[ru] & & \\
        & 2 \ar[ru] & & & & & \\
        1 \ar[ru] & & & & & &}$$

\item  $w_{(0,2)}$ is the path

        $$
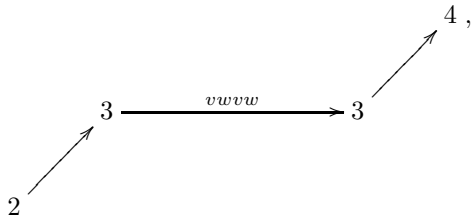
\xymatrix{& & & & & 4\\
        & 3 \ar[rrr]^{vwvw}& & & 3 \ar[ru] & \\
        2 \ar[ru] & & & & &},$$

\item  $w_{(0,3)}$ is the path

        $$\xymatrix{3 \ar[rrrr]^{vwvwv} & & & & 3}$$

\item  $w_{(0,4)}$ is the path

        $$
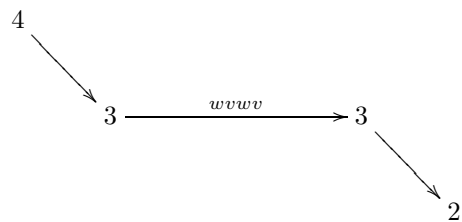
\xymatrix{4 \ar[rd]& & & & & \\
        & 3 \ar[rrr]^{wvwv}& & & 3 \ar[rd] & \\
        & & & & & 2}$$

and

\item  $w_{(0,5)}$ is the path

        $$
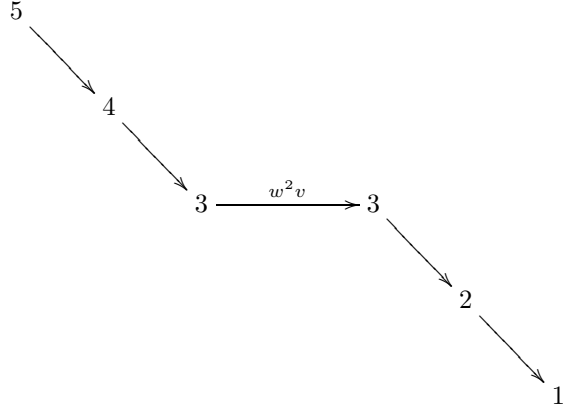
\xymatrix{5 \ar[rd]& & & & & & \\
        & 4 \ar[rd] & & & & & \\
        & & 3 \ar[rr]^{w^2v} & & 3 \ar[rd] & & \\
        & & & & & 2 \ar[rd] & \\
        & & & & & & 1 },$$

\end{enumerate}

Using the mesh relations, one
gets, among others, the equalities $u^2=w^3=v^3=0$, $vwv=wvw$,
$vw^2v=-vwv^2-v^2wv$ and $vwvwv=-wvwvw$.

Then, the table is the following:

\begin{center}

\begin{tabular}{|c|c|c|c|} \hline
 $a$ & $p_a$  & $u(a)$ & $v(a)$\\  \hline
$\alpha :(0,1)\rightarrow (0,2)$ & $\beta v^2w\gamma\delta$ & 0 & 0\\
\hline  $\beta :(0,2)\rightarrow (0,3)$ & $vwvw\gamma$ & 0 & 0\\
\hline  $\gamma :(0,3)\rightarrow (0,4)$ & $\gamma 'wvwv$ & 0 & 0\\
\hline $\delta :(0,4)\rightarrow (0,5)$ & $\delta'\gamma'w^2v\beta
'$ & 1 & 0\\ \hline $\varepsilon :(0,3)\rightarrow (0,0)$ &
$\varepsilon'vwvw$ & 1 & 0\\ \hline $\alpha ':(0,2)\rightarrow
(1,1)$ & $\alpha\beta v^2w\gamma$ & 1 & 0\\ \hline $\beta
':(0,3)\rightarrow (1,2)$ & $\beta vwvw$ & 1 & 0\\ \hline $\gamma
':(0,4)\rightarrow (1,3)$ & $wvwv\beta '$ & 0 & 1\\ \hline $\delta
':(0,5)\rightarrow (1,4)$ & $\gamma 'w^2v\beta '\alpha '$ & 0 & 0\\
\hline $\varepsilon ':(0,0)\rightarrow (1,3)$ & $vwvw\varepsilon$ &
0 & 0\\ \hline
 \end{tabular}

\end{center}

From this table the equalities in 3.a follow.

\item If $\varphi= \rho\tau^m$, we
will consider the slice $T=\{(0,i):$
$i=0,3,4,5\}\cup\{(1,2),(2,1)\}$, which is $\rho$-invariant, and
then take $I'=\{\tau^{-k}(r,i):$ $(r,i)\in T\text{ and }0\leq
k<m\}$.

The paths $w_{(0,i)}$ ($i=0,3,4,5$) will be as in the case
$\varphi =\tau$,  and we will define below the paths $w_{(1,2)}$ and
$w_{(2,1)}$. Then we will take
$w_{\tau^{-k}(r,j)}=\tau^{-k}(w_{(r,j)})$, for all $(r,j)\in T$ and
$0\leq k<m$.

\begin{enumerate}[i.]

 \item $w_{(1,2)}$ is the path

        $$
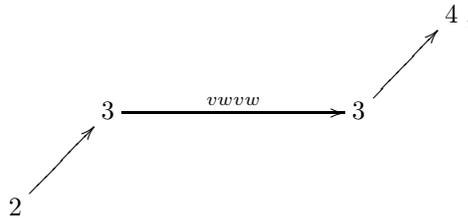
\xymatrix{& & & & & 4\\
        & 3 \ar[rrr]^{vwvw}& & & 3 \ar[ru] & \\
        2 \ar[ru] & & & & &},$$

  and

  \item $w_{(2,1)}$ is the path

  $$
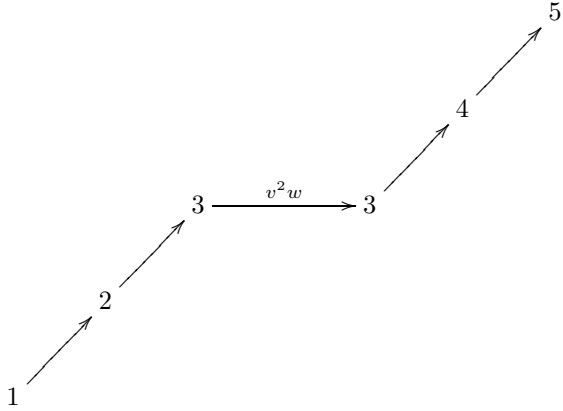
\xymatrix{& & & & & & 5\\
        & & & & & 4\ar[ru] & \\
        & & 3 \ar[rr]^{v^2w} & & 3 \ar[ru] & & \\
        & 2 \ar[ru] & & & & & \\
        1 \ar[ru] & & & & & &}$$

\end{enumerate}

Arguing as in the case of $\mathbb{D}_{n+1}$, we see that the values $u(a)$
and $v(a)$ are the ones in the last table, when $i(a),t(a)\in I'$.
We then need only to give those values for the arrows $a$ with
origin in $I'$ and terminus not in $I'$. We have the table:

\begin{center}

\begin{tabular}{|c|c|c|c|} \hline
 $a$ & $p_a$  & $u(a)$ & $v(a)$\\  \hline
$\alpha :(m+1,1)\rightarrow (m+1,2)$ & $\beta v^2w\gamma\delta$ & 0 & 0\\
\hline  $\beta :(m,2)\rightarrow (m,3)$ & $vwvw\gamma$ & 0 & 1\\
\hline  $\gamma ':(m-1,4)\rightarrow (m,3)$ & $wvwv\beta '$ & 0 & 0\\
\hline $\delta
':(m-1,5)\rightarrow (m,4)$ & $\gamma 'w^2v\beta '\alpha '$ & 0 & 0\\
\hline $\varepsilon ':(m-1,0)\rightarrow (m,3)$ & $vwvw\varepsilon$
& 0 & 1\\ \hline
 \end{tabular}

\end{center}

We have used in the construction of this table the fact that
$w_{(k,2)}=\beta vwvw\gamma$ and $w_{(k,4)}=\gamma 'wvwv\beta '$,
for all $k\in\mathbb{Z}$, while $w_{(2r,3)}=vwvwv$ and
$w_{(2r+1,3)}=wvwvw$.

Note that, with the labeling of vertices that we are using, we have
that $\rho (k,i)=(k+i-3,6-i)$ for each $i\neq 0$ and $\rho(k,0)=(k,0)$. For each $q\in\mathbb{Z}$, we put
$I'(q):=(\rho\tau^{-m})^q(I')$. When passing from a piece $I'(q)$ to
$I'(q+1)$ by applying $\rho\tau^{-m}$, an arrow $\alpha$ is
transformed in an arrow $\delta '$ and an arrow $\delta '$ in an
arrow $\alpha$. From the last two tables we then get that $\eta
(\alpha )=\nu (\alpha )$ and $\eta (\delta ')=\nu (\delta ')$, for
all arrows of the type $\alpha$ or $\delta '$ in $\mathbb{Z}\Delta$.

The argument of the previous paragraph can be applied to the pair of
arrows $(\gamma ,\beta ')$ instead of $(\alpha ,\delta ')$ and we
get from the last two tables that $\eta (\gamma )=\nu (\gamma )$
(resp. $\eta (\beta ')=-\nu (\beta ')$) when $\gamma$ (resp. $\beta
'$) has its origin in $I'(q)$, with $q$ even, and $\eta (\gamma
)=-\nu (\gamma )$ (resp. $\eta (\beta ')=\nu (\beta ')$) otherwise.
From this the formulas in 3.b concerning $\gamma$ and $\beta '$ are
clear.

Next we apply the argument to the pair of arrows $(\delta ,\alpha
')$ and get that $\eta (\delta )=-\nu (\delta )$ (resp. $\eta
(\alpha ' )=-\nu (\alpha ' )$), for all arrows of type $\delta$ or
$\alpha '$ in $\mathbb{Z}\Delta$.

An arrow of type $\varepsilon$ (resp. $\varepsilon '$) is
transformed on one of the same type when applying $\rho\tau^{-m}$.
It then follows that $\eta (\varepsilon )=-\nu (\varepsilon )$, for
any arrow of type $\varepsilon$. It also follows that $\eta
(\varepsilon ')=-\nu (\varepsilon ')$, when the origin of
$\varepsilon '$ is $(k,0)$ with $k\equiv -1\text{ (mod m)}$, and
$\eta (\varepsilon ')=\nu (\varepsilon ')$ otherwise.

We finally apply the argument to the pair of arrows $(\beta ,\gamma
')$. If we look at the  two pieces $I'(0)$ and $I'(1)$, then from
the last two tables we see that if $\beta :(k,2)\rightarrow (k,3)$,
with $(k,3)\in I'(0)\cup I'(1)$, then $\eta (\beta )=\nu (\beta )$,
when $k\in\{1,2,...,m-1,2m\}$, and $\eta (\beta )=-\nu (\beta )$,
when $k\in\{m,m+1,...,2m-1\}$. We then get that $\eta (\beta
)=(-1)^q\nu (\beta )$, where $q$ is the quotient of dividing $k$ by
$m$. By doing the same with $\gamma ':(k,4)\rightarrow (k+1,3)$, we
see that $\eta (\gamma ')=-\nu (\gamma ')$, when
$k\in\{0,1,...,m-2,2m-1\}$, and $\eta (\gamma ')=\nu (\gamma ')$,
when $k\in\{m-1,m,...,2m-2\}$. If now $k\in\mathbb{Z}$ is arbitrary,
then we obtain that $\eta (\gamma ')=\nu (\gamma ')$ if, and only if,
$k\not\in\bigcup_{t\in\mathbb{Z}}(2tm-2,(2t+1)m-1)$. Equivalently,
when $q$ is odd and $r\neq m-1$ or $q$ is even and $r=m-1$.

\end{enumerate}

\item When $\Delta= \mathbb{E}_7$:

We consider the canonical slice $I'= \{(0,i): i\in \Delta_0\}$.
Using the same notation as when $\Delta= \mathbb{E}_6$, we have in this case, among others, the equalities $u^2=w^3=v^4=0$, $vwv=wvw-v^3$,
and $vwvwv=-wvwvw$. Then, we consider

\begin{enumerate}[i.]

\item  $w_{(0,0)}$ is  the path

        $$\xymatrix{ 0 \ar[r] & 3 \ar[rrr]^{vwvwvwv}& & & 3 \ar[r] & 0}$$

\item  $w_{(0,1)}$ is the path

        $$
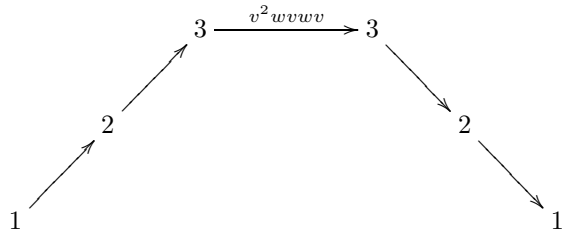
\xymatrix{& & 3 \ar[rr]^{v^2wvwv} & & 3 \ar[rd] & & \\
        & 2 \ar[ru] & & & &2 \ar[rd] & \\
        1 \ar[ru] & & & & & &1}$$

\item  $w_{(0,2)}$ is the path

        $$
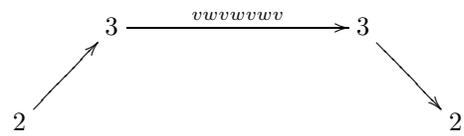
\xymatrix{& 3 \ar[rrr]^{vwvwvwv}& & & 3 \ar[rd] & \\
        2 \ar[ru] & & & & & 2}$$

\item  $w_{(0,3)}$ is the path

        $$\xymatrix{3 \ar[rrrr]^{vwvwvwvw} & & & & 3}$$

\item  $w_{(0,4)}$ is the path

        $$
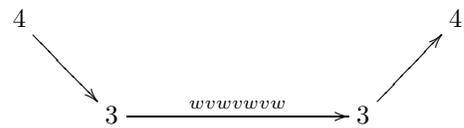
\xymatrix{4 \ar[rd]& & & & & 4\\
        & 3 \ar[rrr]^{wvwvwvw}& & & 3 \ar[ru] &}$$

\item  $w_{(0,5)}$ is the path

        $$
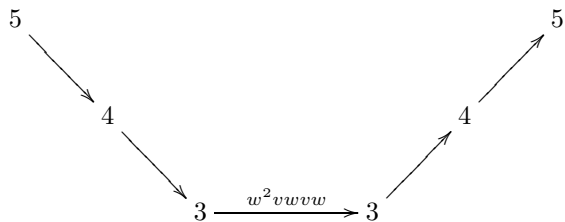
\xymatrix{5 \ar[rd]& & & & & & 5\\
        & 4 \ar[rd] & & & &4 \ar[ru] & \\
        & & 3 \ar[rr]^{w^2vwvw} & & 3 \ar[ru] & &}$$

\item  $w_{(0,6)}$ is the path

        $$
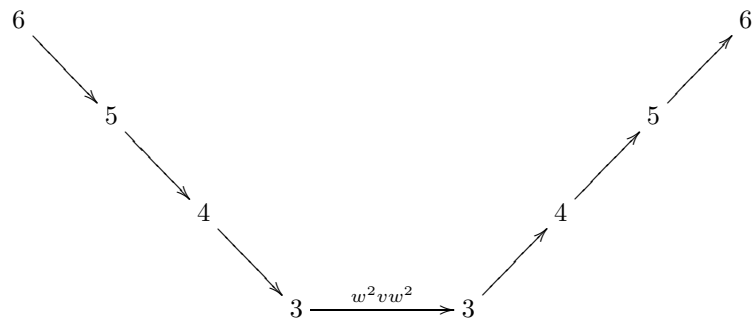
\xymatrix{6 \ar[rd] & & & & & & & & 6\\
        & 5 \ar[rd]& & & & & & 5 \ar[ru] &\\
        & & 4 \ar[rd] & & & &4 \ar[ru] & &\\
        & & & 3 \ar[rr]^{w^2vw^2} & & 3 \ar[ru] & & &}$$

\end{enumerate}

Hence, we get the following table:

\begin{center}

\begin{tabular}{|c|c|c|c|} \hline
 $a$ & $p_a$  & $u(a)$ & $v(a)$\\  \hline
$\alpha :(0,1)\rightarrow (0,2)$ & $\beta v^2wvwv\beta'\alpha'$ & 0 & 0\\
\hline  $\beta :(0,2)\rightarrow (0,3)$ & $vwvwvwv\beta'$ & 0 & 0\\
\hline  $\gamma :(0,3)\rightarrow (0,4)$ & $\gamma 'wvwvwvw$ & 0 & 0\\
\hline $\delta :(0,4)\rightarrow (0,5)$ & $\delta'\gamma'wv^2wvw\gamma
$ & 0 & 1\\
\hline $\zeta :(0,5)\rightarrow (0,6)$ & $\zeta'\delta'\gamma'wv^3w\gamma
\delta$ & 0 & 0\\ \hline $\varepsilon :(0,3)\rightarrow (0,0)$ &
$\varepsilon'wvwvwvw$ & 0 & 1\\ \hline $\alpha ':(0,2)\rightarrow
(1,1)$ & $\alpha\beta vwvwv^2 \beta'$ & 0 & 1\\ \hline $\beta
':(0,3)\rightarrow (1,2)$ & $\beta vwvwvwv$ & 1 & 0\\ \hline $\gamma
':(0,4)\rightarrow (1,3)$ & $wvwvwvw\gamma$ & 0 & 1\\ \hline $\delta
':(0,5)\rightarrow (1,4)$ & $\gamma 'w^2vwvw\gamma\delta$ & 0 & 0\\
\hline $\zeta':(0,6)\rightarrow (1,5)$ & $\delta'\gamma'w^2vw^2\gamma\delta\zeta$ & 0 & 1\\
\hline $\varepsilon ':(0,0)\rightarrow (1,3)$ & $vwvwvwv\varepsilon$ &
0 & 0\\ \hline
 \end{tabular}

\end{center}

From this table the equalities in 4 follow.

\item When $\Delta= \mathbb{E}_8$:

As in the previous case, we consider the canonical slice and follow the same notation. In addition, note that we have the equalities $u^2=w^3=v^5=0$, $vwv=wvw-v^3$, $(vw)^3=(wv)^3+vwv^4-v^4wv$, $(vw)^6= (wv)^6+(wv)^3vwv^4-v^4wv^2wv^4$
and $(vw)^7=-(wv)^7$. Then, considering the following paths

\begin{enumerate}[i.]

\item  $w_{(0,0)}$ is  the path

        $$\xymatrix{ 0 \ar[r] & 3 \ar[rrr]^{(vw)^6v}& & & 3 \ar[r] & 0}$$

\item  $w_{(0,1)}$ is the path

        $$
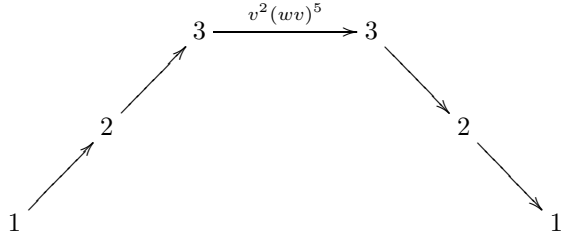
\xymatrix{& & 3 \ar[rr]^{v^2(wv)^5} & & 3 \ar[rd] & & \\
        & 2 \ar[ru] & & & &2 \ar[rd] & \\
        1 \ar[ru] & & & & & &1}$$

\item  $w_{(0,2)}$ is the path

        $$
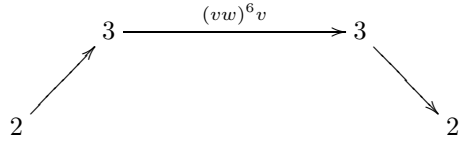
\xymatrix{& 3 \ar[rrr]^{(vw)^6v}& & & 3 \ar[rd] & \\
        2 \ar[ru] & & & & & 2}$$

\item  $w_{(0,3)}$ is the path

        $$\xymatrix{3 \ar[rrrr]^{(vw)^7} & & & & 3}$$

\item  $w_{(0,4)}$ is the path

        $$
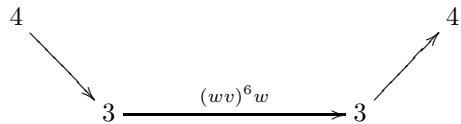
\xymatrix{4 \ar[rd]& & & & & 4\\
        & 3 \ar[rrr]^{(wv)^6w}& & & 3 \ar[ru] &}$$

\item  $w_{(0,5)}$ is the path

        $$
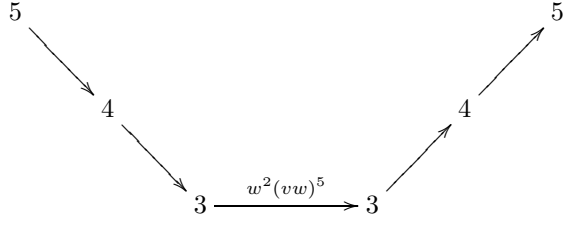
\xymatrix{5 \ar[rd]& & & & & & 5\\
        & 4 \ar[rd] & & & &4 \ar[ru] & \\
        & & 3 \ar[rr]^{w^2(vw)^5} & & 3 \ar[ru] & &}$$

\item  $w_{(0,6)}$ is the path

        $$
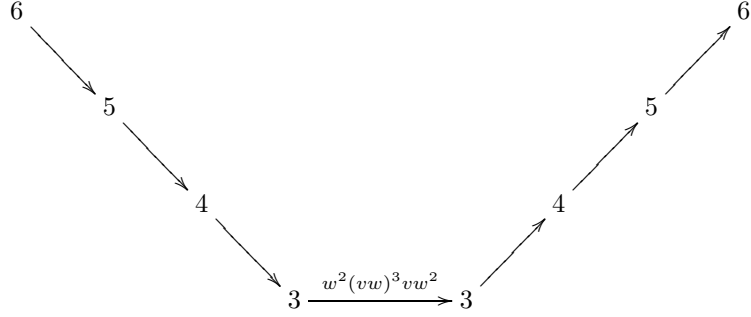
\xymatrix{6 \ar[rd] & & & & & & & & 6\\
        & 5 \ar[rd]& & & & & & 5 \ar[ru] &\\
        & & 4 \ar[rd] & & & &4 \ar[ru] & &\\
        & & & 3 \ar[rr]^{w^2(vw)^3vw^2} & & 3 \ar[ru] & & &}$$

\item  $w_{(0,7)}$ is the path

        $$
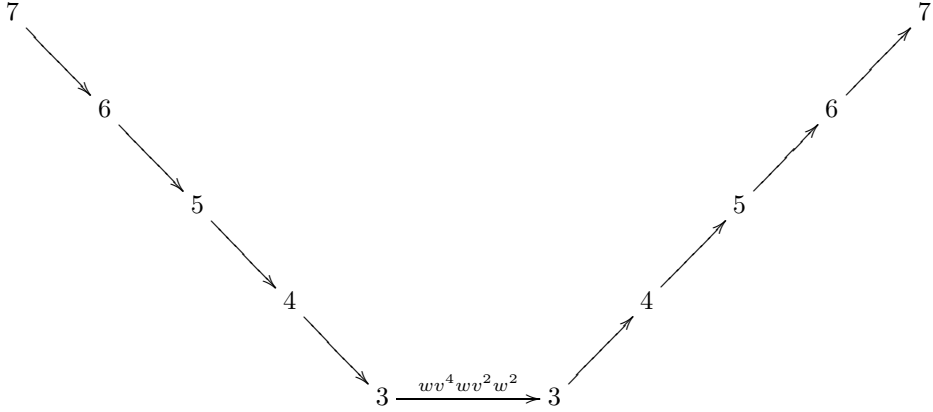
\xymatrix{7 \ar[rd] & & & & & & & & & & 7\\
        & 6 \ar[rd] & & & & & & & & 6 \ar[ru] &\\
        & & 5 \ar[rd]& & & & & & 5 \ar[ru] & &\\
        & & & 4 \ar[rd] & & & &4 \ar[ru] & & &\\
        & & & & 3 \ar[rr]^{wv^4wv^2w^2} & & 3 \ar[ru] & & & &}$$

\end{enumerate}

we obtain the table below:

\begin{center}

\begin{tabular}{|c|c|c|c|} \hline
 $a$ & $p_a$  & $u(a)$ & $v(a)$\\
\hline$\alpha :(0,1)\rightarrow (0,2)$ & $\beta v^2(wv)^5\beta'\alpha'$ & 0 & 0\\
\hline  $\beta :(0,2)\rightarrow (0,3)$ & $(vw)^6v\beta'$ & 0 & 0\\
\hline  $\gamma :(0,3)\rightarrow (0,4)$ & $\gamma 'w(vw)^6$ & 0 & 0\\
\hline $\delta :(0,4)\rightarrow (0,5)$ & $\delta'\gamma'wv^2(wv)^4w\gamma$ & 0 & 1\\
\hline $\zeta :(0,5)\rightarrow (0,6)$ & $\zeta'\delta'\gamma'wv^3(wv)^3w\gamma\delta$ & 0 & 0\\
\hline $\theta :(0,6)\rightarrow (0,7)$ & $\theta'\zeta'\delta'\gamma'wv^4wv^2w^2\gamma\delta\zeta$ & 0 & 0\\
\hline $\varepsilon :(0,3)\rightarrow (0,0)$ &$\varepsilon'w(vw)^6$ & 0 & 1\\
\hline $\alpha ':(0,2)\rightarrow(1,1)$ & $\alpha\beta(vw)^5v^2\beta'$ & 0 & 1\\
\hline $\beta':(0,3)\rightarrow (1,2)$ & $\beta(vw)^6v$ & 1 & 0\\
\hline $\gamma':(0,4)\rightarrow (1,3)$ & $(wv)^6w\gamma$ & 0 & 1\\
\hline $\delta':(0,5)\rightarrow (1,4)$ & $\gamma 'w^2(vw)^5\gamma\delta$ & 0 & 0\\
\hline $\zeta':(0,6)\rightarrow (1,5)$ & $\delta'\gamma'w^2(vw)^4w\gamma\delta\zeta$ & 0 & 1\\
\hline $\theta':(0,7)\rightarrow (1,6)$ & $\zeta'\delta'\gamma 'wv^4wv^2w^2\gamma\delta\zeta\theta$ & 0 & 1\\
\hline $\varepsilon ':(0,0)\rightarrow (1,3)$ & $(vw)^6v\varepsilon$ &
0 & 0\\ \hline
 \end{tabular}

\end{center}

From this table the equalities in 5 follow.

\end{enumerate}

\end{enumerate}

\section{Acknowledgements}

The authors are supported by research projects of
the Spanish Ministry of Education and Science and the Fundaci\'on S\'eneca
of Murcia, with a part of FEDER funds. They thank both institutions for
their support.

The authors also thank Hideto Asashiba, Alex Dugas,
Karin Erdmann
 and Henning Krause for answering some questions that
were very helpful in the preparation of the paper.

Finally, the authors are specially thankful to both anonymous referees for the careful reading of the manuscript  and for many comments and suggestions that helped to improve a lot the earlier version of the paper.

\end{document}